\documentclass[11pt]{article}%
\usepackage{amsmath}
\usepackage{cases}
\usepackage[square, numbers]{natbib}
\usepackage{amsthm,amssymb}
\usepackage[mathscr]{eucal}
\usepackage{graphicx}
\usepackage{latexsym}
\usepackage{amssymb}
\usepackage{psfrag}
\usepackage{epsfig}
\usepackage{enumerate}
\usepackage{comment}
\usepackage{amsfonts}
\usepackage{graphicx}
\usepackage{color}
\usepackage{mathtools}%
\providecommand{\U}[1]{\protect\rule{.1in}{.1in}}
\DeclareMathAlphabet{\mathpzc}{OT1}{pzc}{m}{it}
\include{psbox}
\newtheorem{theorem}{\bf Theorem}[section]

\newtheorem{lemma}[theorem]{\bf Lemma}
\newtheorem{assumption}[theorem]{Assumption}

\newtheorem{proposition}[theorem]{\bf Proposition}

\theoremstyle{remark}

\newtheorem{remark}[theorem]{\bf Remark}

\numberwithin{equation}{section} \numberwithin{theorem}{section}
\begin{document}

\title{Large Deviations for Small Noise Diffusions in a Fast Markovian Environment}
\author{Amarjit Budhiraja\thanks{Research supported in part by the National Science
Foundation (DMS-1305120), the Army Research Office (W911NF-14-1-0331) and
DARPA (W911NF-15-2-0122).}, Paul Dupuis\thanks{Research supported in part by
the National Science Foundation (DMS-1317199), the Army Research Office
(W911NF-12-1-0222), and the Defense Advanced Research Projects Agency
(W911NF-15-2-0122).}, Arnab Ganguly\thanks{Research supported in part by
Louisiana Board of Regents through the Board of Regents Support Fund (contract
number: LEQSF(2016-19)-RD-A-04).}}
\maketitle

\begin{abstract}
\noindent A large deviation principle is established for a two-scale
stochastic system in which the slow component is a continuous process given by
a small noise finite dimensional It\^{o} stochastic differential equation, and
the fast component is a finite state pure jump process. Previous works have
considered settings where the coupling between the components is weak in a
certain sense. In the current work we study a fully coupled system in which
the drift and diffusion coefficient of the slow component and the jump
intensity function and jump distribution of the fast process depend on the
states of both components. In addition, the diffusion can be degenerate. Our
proofs use certain stochastic control representations for expectations of
exponential functionals of finite dimensional Brownian motions and Poisson
random measures together with weak convergence arguments. A key challenge is
in the proof of the large deviation lower bound where, due to the interplay
between the degeneracy of the diffusion and the full dependence of the
coefficients on the two components, the associated local rate function has
poor regularity properties. \newline\ \newline

\noindent

\noindent\textbf{AMS 2010 subject classifications:} 60F10, 60J75, 60G35,
60K37.\newline\ \newline

\noindent\textbf{Keywords:} Large deviations, variational representations,
stochastic averaging, averaging principle, small noise asymptotics,
multi-scale analysis, switching diffusions, Markov modulated diffusions,
Poisson random measures.

\end{abstract}



\section{Introduction}

\label{intro} We study a stochastic system with two time scales where the slow
scale evolution is described through a continuous stochastic process, given by
a small noise finite dimensional It\^{o} stochastic differential equation, and
the fast component is given as a rapidly oscillating pure jump process. The
two processes are fully coupled in that the drift and diffusion coefficient of
the slow process and the jump intensity function and jump distribution of the
fast process depend on the states of both components. Multiscale systems of
the form considered in this work arise in many problems from systems biology,
financial engineering, queuing systems, etc. For example, most cellular
processes are inherently multiscale in nature with reactions occurring at
varying speeds. This is especially true in many genetic networks, where
protein concentration, usually modeled by a small-noise diffusion process, is
controlled by different genes rapidly switching between their respective
active and inactive states \cite{cdr:09}. The key characterizing feature of
such slow-fast systems is that the fast component reaches its equilibrium
state at much shorter time scales at which the slow system effectively remains
unchanged. This local equilibration phenomenon allows the approximation of the
properties of the slow system by averaging out the coefficients over the local
stationary distributions of the fast component. Such approximations yield a
significant model simplification and are mathematically justified by
establishing an appropriate \emph{averaging principle}.

The averaging principle, which has its roots in the works of Laplace and
Lagrange, has a long history of applications in celestial mechanics,
oscillation theory, radiophysics, etc. For deterministic systems, the first
rigorous results were obtained by Bogoliubov and Mitropolsky \cite{BM61}, and
further developments and generalizations were subsequently carried out by
Volosov, Anosov, Neishtadt, Arnold and others (for example, see
\cite{ArKoNe06, Neis90}). The stochastic version of the theory originated with
the seminal paper of Khasminskii \cite{Khas68} and later advanced in the works
of Freidlin, Lipster, Skorohod, Veretennikov, Wentzel and others (for example,
see \cite{FW98, Skh76, Ver90}). Stochastic averaging principles for various
models arising from systems biology have been studied in \cite{BKPR06, KK13,
KK14}. As noted above, an averaging principle provides a model simplification
in an appropriate scaling regime. In order to capture the approximation errors
due to the use of such simplified models one needs a more precise asymptotic
analysis. The goal of the current work is to study one such asymptotic result
that gives a \emph{large deviation principle }(LDP) for the slow process as
the parameter governing the magnitude of the small noise in the diffusion
component and the speed of the fast component approaches its limit. Such a
result, in addition to providing estimates on the rate of convergence of the
trajectories of the slow component to that of the averaged system, is a
starting point for developing accelerated Monte-Carlo schemes for the
estimation of probabilities of rare events (cf. \cite{DupSpiWan}).

For a two-scaled system where both components are continuous processes given
through finite dimensional It\^{o} stochastic differential equations, the
problem has been studied in \cite{Lip96, FW98, Ver99, Ver2000}. In all these
works the coupling between the two components is weak in a certain sense. By
this we mean that either the slow component has no diffusion term \cite{FW98,
Ver99}, or the dynamics of the fast component does not depend on the slow one
\cite{Lip96}, or at least the diffusion coefficient of the fast component does
not depend on the slow term \cite{Ver2000}. A recent paper by Puhalskii
\cite{Puh16} studies a large deviation principle for a fully coupled two-scale
diffusion system. Under various conditions on the coefficients of the two
diffusions including in particular certain non-degeneracy conditions on the
diffusion coefficients, the paper uses the exponential tightness and limit
characterization approach of \cite{FK06} to establish a LDP for the slow component.

For settings where the fast component is a jump process, there are only a few
results. In \cite{HeYin14,HMS15} the authors study a large deviation principle
for a two-scale system in which the trajectories of the slow diffusion
component is modulated by a fast moving Markov chain (whose evolution does not
depend on the slow component). An earlier paper, \cite{HYZ11}, considered a
simpler case with no diffusion term in the equation for the slow component.
This simpler case, under a somewhat more restrictive condition, was also
studied by Freidlin and Wentzell in \cite{FW98}. However, in all of these
works the dynamics of the Markov chain do not depend on that of the slow
diffusion component.
Large deviation problems for general two-scale jump diffusions have recently
been considered in \cite{KuPo15}. The authors prove a large deviation
principle for each fixed time $t>0$ using the nonlinear semigroup and
viscosity solution based approach developed in \cite{FK06}; however, a process
level large deviation result is not considered. One of the critical
assumptions in this work is the validity of a comparison principle for a
certain nonlinear Cauchy problem (see Theorem 3 therein). Verification of the
comparison principle is in general a challenging task which needs to be done
on a case by case basis for different systems. In particular the two examples
in \cite{KuPo15} (see Section 4 therein) where the comparison principle is
shown to hold, are of specific forms -- for the first example, the evolution
of the fast component does not depend on the slow component, whereas in the
second example the fast component is a two state Markov chain whose stationary
distribution does not depend on the state of the slow component. We also note
that \cite{KuPo15} makes the assumption that the jump coefficients are
Lipschitz continuous in an appropriate sense. Such a property fails to hold
for systems considered in the current paper; specifically, the integrand in
the second equation in \eqref{fastslow} is not Lipschitz continuous (in fact
not even continuous).

As noted previously, the current paper studies a setting where the two
components are fully coupled. Specifically, for fixed $\varepsilon>0$, we
consider a two component Markov process $(X^{\varepsilon},Y^{\varepsilon})$,
where $X^{\varepsilon}$ is a $d$-dimensional continuous stochastic process
given as the solution of a stochastic equation of the form
\[
dX^{\varepsilon}(t)=b(X^{\varepsilon}(t),Y^{\varepsilon}(t))dt+\sqrt
{\varepsilon}a(X^{\varepsilon}(t),Y^{\varepsilon}(t))dW(t),
\]
where $W$ is a $m$-dimensional Brownian motion, and $Y^{\varepsilon}$ is a
process with a finite state space described in terms of a jump intensity
function $c(\cdot,\cdot)$ and a probability transition kernel $r(\cdot
,\cdot,dy)$, both of which depend on the states of $X^{\varepsilon}$ and
$Y^{\varepsilon}$. We make standard Lipschitz assumptions on the coefficients
of the diffusion, however we do not impose any non-degeneracy restrictions on
the diffusion. In the setting we consider methods based on approximations,
exponential tightness estimates and Girsanov change of measure appear to be
quite hard to implement. One of the main challenges in the analysis is due to
the interplay between the possible degeneracy of the diffusion coefficient and
the dependence of the various coefficients ($b,a,c$ and $r$) on both
components. In our approach we bypass discretizations and approximations by
using certain variational representations of expectations of positive
functionals of Brownian motions and Poisson random measures together with weak
convergence techniques. The variational representations for these noise
processes that we use were developed in \cite{BoDu98, BDM11} and have been
previously used in proving large deviation principles for a variety of complex
systems (see \cite{BDV08, BCD13, BDG14} and references therein). Using these
representations, the proof of the upper bound reduces to proving the tightness
and characterizations of weak limit points of certain controlled versions of
the state process $X^{\varepsilon}$. We note that in the description of these
controlled systems there are two types of controls -- one that controls the
drift of the Brownian noise and the other that controls the intensity of the
underlying Poisson random measure through a random `thinning' function. The
presence of these two controls coupled with the strong dependence of the
coefficients on both the components make the required asymptotic analysis challenging.

The main challenge in this work arises in the proof of the lower bound. When
using the variational representations, the proof of the lower bound requires
the construction of controls which lead to a prescribed limit trajectory with
a prescribed cost. In particular, when multiple times scales are present, one
generally needs to establish the convergence of the empirical measure for the
fast variables to an a priori identified measure (which could depend on the
state of the slow variables). A natural technique is to first show that the
velocities of the trajectory can be made piecewise smooth (e.g., piecewise
constant), so that transition probabilities associated with the fast variables
can be treated as essentially constant over each interval where the velocity
is continuous. Unfortunately, this smoothing in time of the state requires
establishing regularity properties of the local rate function, which is the
function $L(x,\beta)$ when the rate function is written in the somewhat
standard form
\[
I(\xi)=\int_{0}^{T}L(\xi(t),\dot{\xi}(t))dt.
\]
It is the need for these regularity properties which leads to undesirable
assumptions that may not in fact be necessary (e.g., nondegeneracy of a
diffusion coefficient).

We will use a different method to establish convergence that does not rely on
any smoothing in the time variable, and which in particular will allow for
degenerate diffusion coefficients. This alternative approach instead slightly
perturbs the controls used on the noise space (both the control of the
Brownian term that directly impacts the slow variables and the control of the
Poisson term determining evolution of the fast variables), in such a way that
the resulting mapping from controls into the state trajectory is unique. This
uniqueness result is the key to the construction of near optimal controls for
the prelimit process for which the appropriate convergence properties can be
proved and from which the lower bound follows readily. The perturbation
argument and resulting uniqueness, which is given in Proposition
\ref{prop:prop4.1}, is described in detail at the beginning of Section
\ref{lowbd}. The strategy for the proof of Proposition \ref{prop:prop4.1} is
explained in Remark \ref{rem:stratofproof}.

The rest of the paper is organized as follows. In Section \ref{sec:ave} we
give a precise mathematical formulation of the model and the statement of our
main result. The large deviation upper bound is proved in Section \ref{uppbd}.
Section \ref{sec:uniqchae17} constructs suitable near optimal controls and
controlled trajectories with appropriate uniqueness properties. The large
deviation lower bound is proved in Section \ref{lowbd}.

\emph{Notation:} The following mathematical notation and conventions will be
used in the paper. For a Polish space $S$, we denote by $\mathcal{P}(S)$
(resp. $\mathcal{M}_{F}(S)$) the space of probability measures (resp. finite
measures) on $S$ equipped with the topology of weak convergence. We denote by
$C_{b}(S)$ the space of real continuous and bounded functions on $S$. The
space of continuous functions from $[0,T]$ to $S$, equipped with the uniform
topology, will be denoted as $C([0,T]: S)$. For a bounded $\mathbb{R}^{d}$
valued function $f$ on $S$, we define $\|f\|_{\infty} = \sup_{x \in S}
\|f(x)\|$. For a finite set $\mathbb{L}$, we denote by $\mathbb{M}%
(\mathbb{L})$ the space of real functions on $\mathbb{L}$. Cardinality of such
a set will be denoted as $|\mathbb{L}|$. Given a probability function $r:
\mathbb{L} \to[0,1]$ (i.e. $\sum_{x \in\mathbb{L}} r(x) =1$), we denote,
abusing notation, $\sum_{x\in A} r(x)$ by $r(A)$ for all $A \subset\mathbb{L}$
and $\sum_{x\in\mathbb{L}} f(x) r(x)$ by $\int_{\mathbb{L}}f(x) r(dx)$ for all
$f \in\mathbb{M}(\mathbb{L})$. Space of Borel measurable maps from $[0,T]$ to
a metric space $S$ will be denoted as $\mathbb{M}([0,T]:S)$. Infimum over an
empty set, by convention, is taken to be $\infty$.

\section{ Mathematical Preliminaries and Main Result}

\label{sec:ave}

For fixed $\varepsilon>0$, we consider a two component Markov process
$\{(X^{\varepsilon}(t),Y^{\varepsilon}(t))\}_{0\leq t\leq T}$ with values in
$\mathbb{G}=\mathbb{R}^{d}\times\mathbb{L}$, where $\mathbb{L}=\{1,\ldots
,|\mathbb{L}|\}$ is equipped with the usual operation of addition modulo
$|\mathbb{L}|$. A precise stochastic evolution equation for the pair
$(X^{\varepsilon},Y^{\varepsilon})$ will be given below in terms of a
$m$-dimensional Brownian motion and suitable Poisson random measure. However,
roughly speaking, the pair $(X^{\varepsilon},Y^{\varepsilon})$ describes a
jump-diffusion, where the diffusion component (namely $X^{\varepsilon}$) has
\textquotedblleft small noise\textquotedblright\ while the jump component
($Y^{\varepsilon}$) has jumps at rate $O(\varepsilon^{-1})$. The drift and
diffusion coefficients of the continuous component are given by suitable
functions $b:\mathbb{G}\rightarrow\mathbb{R}^{d}$ and $a:\mathbb{G}%
\rightarrow\mathbb{R}^{d\times m}$. The evolution of the pure-jump fast
component is described through a jump intensity function $c:\mathbb{G}%
\rightarrow\lbrack0,\infty)$ and a transition probability function
$r:\mathbb{G}\times\mathbb{L}\times\mathbb{L}\rightarrow\lbrack0,1]$. Our main
assumptions on these functions are as follows.

\begin{assumption}
\label{assum}$\,$

\begin{enumerate}
\item There exists $d_{{\tiny {\mbox{{\em lip}}}}} \in(0,\infty)$ such that
for all $y,y^{\prime}\in\mathbb{L}$ and $x,x^{\prime}\in\mathbb{R}^{d}$,
\[
|c(x,y) - c(x^{\prime},y)| + \|a(x,y) - a(x^{\prime},y)\| + \|b(x,y) -
b(x^{\prime},y)\|+ |r(x,y,y^{\prime}) - r(x^{\prime},y,y^{\prime})| \le
d_{{\tiny {\mbox{{\em lip}}}}} \|x-x^{\prime}\|.
\]

\item $c$ is a bounded function.

\item For all $(x,y) \in\mathbb{G}$, $\sum_{y^{\prime}\in\mathbb{L}}
r(x,y,y^{\prime})=1$, $r(x,y,y)=0$.
\end{enumerate}
\end{assumption}

We will occasionally write $c(x,y), a(x,y), b(x,y), r(x,y,y^{\prime})$ as
$c_{y}(x), a_{y}(x), b_{y}(x), r_{yy^{\prime}}(x)$, respectively.

\begin{remark}
\label{assmp_lingrowth} Assumption \ref{assum}(1) implies that, for some
$\kappa_{1}\in(0,\infty)$,
\[
\Vert b(x,y)\Vert+\Vert a(x,y)\Vert\leq\kappa_{1}(1+\Vert x\Vert
),\mbox{ for all }(x,y)\in{\mathbb{R}}^{d}\times\mathbb{L}.
\]

\end{remark}

Let
\[
\bar{\varsigma}\doteq\sup_{(x,y)\in\mathbb{G}}c_{y}(x),\;\zeta\doteq
\bar{\varsigma}+1
\]
and let $\lambda_{\zeta}$ be Lebesgue measure on $([0,\zeta],\mathcal{B}%
([0,\zeta]))$. For $(x,y,y^{\prime})\in\mathbb{R}^{d}\times\mathbb{L}%
\times\mathbb{L}$, $y\neq y^{\prime}$, let
\begin{equation}
E_{yy^{\prime}}(x)\doteq\lbrack0,c_{y}(x)r_{yy^{\prime}}(x)].
\label{eqn:defEij}%
\end{equation}
From Assumption \ref{assum}, for some $\kappa_{2}\in(0,\infty)$,
\begin{equation}
\sup_{(y,y^{\prime})\in\mathbb{L}\times\mathbb{L},y\neq y^{\prime}}%
\lambda_{\zeta}[E_{yy^{\prime}}(x)\triangle E_{yy^{\prime}}(x^{\prime}%
)]\leq\kappa_{2}\Vert x-x^{\prime}\Vert\label{liptheta}%
\end{equation}
for all $x,x^{\prime}\in\mathbb{R}^{d}$, where $\Delta$ denotes the symmetric
difference. For each fixed $x\in\mathbb{R}^{d}$, the operator $\Pi_{x}$ acting
on $\mathbb{M}(\mathbb{L})$ and defined by%
\begin{align*}
\Pi_{x}\phi(y)  &  \doteq c_{y}(x)\sum_{y^{\prime}\in\mathbb{L}}%
(\phi(y^{\prime})-\phi(y))r_{yy^{\prime}}(x)\\
&  =c_{y}(x)\left(  \sum_{y^{\prime}\in\mathbb{L}}\phi(y^{\prime
})r_{yy^{\prime}}(x)-\phi(y)\right)
\end{align*}
describes the generator of an $\mathbb{L}$-valued Markov process. Let
\[
\hat{r}_{yz}^{n}(x)\doteq\sum_{y^{\prime}\in\mathbb{L}}r_{y^{\prime}z}%
(x)\hat{r}_{yy^{\prime}}^{n-1}(x),\;n>1;\;\hat{r}_{yz}^{1}(x)=r_{yz}(x)
\]
be the $n$-step transition probability kernel of the corresponding embedded
chain. Define
\[
\alpha_{x}\doteq\min_{y,z\in\mathbb{L}}\sum_{n=1}^{|\mathbb{L}|}\hat{r}%
_{yz}^{n}(x),\quad\alpha\doteq\inf_{x}\alpha_{x};
\]%
\[
\underline{\varsigma}_{x}\doteq\min_{y\in\mathbb{L}}c(x,y),\quad
\underline{\varsigma}\doteq\inf_{x}\underline{\varsigma}_{x},\quad
\bar{\varsigma}_{x}\doteq\max_{y\in\mathbb{L}}c(x,y).
\]
Recall that, from Assumption \ref{assum}(2) $\bar{\varsigma}\doteq\sup_{x}%
\bar{\varsigma}_{x}<\infty$. Let
\[
\mathbb{T}\doteq\{(y,y^{\prime})\in\mathbb{L}\times\mathbb{L}:r_{yy^{\prime}%
}(x)>0\mbox{ for some }x\in\mathbb{R}^{d}\}
\]
and let
\[
\kappa_{3}\doteq\inf_{x\in\mathbb{R}^{d}}\min_{(y,y^{\prime})\in\mathbb{T}%
}r_{yy^{\prime}}(x).
\]
We will make the following assumption.

\begin{assumption}
\label{assum2} $\alpha>0$, $\underline{\varsigma}>0$ and $\kappa_{3} >0$.
\end{assumption}

Assumptions \ref{assum} and \ref{assum2} will be taken to hold throughout this
work and will not always be mentioned in the statement of various results. Let
$\mathcal{A}$ be an $|\mathbb{L}|\times|\mathbb{L}|$ matrix with
$\mathcal{A}_{ij}=1$ if $(i,j)\in\mathbb{T}$ and $0$ otherwise. Then,
Assumption \ref{assum2} in particular says that the adjacency matrix
$\mathcal{A}$ is irreducible.

The evolution of $Y^{\varepsilon}$ can be described through a stochastic
differential equation driven by a finite collection of Poisson random measures
which is constructed as follows. For $(i,j)\in\mathbb{T}$ let $\bar{N}_{ij}$
be a Poisson random measure (PRM) on $[0,\zeta]\times\lbrack0,T]\times
\mathbb{R}_{+}$ with intensity measure $\lambda_{\zeta}\otimes\lambda
_{T}\otimes\lambda_{\infty}$, where $\lambda_{T}$ (resp. $\lambda_{\infty}$)
denotes the Lebesgue measure on $[0,T]$ (resp. $\mathbb{R}_{+}$), on some
complete filtered probability space $(\Omega,\mathcal{F},\mathbb{P}%
,\{\mathcal{F}_{t}\}_{0\leq t\leq T})$ such that for $t\in\lbrack0,T]$,
\[
\bar{N}_{ij}(A\times\lbrack0,t]\times B)-t\lambda_{\zeta}(A)\lambda_{\infty
}(B)
\]
is a $\{\mathcal{F}_{t}\}$-martingale for all $A\in\mathcal{B}[0,\zeta]$ and
$B\in\mathcal{B}(\mathbb{R}_{+})$ with $\lambda_{\infty}(B)<\infty$. Then
\[
N_{ij}^{\varepsilon^{-1}}(dr\times dt)\doteq\bar{N}_{ij}(dr\times
dt\times\lbrack0,\varepsilon^{-1}])
\]
is a PRM on $[0,\zeta]\times\lbrack0,T]$ with intensity measure $\varepsilon
^{-1}\lambda_{\zeta}\otimes\lambda_{T}$, and can be regarded as a random
variable with values in $\mathcal{M}_{F}([0,\zeta]\times\lbrack0,T])$, the
space of finite measures on $[0,\zeta]\times\lbrack0,T]$ equipped with the
weak topology. The processes $(\bar{N}_{ij})_{(i,j)\in\mathbb{T}}$ are taken
to be mutually independent. We also suppose that on this filtered probability
space there is an $m$-dimensional $\mathcal{F}_{t}$-Brownian motion
$W=\{W(t)\}_{0\leq t\leq T}$ (which is then independent of $\bar{N}$). In
terms of $W$ and $N^{\varepsilon^{-1}}$, the Markov process $(X^{\varepsilon
},Y^{\varepsilon})\equiv\{(X^{\varepsilon}(t),Y^{\varepsilon}(t))\}_{0\leq
t\leq T}$ with initial condition $(x_{0},y_{0})\in\mathbb{G}$ is defined as
the unique pathwise solution of the following system of equations:
\begin{equation}
\left.
\begin{aligned} dX^{\varepsilon}(t) &= b(X^{\varepsilon}(t), Y^{\varepsilon}(t)) dt + \sqrt{\varepsilon} a(X^{\varepsilon}(t), Y^{\varepsilon}(t)) dW(t),&\quad X^{\varepsilon}(0) = x_0\\ dY^{\varepsilon}(t) &= \sum_{(i,j) \in \mathbb{T}}\int_{r \in [0,\zeta]} (j-i) 1_{\{Y^{\varepsilon}(t-)=i\}} 1_{E_{ij}(X^{\varepsilon}(t))}(r) N_{ij}^{\varepsilon^{-1}}( dr \times dt),&\quad Y^{\varepsilon}(0) = y_0. \end{aligned}\right\}
\label{fastslow}%
\end{equation}
From unique pathwise solvability it follows that for every $\varepsilon>0$,
there exists a measurable map $\mathcal{G}^{\varepsilon}:C([0,T]:\mathbb{R}%
^{m})\times(\mathcal{M}_{F}([0,\zeta]\times\lbrack0,T]))^{|\mathbb{T}%
|}\rightarrow C([0,T]:\mathbb{R}^{d})$ such that $X^{\varepsilon}%
=\mathcal{G}^{\varepsilon}(\sqrt{\varepsilon}W,\{\varepsilon N_{ij}%
^{\varepsilon^{-1}}\}_{ij})$.

The following is an immediate consequence of our assumptions.

\begin{theorem}
\label{assum3} For each $x\in\mathbb{R}^{d}$, there is a unique invariant
probability measure, $\nu(x)$ for the $\mathbb{L}$-valued Markov process with
generator $\Pi_{x}$.
\end{theorem}

\begin{lemma}
\label{invmeas_lip2} The mapping $x\in{\mathbb{R}}^{d}\rightarrow\nu
(x)\in\mathcal{P}(\mathbb{L})$ is Lipschitz continuous with the Lipschitz
constant $L_{{\tiny {\mbox{{\em lip}}}}}^{\nu}$ (with respect to the total
variation metric) depending only on $\alpha,\underline{\varsigma}%
,\bar{\varsigma},\kappa_{2}$ and $\kappa_{3}$. Furthermore, $\inf
_{x\in\mathbb{R}^{d}}\min_{y\in\mathbb{L}}\nu_{y}(x)\doteq\underline{\nu}>0$.
\end{lemma}

\begin{proof}
Since $\alpha>0$, the $\mathbb{L}$-valued Markov chain with transition
probabilities
\[
p_{yy^{\prime}}^{x}=\frac{1}{|\mathbb{L}|}\sum_{n=1}^{|\mathbb{L}|}\hat
{r}_{yy^{\prime}}^{n}(x),\;y,y^{\prime}\in\mathbb{L},
\]
is ergodic for each $x\in\mathbb{R}^{d}$. Since from Assumption \ref{assum}
$x\mapsto r_{yy^{\prime}}(x)$ is Lipschitz continuous, it follows that
$p_{yy^{\prime}}^{x}$ is Lipschitz continuous in $x$ and $\inf_{y,y^{\prime
}\in\mathbb{L}}\inf_{x\in\mathbb{R}^{d}}p_{yy^{\prime}}^{x}>0$. Denote the
unique invariant measure of this chain by $\pi(x)$. From Lemma 3.1 in
\cite{FW98}, $\pi(x)$ is given as a ratio of polynomials in $\{p_{yz}%
^{x}\}_{y,z\in\mathbb{L}}$. Thus $x\mapsto\pi_{y}(x)$ is Lipschitz continuous
for every $y\in\mathbb{L}$ (with Lipschitz constant depending on $\kappa
_{2},\kappa_{3}$). The lemma now follows on observing that $\nu_{y}%
(x)\propto\frac{\pi_{y}(x)}{c(x,y)},$ and hence the assertion follows from the
Lipschitz property of $x\mapsto c(x,y)$ and the properties $\underline
{\varsigma}>0$ and $\bar{\varsigma}<\infty$.
\end{proof}

\begin{lemma}
\label{lem:lemloclip} Let $f : {\mathbb{R}}^{d} \times\mathbb{L}
\rightarrow{\mathbb{R}}$ satisfy
\[
|f(x,y) - f(x^{\prime},y)|\leq L^{f}_{{\tiny {\mbox{{\em lip}}}}}
\|x-x^{\prime}\|, \quad x,x^{\prime}\in{\mathbb{R}}^{d}, y\in\mathbb{L}%
\]
for some $L^{f}_{{\tiny {\mbox{{\em lip}}}}}\in(0,\infty)$. Define $\hat f(x)
= \sum_{y \in\mathbb{L}} f(x,y) \nu_{y}(x). $ Then $\hat f$ is a locally
Lipschitz function on ${\mathbb{R}}^{d}$ with linear growth.
\end{lemma}

\begin{proof}
The linear growth property is clear from the Lipschitz property of $f$. The
local Lipschitz property follows by noting that for any compact $K\subset
{\mathbb{R}}^{d}$ and $x,x^{\prime}\in K$
\begin{align*}
|\hat{f}(x)-\hat{f}(x^{\prime})|  &  \leq\sum_{y\in\mathbb{L}}\left[  \nu
_{y}(x)|f(x,y)-f(x^{\prime},y)|+\sup_{x\in K}|f(x,y)||\nu_{y}(x)-\nu
_{y}(x^{\prime})|\right] \\
&  \leq\left(  L_{{\tiny {\mbox{lip}}}}^{f}+\max_{y\in\mathbb{L}}\sup_{x\in
K}|f(x,y)|L_{{\tiny {\mbox{lip}}}}^{\nu}\right)  \left\vert \mathbb{L}%
\right\vert \Vert x-x^{\prime}\Vert.
\end{align*}
\end{proof}

Let $\hat{b}(x)=\sum_{y\in\mathbb{L}}b(x,y)\nu_{y}(x)$, and note that by Lemma
\ref{lem:lemloclip} $\hat{b}$ is locally Lipschitz function with linear
growth. The proof of the following theorem follows along the lines of
\cite[Chapter 2, Theorem 8]{Skh76}. We omit the details since a similar result
in a controlled setting will be shown in Proposition \ref{ctrlmeas_tight}.

\begin{theorem}
\label{thm:lln} Fix $(x_{0},y_{0}) \in\mathbb{G}$. Let $(X^{\varepsilon},
Y^{\varepsilon})$ be the solution of \eqref{fastslow}. Then as $\varepsilon
\rightarrow0$, $X^{\varepsilon}$ converges uniformly on compacts in
probability to the unique solution of
\begin{equation}
\label{ode}\frac{d \xi(t)}{dt} = \hat b (\xi(t)), \quad\xi(0) = x_{0}.
\end{equation}

\end{theorem}

The unique solvability of \eqref{ode} is a consequence of the properties of
$\hat{b}$ stated before the theorem.

The solution $X^{\varepsilon}$ of the system \eqref{fastslow} can be regarded
as a $C([0,T]:\mathbb{R}^{d})$-valued random variable. The main result of this
work establishes a large deviation principle (LDP) for $X^{\varepsilon}$ in
$C([0,T]:\mathbb{R}^{d})$ as $\varepsilon\rightarrow0$. In rest of this
section we formulate the rate function for $\{X^{\varepsilon}\}$ and present
our main result.

\subsection{Rate function}

Recall that $\mathbb{M}([0,T]:\mathcal{P}(\mathbb{L}))$, $\mathbb{M}%
([0,T]:\mathbb{R}^{d})$ denote the space of measurable maps from $[0,T]$ to
$\mathcal{P}(\mathbb{L})$ and from $[0,T]$ to $\mathbb{R}^{d}$ respectively.
For $\psi=(\psi_{j})_{j\in\mathbb{L}}$, with $\psi_{j}:[0,\zeta]\rightarrow
\mathbb{R}_{+}$ a measurable map for every $j$, let
\begin{equation}
A_{ij}^{\psi}(x)=\left\{
\begin{array}
[c]{cc}%
\rho_{ij}^{\psi}(x), & i\neq j\\
-\sum_{y:y\neq j}\rho_{jy}^{\psi}(x) & i=j,
\end{array}
\right.  \label{eq:eq855}%
\end{equation}
where for $i,j\in\mathbb{L}$, $i\neq j$, $x\in\mathbb{R}^{d}$
\begin{equation}
\rho_{ij}^{\psi}(x)\doteq\int_{E_{ij}(x)}\psi_{j}(z)\lambda_{\zeta}(dz).
\label{eq:contrates}%
\end{equation}
Note that any $\psi$ as above, such that $\psi_{j}$ is integrable for each
$j$, can be identified with $\eta=(\eta_{j})_{j\in\mathbb{L}}$ such that for
each $j$, $\eta_{j}\in\mathcal{M}_{F}([0,\zeta])$ on setting $\eta
_{j}(dz)\doteq\psi_{j}(z)dz$. More generally, for $x\in\mathbb{R}^{d}$, and
any $\eta=(\eta_{j})_{j\in\mathbb{L}}$ such that each $\eta_{j}\in
\mathcal{M}_{F}([0,\zeta])$, we define $A^{\eta}(x)$ as
\begin{equation}
A_{ij}^{\eta}(x)=\left\{
\begin{array}
[c]{cc}%
\eta_{j}(E_{ij}(x)), & i\neq j\\
\  & \\
-\sum_{y:y\neq j}\eta_{y}(E_{jy}(x)) & i=j.
\end{array}
\right.  \label{eq:eq855mzr}%
\end{equation}
We will make use of such $A^{\eta}$ in the next section. Although the
introduction of a second notation for the controlled intensities is
regrettable, the measure formulation is more natural when discussing topologies.

Define $\ell:[0,\infty)\rightarrow\lbrack0,\infty)$ by $\ell(x)\doteq x\log
x-x+1$ and let
\[
\mathcal{R}\doteq\{\varphi=(\varphi_{ij})_{(i,j)\in\mathbb{T}}:\varphi
_{ij}:[0,T]\times\lbrack0,\zeta]\rightarrow\mathbb{R}_{+}\mbox{ is a
measurable map},(i,j)\in\mathbb{T}\}.
\]
For $\xi\in C([0,T]:\mathbb{R}^{d})$, define
\begin{equation}
I(\xi)\doteq\inf_{(u,\varphi,\pi)\in\mathcal{A}(\xi)}\left\{  \sum
_{i\in\mathbb{L}}\frac{1}{2}\int_{0}^{T}\Vert u_{i}(s)\Vert^{2}\pi
_{i}(s)ds+\sum_{(i,j)\in\mathbb{T}}\int_{[0,\zeta]\times\lbrack0,T]}%
\ell(\varphi_{ij}(s,z))\pi_{i}(s)\lambda_{\zeta}(dz)ds\right\}  ,
\label{eq:rtfnnew}%
\end{equation}
where $\mathcal{A}(\xi)$ is the collection of all
\[
(u=(u_{i}),\varphi=(\varphi_{ij}),\pi=(\pi_{i}))\in\mathbb{M}([0,T]:\mathbb{R}%
^{m})^{|\mathbb{L}|}\times\mathcal{R}\times\mathbb{M}([0,T]:\mathcal{P}%
(\mathbb{L}))
\]
such that $\int_{0}^{T}\Vert u_{i}(s)\Vert^{2}\pi_{i}(s)ds<\infty$ for each
$i\in\mathbb{L}$,
\begin{equation}
\xi(t)=x_{0}+\sum_{j\in\mathbb{L}}\int_{0}^{t}b_{j}(\xi(s))\pi_{j}%
(s)ds+\sum_{j\in\mathbb{L}}\int_{0}^{t}a_{j}(\xi(s))u_{j}(s)\pi_{j}%
(s)ds,\;t\in\lbrack0,T] \label{eq:stateq}%
\end{equation}
and
\begin{equation}
\sum_{i\in\mathbb{L}}\pi_{i}(s)A_{ij}^{\varphi_{i,\cdot}(s,\cdot)}%
(\xi(s))=0,\mbox{ for a.e. }s\in\lbrack0,T]\mbox{ and }j\in\mathbb{L},
\label{eq:eqinvar}%
\end{equation}
where $\varphi_{i,\cdot}=(\varphi_{i,j})_{j\in\mathbb{L}}$ (with the
convention $\varphi_{i,j}=1$ if $(i,j)\notin\mathbb{T}$). Equation
\eqref{eq:eqinvar} characterizes the invariant distributions that would be
associated with controlled PRMs with controls $\varphi_{ij}$, which influence
the rate of transition from $i$ to $j$ through (\ref{eq:contrates}). This form
of the rate function is very much analogous to the control formulation of a
small noise diffusion as in \cite{BoDu98}. In \eqref{eq:eqinvar} we follow the
convention that $0\cdot\infty=0$ and $\infty-\infty=\infty$.

The following is the main result of this work. Recall that Assumptions
\ref{assum} and \ref{assum2} are taken to hold throughout the paper. A
function $I:C([0,T]:\mathbb{R}^{d})\rightarrow\lbrack0,\infty]$ is called a
rate function on $C([0,T]:\mathbb{R}^{d})$ if it has compact sub-level sets,
namely for every $\alpha\in(0,\infty)$, the set $\{\xi\in C([0,T]:\mathbb{R}%
^{d}):I(\xi)\leq\alpha\}$ is a compact subset of $C([0,T]:\mathbb{R}^{d})$.

\begin{theorem}
\label{thm:avg} The map $I$ in \eqref{eq:rtfnnew} is a rate function on
$C([0,T]:\mathbb{R}^{d})$ and $\left\{  X^{\varepsilon}\right\}
_{\varepsilon>0}$ satisfies the Laplace principle on $C([0,T]:\mathbb{R}^{d}%
)$, as $\varepsilon\rightarrow0$, with rate function $I$. Namely, for all
$F\in C_{b}(C([0,T]:\mathbb{R}^{d}))$,
\begin{equation}
\lim_{\varepsilon\rightarrow0}-\varepsilon\log\mathbb{E}\left[  \exp\left(
-\varepsilon^{-1}F(X^{\varepsilon})\right)  \right]  =\inf_{\xi\in
C([0,T]:\mathbb{R}^{d})}\{F(\xi)+I(\xi)\}. \label{maintoshowlappri}%
\end{equation}

\end{theorem}

The proof of the Laplace upper bound
\begin{equation}
\limsup_{\varepsilon\rightarrow0}\varepsilon\log\mathbb{E}\left[  \exp\left(
-\varepsilon^{-1}F(X^{\varepsilon})\right)  \right]  \leq-\inf_{\xi\in
C([0,T]:\mathbb{R}^{d})}\{F(\xi)+I(\xi)\}, \label{maintoshowupp}%
\end{equation}
which corresponds to a variational lower bound, is given in Section
\ref{uppbd}. The corresponding lower bound
\begin{equation}
\liminf_{\varepsilon\rightarrow0}\varepsilon\log\mathbb{E}\left[  \exp\left(
-\varepsilon^{-1}F(X^{\varepsilon})\right)  \right]  \geq-\inf_{\xi\in
C([0,T]:\mathbb{R}^{d})}\{F(\xi)+I(\xi)\}, \label{maintoshowlow}%
\end{equation}
which is a variational upper bound, is proven in Section \ref{lowbd}. The fact
that $I$ is a rate function is shown in Section \ref{sec:cptlevset}
(Proposition \ref{prop611}).

\begin{remark}
$\,$

\begin{enumerate}
[(a)]

\item \textrm{Note that the rate function $I$ depends on the initial value of
$(X^{\varepsilon},Y^{\varepsilon})$, namely $(x_{0},y_{0})$. To emphasize this
dependence denote $I$ as $I_{x_{0},y_{0}}$. Using a straightforward argument
by contradiction, one can show that the Laplace limit (\ref{maintoshowlappri}%
), with $I$ replaced by $I_{x_{0},y_{0}}$ on the right side, hold uniformly
for }$y_{0}\in\mathbb{L}$\textrm{ and for }$x_{0}$\textrm{ in any compact
subset of }$R^{d}$\textrm{.}

\item For $\varepsilon>0$, define the $C([0,T]:\mathcal{M}_{F}(\mathbb{L}))$
valued random variable $\varrho^{\varepsilon}$ by
\[
\varrho^{\varepsilon}(t,A)\doteq\int_{0}^{t}1_{A}(Y^{\varepsilon}%
(s))ds,\;t\in\lbrack0,T],\;A\subset\mathbb{L}.
\]
Then Theorem \ref{thm:avg} can be strengthened as follows: The pair
$(X^{\varepsilon},\varrho^{\varepsilon})$ satisfies a large deviation
principle on $C([0,T]:\mathbb{R}^{d}\times\mathcal{M}_{F}(\mathbb{L}))$ with
rate function $\bar{I}$ where for $(\xi,\vartheta)\in C([0,T]:\mathbb{R}%
^{d}\times\mathcal{M}_{F}(\mathbb{L}))$, $\bar{I}(\xi,\vartheta)$ is defined
by the right side of \eqref{eq:rtfnnew} by replacing $\mathcal{A}(\xi)$ with
$\bar{\mathcal{A}}(\xi,\vartheta)$ which is the collection of all
$(u,\varphi,\pi)$ that satisfy in addition to \eqref{eq:stateq} and
\eqref{eq:eqinvar} the equality
\[
\vartheta(t,A)=\sum_{j\in A}\int_{0}^{t}\pi_{j}(s)ds,\;t\in\lbrack
0,T],\;A\subset\mathbb{L}.
\]

\end{enumerate}
\end{remark}

\subsection{An Equivalent Representation for the Rate Function}

\label{sec:equivrepnrf} In this section we present a different representation
for the rate function that will be more convenient to work with in some
instances. Recall that $\lambda_{\zeta}$ denotes Lebesgue measure on
$[0,\zeta]$.

Let $\hat{\ell}:\mathcal{M}_{F}[0,\zeta]\rightarrow\lbrack0,\infty]$ be
defined by
\[
\hat{\ell}(\eta)\doteq%
\begin{cases}
\int_{\lbrack0,\zeta]}\ell\left(  \frac{d\eta}{d\lambda_{\zeta}}(z)\right)
\lambda_{\zeta}(dz), & \quad\mbox{ if }\eta\ll\lambda_{\zeta},\\
\infty, & \quad\mbox{otherwise}.
\end{cases}
\]
For $\eta=(\eta_{i})_{i\in\mathbb{L}}$, where each $\eta_{i}\in\mathcal{M}%
_{F}[0,\zeta]$, with an abuse of notation we define
\[
\hat{\ell}(\eta)=\sum_{i\in\mathbb{L}}\hat{\ell}(\eta_{i}).
\]

Let $\mathcal{P}_{1}(\mathbb{H}_{T})$ denote the space of finite measures $Q$
on
\begin{equation}
\mathbb{H}_{T}\doteq\lbrack0,T]\times\mathbb{L}\times(\mathcal{M}_{F}%
[0,\zeta])^{|\mathbb{L}|}\times\mathbb{R}^{m}\label{eqn:defofH}%
\end{equation}
such that
\[
Q([a,b]\times\mathbb{L}\times(\mathcal{M}_{F}[0,\zeta])^{|\mathbb{L}|}%
\times\mathbb{R}^{m})=b-a,\;\mbox{ for all }0\leq a\leq b\leq T.
\]
In other words, denoting the marginal on the $i^{th}$ coordinate of
$\mathbb{H}_{T}$ by $[Q]_{i}$, $Q\in\mathcal{M}_{F}(\mathbb{H}_{T})$ is in
$\mathcal{P}_{1}(\mathbb{H}_{T})$ if and only if $[Q]_{1}=\lambda_{T}$, where
$\lambda_{T}$ is Lebesgue measure on $[0,T]$. For notational simplicity, we
will denote a typical $(s,y,\eta,z)\in\mathbb{H}_{T}$ as $\mathbf{v}$. $Q$
encodes time ($s$), the state of the controlled fast process ($y$), the
measures controlling the jump rates ($\eta$), and the control ($z$) applied to
perturb the mean of the Brownian motion. Recall that $b(x,y),a(x,y)$ are the
same as $b_{y}(x),a_{y}(x)$. For $\xi\in C([0,T]:\mathbb{R}^{d})$, let
$\hat{\mathcal{A}}(\xi)$ be the family of all $Q\in\mathcal{P}_{1}%
(\mathbb{H}_{T})$ such that
\begin{equation}
\int_{\mathbb{H}_{T}}\Vert z\Vert^{2}Q(d\mathbf{v})<\infty,\label{ab2013}%
\end{equation}%
\begin{equation}
\xi(t)=x+\int_{\mathbb{H}_{t}}b(\xi(s),y)Q(d\mathbf{v})+\int_{\mathbb{H}_{t}%
}a(\xi(s),y)zQ(d\mathbf{v}),\label{xi}%
\end{equation}
and
\begin{equation}
\int_{\mathbb{H}_{t}}A_{y,j}^{\eta}(\xi(s))Q(d\mathbf{v})=0\quad\mbox{for
all }j\in\mathbb{L}\mbox{ and a.e. }t\in\lbrack0,T],\label{inv}%
\end{equation}
where $\mathbb{H}_{t}=[0,t]\times\mathbb{L}\times(\mathcal{M}_{F}%
[0,\zeta])^{|\mathbb{L}|}\times\mathbb{R}^{m}$. Equation (\ref{xi}) gives the
controlled dynamics, and (\ref{inv}) guarantees that the conditional
distribution of $Q$, in the $y$-variable (i.e. the second coordinate) given
the time instant $s\in\lbrack0,T]$, the state $\xi(s)$ of the dynamics, and
that the rate control measure $\eta$ are used, is the stationary distribution
associated with the generator $A^{\eta}(\xi(s))$.
Define the function $\hat{I}:C([0,T]:\mathbb{R}^{d})\rightarrow\lbrack
0,\infty)$ by
\begin{equation}
\hat{I}(\xi)=\inf_{Q\in\hat{\mathcal{A}}(\xi)}\left\{  \int_{\mathbb{H}_{T}%
}\left[  \frac{1}{2}\Vert z\Vert^{2}+\hat{\ell}(\eta)\right]  Q(d\mathbf{v}%
)\right\}  .\label{ratefn}%
\end{equation}
In the expression for $\hat{I}$, all statistical relations between the
controls $(z,\eta)$ and the empirical measure for the fast variables are
determined by the joint distribution. It is a natural object for purposes of
weak convergence analysis, and these relations can be determined by the use of
suitable test functions. The following result shows that $\hat{I}$ and {$I$}
are the same.

\begin{proposition}
\label{prop:ieqihat} For every $\xi\in C([0,T]:\mathbb{R}^{d})$ , $\hat{I}%
(\xi)=I(\xi)$.
\end{proposition}

\begin{proof}
Fix $\xi\in C([0,T]:\mathbb{R}^{d})$. We first show that $\hat{I}(\xi)\leq
I(\xi)$. Without loss of generality we assume that {$I$}$(\xi)<\infty$. Fix
$\varepsilon>0$ and let $(u,\varphi,\pi)\in\mathcal{A}(\xi)$ be such that
\begin{equation}
\sum_{i\in\mathbb{L}}\frac{1}{2}\int_{0}^{T}\Vert u_{i}(s)\Vert^{2}\pi
_{i}(s)ds+\sum_{(i,j)\in\mathbb{T}}\int_{[0,\zeta]\times\lbrack0,T]}%
\ell(\varphi_{ij}(s,z))\pi_{i}(s)\lambda_{\zeta}(dz)ds\leq{I}(\xi
)+\varepsilon.\label{eq:eq607-17}%
\end{equation}
Define $\bar{\eta}_{ij}(s,dz)\in\mathcal{M}_{F}[0,\zeta]$ for $i,j\in
\mathbb{L}$ and $s\in\lbrack0,T]$ by%
\[
\bar{\eta}_{ij}(s,dz)\doteq%
\begin{cases}
\varphi_{ij}(s,z)\lambda_{\zeta}(dz), & \quad\mbox{ if }(i,j)\in
\mathbb{T}\mbox{ and }z\mapsto\varphi_{ij}(s,z)\mbox{ is integrable},\\
\lambda_{\zeta}(dz), & \quad\mbox{ otherwise. }
\end{cases}
\]
Let $\bar{\eta}_{i}(s)\doteq(\bar{\eta}_{ij}(s,\cdot))_{j\in\mathbb{L}}$.
Define $Q\in\mathcal{P}_{1}(\mathbb{H}_{T})$ by
\[
Q([a,b]\times\{i\}\times A\times B)\doteq\int_{a}^{b}\pi_{i}(s)\delta
_{\bar{\eta}_{i}(s)}(A)\delta_{u_{i}(s)}(B)\ ds,
\]
for $A\in\mathcal{B}((\mathcal{M}_{F}[0,\zeta])^{|\mathbb{L}|})$,
$B\in\mathcal{B}(\mathbb{R}^{m})$, $0\leq a\leq b\leq T$, and $i\in\mathbb{L}%
$. Then it is easy to verify that $Q\in\hat{\mathcal{A}}(\xi)$ and
\[
\int_{\mathbb{H}_{T}}\left[  \frac{1}{2}\Vert z\Vert^{2}+\hat{\ell}%
(\eta)\right]  Q(d\mathbf{v})
\]
equals the left side of \eqref{eq:eq607-17}. This proves that $\hat{I}%
(\xi)\leq I(\xi)+\varepsilon$. Since $\varepsilon>0$ is arbitrary, we have
$\hat{I}(\xi)\leq I(\xi)$. We now consider the reverse inequality, namely
{$I$}$(\xi)\leq\hat{I}(\xi)$. We assume without loss of generality that
$\hat{I}(\xi)<\infty$. Let $Q\in\hat{\mathcal{A}}(\xi)$ be such that
\[
\int_{\mathbb{H}_{T}}\left[  \frac{1}{2}\Vert z\Vert^{2}+\hat{\ell}%
(\eta)\right]  Q(d\mathbf{v})\leq\hat{I}(\xi)+\varepsilon.
\]
Let $[Q]_{34|12}(d\eta\times dz|y,s)$ denote the conditional distribution on
the third and fourth coordinates given the first and second. Disintegrate the
measure $Q$ as
\[
Q(ds\times\{y\}\times d\eta\times dz)=ds\,\pi_{y}(s)\,[Q]_{34|12}(d\eta\times
dz|y,s).
\]
Define
\[
u_{y}(s)\doteq\int_{(\mathcal{M}_{F}[0,\zeta])^{|\mathbb{L}|}\times
\mathbb{R}^{m}}z[Q]_{34|12}(d\eta\times dz|y,s),\;y\in\mathbb{L},s\in
\lbrack0,T].
\]
Also, for $(y,s)\in\mathbb{L}\times\lbrack0,T]$, let
\begin{equation}
\bar{\eta}_{y}(s)\doteq\int_{(\mathcal{M}_{F}[0,\zeta])^{|\mathbb{L}|}%
\times\mathbb{R}^{m}}\eta\lbrack Q]_{34|12}(d\eta\times
dz|y,s)\label{eqn:defofetabar}%
\end{equation}
and write $\bar{\eta}_{y}=(\bar{\eta}_{yy^{\prime}})_{y^{\prime}\in\mathbb{L}%
}$. By convexity
\[
\hat{\ell}(\bar{\eta}_{y}(s))\leq\int_{(\mathcal{M}_{F}[0,\zeta])^{|\mathbb{L}%
|}\times\mathbb{R}^{m}}\hat{\ell}(\eta)[Q]_{34|12}(d\eta\times dz|y,s)
\]
and therefore
\[
\sum_{y\in\mathbb{L}}\int_{[0,T]}\pi_{y}(s)\hat{\ell}(\bar{\eta}_{y}%
(s))ds\leq\int_{\mathbb{H}_{T}}\hat{\ell}(\eta)Q(d\mathbf{v}).
\]
Define
\begin{equation}
\varphi_{yy^{\prime}}(s,z)\doteq%
\begin{cases}
\frac{d\bar{\eta}_{yy^{\prime}}(s,\cdot)}{d\lambda_{\zeta}(\cdot)}(z), &
\quad\mbox{ if }\bar{\eta}_{yy^{\prime}}(s,\cdot)\ll\lambda_{\zeta}(\cdot),\\
1, & \quad\mbox{ otherwise}.
\end{cases}
\label{eqn:varphi}%
\end{equation}
Then note that
\[
\sum_{(i,j)\in\mathbb{T}}\int_{[0,\zeta]\times\lbrack0,T]}\ell(\varphi
_{ij}(s,z))\pi_{i}(s)\lambda_{\zeta}(dz)ds\leq\int_{\mathbb{H}_{T}}\hat{\ell
}(\eta)Q(d\mathbf{v}).
\]
A similar convexity argument shows that
\[
\sum_{i\in\mathbb{L}}\frac{1}{2}\int_{0}^{T}\Vert u_{i}(s)\Vert^{2}\pi
_{i}(s)ds\leq\frac{1}{2}\int_{\mathbb{H}_{T}}\Vert z\Vert^{2}Q(d\mathbf{v}).
\]
To complete the proof it suffices to show that
\[
(u=(u_{i}),\varphi=(\varphi_{ij}),\pi=(\pi_{i}))\in\mathcal{A}(\xi).
\]
First, it is easily checked that $\xi$ satisfies \eqref{eq:stateq}. Thus it
remains to verify \eqref{eq:eqinvar}. Since $Q\in\hat{\mathcal{A}}(\xi)$, from
\eqref{inv} we have that for all $j\in\mathbb{L}$ and a.e. $s\in\lbrack0,T]$
\[
\sum_{y\in\mathbb{L}}\pi_{y}(s)\int_{(\mathcal{M}_{F}[0,\zeta])^{|\mathbb{L}%
|}\times\mathbb{R}^{m}}A_{y,j}^{\eta}(\xi(s))[Q]_{34|12}(d\eta\times
dz|y,s)=0.
\]
This equality can be rewritten as
\[
\sum_{y:y\neq j}\pi_{y}(s)\int A_{y,j}^{\eta}(\xi(s))[Q]_{34|12}(d\eta\times
dz|y,s)=\pi_{j}(s)\sum_{i:i\neq j}\int A_{j,i}^{\eta}(\xi(s))[Q]_{34|12}%
(d\eta\times dz|j,s).
\]
Using the definition of $A^{\eta}$ in (\ref{eq:eq855}), the last display
becomes
\[
\sum_{y:y\neq j}\pi_{y}(s)\int\eta_{j}(E_{yj}(\xi(s)))[Q]_{34|12}(d\eta\times
dz|y,s)=\pi_{j}(s)\sum_{i:i\neq j}\int\eta_{i}(E_{ji}(\xi(s)))[Q]_{34|12}%
(d\eta\times dz|j,s),
\]
which owing to the definition of $\bar{\eta}_{y}$ in (\ref{eqn:defofetabar})
is the same as
\[
\sum_{y:y\neq j}\pi_{y}(s)\bar{\eta}_{yj}(s,E_{yj}(\xi(s)))=\pi_{j}%
(s)\sum_{i:i\neq j}\bar{\eta}_{ji}(s,E_{ji}(\xi(s))).
\]
From the definition of $(\varphi_{ij})$ in (\ref{eqn:varphi}) it is now
immediate that this is same as \eqref{eq:eqinvar}.
\end{proof}

\subsection{Compact Level Sets}

\label{sec:cptlevset} We first prove the following lemmas, which will be used
in the proof of the main result of this section.

\begin{lemma}
\label{Gen_ineq} For $\eta\in(\mathcal{M}_{F}[0,\zeta])^{|\mathbb{L}|}$ and
$x\in\mathbb{R}^{d}$, let $A^{\eta}(x)$ be as defined in \eqref{eq:eq855mzr}.
Then there is a $c_{1}\in(0,\infty)$ such that for any $M\in\lbrack1,\infty)$,
$x,x^{\prime}\in\mathbb{R}^{d}$, $\eta\in(\mathcal{M}_{F}[0,\zeta
])^{|\mathbb{L}|}$,
\[
\sup_{i,j\in\mathbb{L},i\neq j}\left\vert A_{ij}^{\eta}(x)-A_{ij}^{\eta
}(x^{\prime})\right\vert \leq c_{1}\left(  e^{M}\Vert x-x^{\prime}\Vert
+\frac{\hat{\ell}(\eta)}{M}\right)  .
\]

\end{lemma}

\begin{proof}
Let $\eta=(\eta_{i})_{i=1}^{|\mathbb{L}|}$. Note that the result is automatic
if $\eta_{i}\not \ll \lambda_{\zeta}$ for some $i\in\mathbb{L}$. Assume now
that $\eta_{i}\ll\lambda_{\zeta}$ for each $i$. The following inequality will
be used in the proof: for $u,v\in(0,\infty)$ and $\sigma\in\lbrack1,\infty)$,
\begin{equation}
uv\leq e^{\sigma u}+\frac{1}{\sigma}(v\log v-v+1)=e^{\sigma u}+\frac{1}%
{\sigma}\ell(v). \label{eq:eq1824}%
\end{equation}
Fix $x,x^{\prime}\in\mathbb{R}^{d}$ and $i\neq j$ in $\mathbb{L}$. Denoting
the set $E_{ij}(x)\Delta E_{ij}(x^{\prime})$ by $\bar{E}_{ij}$, for
$M\in\lbrack1,\infty)$
\begin{align*}
\left\vert A_{ij}^{\eta}(x)-A_{ij}^{\eta}(x^{\prime})\right\vert  &
=|\eta_{j}(E_{ij}(x))-\eta_{j}(E_{ij}(x^{\prime}))|\leq\eta_{j}(\bar{E}%
_{ij})\\
&  =\int_{\bar{E}_{ij}}1\cdot\frac{d\eta_{j}}{d\lambda_{\zeta}}(r)\lambda
_{\zeta}(dr)\leq\int_{\bar{E}}e^{M}\lambda_{\zeta}(dr)+\frac{1}{M}\int
_{\bar{E}_{ij}}\ell\left(  \frac{d\eta_{j}}{d\lambda_{\zeta}}(r)\right)
\lambda_{\zeta}(dr)\\
&  \leq e^{M}\kappa_{2}\Vert x-x^{\prime}\Vert+\frac{1}{M}\hat{\ell}(\eta
_{j}),
\end{align*}
where the inequality on the second line follows from \eqref{eq:eq1824} and the
last inequality follows from \eqref{liptheta}. The result follows.
\end{proof}

\begin{remark}
\label{rem:rem1017} Using the inequality in \eqref{eq:eq1824} one can
similarly show that there exists a $C_{1}\in(0,\infty)$ such that for any
$M\in\lbrack1,\infty)$ and $\eta\in(\mathcal{M}_{F}[0,\zeta])^{|\mathbb{L}|}%
$,
\[
\sup_{x\in\mathbb{R}^{d}}\max_{i,j\in\mathbb{L},i\neq j}A_{ij}^{\eta}(x)\leq
C_{1}\left(  e^{M}+\frac{\hat{\ell}(\eta)}{M}\right)  .
\]

\end{remark}

The following lemma will be used at several places in weak convergence arguments.

\begin{lemma}
\label{lem:genQ} Let $(\eta^{n},Z^{n},Y^{n})$ be a sequence of $(\mathcal{M}%
_{F}[0,\zeta])^{|\mathbb{L}|}\times\mathbb{R}^{d}\times\mathbb{L}$ valued
random variables given on a probability space $(\bar{\Omega},\bar{\mathcal{F}%
},\bar{\mathbb{P}})$, which converges in probability to $(\bar{\eta},\bar
{Z},\bar{Y})$. Further suppose that, for some $\bar{C}\in(0,\infty)$,
$\sup_{n\in\mathbb{N}}\bar{\mathbb{E}}\hat{\ell}(\eta^{n})\leq\bar{C}$. Then
for all $j\in\mathbb{L}$, $A_{Y^{n},j}^{\eta^{n}}(Z^{n})$ converges in $L^{1}$
to $A_{\bar{Y},j}^{\bar{\eta}}(\bar{Z})$, as $n\rightarrow\infty$.
\end{lemma}

\begin{proof}
Fix $j\in\mathbb{L}$. Using Lemma \ref{Gen_ineq} we see that
\begin{equation}
|A_{Y^{n},j}^{\eta^{n}}(Z^{n})-A_{Y^{n},j}^{\eta^{n}}(\bar{Z})|\rightarrow
0,\quad\mbox{in probability.} \label{E1new}%
\end{equation}
From Fatou's lemma and the lower semicontinuity of $\hat{\ell}$ it follows
that $\bar{\mathbb{E}}\hat{\ell}(\bar{\eta})\leq\bar{C}$. Next, assume without
loss of generality by using a subsequential argument that the convergence of
$(\eta^{n},Z^{n},Y^{n})$ holds a.s. Let $\Omega_{0}\in\bar{\mathcal{F}}$ be
such that $\bar{\mathbb{P}}(\Omega_{0})=1$ and $\forall\omega\in\Omega_{0}$,
\[
(\eta^{n}(\omega),Z^{n}(\omega),Y^{n}(\omega))\rightarrow(\bar{\eta}%
(\omega),\bar{Z}(\omega),\bar{Y}(\omega)),
\]
and $\hat{\ell}(\bar{\eta}(\omega))<\infty$. Fix $\omega\in\Omega_{0}$. We
will suppress $\omega$ from the notation at some places below. Since
$\mathbb{L}$ is a finite set, there exists an $N\equiv N(\omega)$ such that
for $n\geq N$, $Y^{n}(\omega)=Y(\omega)$, and consequently, $A_{Y^{n},j}%
^{\eta^{n}}(\bar{Z})=A_{\bar{Y},j}^{\eta^{n}}(\bar{Z})$. Since $\eta
^{n}\rightarrow\bar{\eta}$ and $\bar{\eta}_{j}$ is absolutely continuous with
respect to $\lambda_{\zeta}$ for every $j$, we conclude that
\[
\eta_{j}^{n}(E_{ij}(x))\rightarrow\bar{\eta}_{j}(E_{ij}(x)),\quad\forall
i,j\in\mathbb{L},\,i\neq j,\,x\in\mathbb{R}^{d}.
\]
Using this in (\ref{eq:eq855mzr}) we now have that
\[
A_{Y^{n},j}^{\eta^{n}}(\bar{Z})\rightarrow A_{\bar{Y},j}^{\bar{\eta}}(\bar
{Z}),\quad\mbox{ a.s.}
\]
Combining this with (\ref{E1new}) we have that
\[
A_{Y^{n},j}^{\eta^{n}}(Z^{n})\rightarrow A_{\bar{Y},j}^{\bar{\eta}}(\bar
{Z})\quad\mbox{in probability.}
\]
Finally to show the $L^{1}$-convergence, it suffices to argue that $\eta
_{j}^{n}[0,\zeta]$, or equivalently $\eta_{\ast,j}^{n}[0,\zeta]\doteq\eta
^{n}[0,\zeta]/\lambda_{\zeta}[0,\zeta]$, is uniformly integrable. For this
note that
\[
\bar{\mathbb{E}}(\ell(\eta_{\ast,j}^{n}[0,\zeta]))=\bar{\mathbb{E}}\left(
\ell\left(  \frac{1}{\zeta}\int_{[0,\zeta]}\frac{d\eta_{j}^{n}}{d\lambda
_{\zeta}}(r)\lambda_{\zeta}(dr)\right)  \right)  \leq\frac{1}{\zeta}%
\bar{\mathbb{E}}(\hat{\ell}(\eta_{j}^{n}))\leq\frac{\bar{C}}{\zeta},
\]
where the first inequality follows from the convexity of $\ell$. The desired
uniform integrability is now an immediate consequence of the superlinearity of
$\ell$.
\end{proof}

We now show that the function $I$, which is same as the function $\hat{I}$
defined in \eqref{ratefn}, is a rate function on $C([0,T]:\mathbb{R}^{d})$.

\begin{proposition}
\label{prop611} For every $M\in(0,\infty)$, the set $\mathbb{U}_{M}\doteq
\{\xi\in C([0,T]:\mathbb{R}^{d})|I(\xi)\leq M\}$ is compact, and consequently
$I$ is a rate function on $C([0,T]:\mathbb{R}^{d})$.
\end{proposition}

\begin{proof}
Let $\{\xi_{n}\}_{n\in\mathbb{N}}$ be a sequence in $\mathbb{U}_{M}$. Since
$I(\xi_{n})\leq M$, we have from Proposition \ref{prop:ieqihat} that for each
$n\in\mathbb{N}$, there exists some $Q_{n}\in\hat{\mathcal{A}}(\xi_{n})$, such
that
\begin{equation}
\int_{\mathbb{H}_{T}}\left[  \frac{1}{2}\Vert z\Vert^{2}+\hat{\ell}%
(\eta)\right]  Q_{n}(d\mathbf{v})\leq M+\frac{1}{n}. \label{bd1}%
\end{equation}
Recall that $\mathcal{P}_{1}(\mathbb{H}_{T})$ is the space of finite measures
on $\mathbb{H}_{T}$ defined in (\ref{eqn:defofH}) whose first marginal is
Lebesgue measure. It suffices to show that $\{\xi_{n}\}$ is pre-compact, and
every limit point belongs to $\mathbb{U}_{M}$. For this, we prove that:
\begin{enumerate}
[(i)]
\item $\{Q_{n},\xi_{n}\}_{n\in\mathbb{N}}$ is pre-compact in $\mathcal{P}%
_{1}(\mathbb{H}_{T})\times C([0,T]:\mathbb{R}^{d})$;
\item Any limit point $\{Q,\xi\}$ satisfies the properties
\begin{enumerate}
[(a)]
\item $\int_{\mathbb{H}_{T}}\left[  \frac{1}{2}\Vert z\Vert^{2}+\hat{\ell
}(\eta)\right]  Q(d\mathbf{v})\leq M$,
\item (\ref{xi}) holds,
\item (\ref{inv}) holds.
\end{enumerate}
\end{enumerate}
We now prove (i). Since $\mathbb{L}$ is a finite (and hence compact) set and
$[Q_{n}]_{1}=\lambda_{T}$ for all $n$, in order to prove the pre-compactness
of $\{Q_{n}\}$, it suffices to show that for every $\delta>0$, there exists a
$C_{1}\in(0,\infty)$ such that
\begin{equation}
\sup_{n\in\mathbb{N}}Q_{n}\left\{  (s,y,\eta,z)\in\mathbb{H}_{T}:\,\sum
_{j\in\mathbb{L}}\eta_{j}[0,\zeta]+\Vert z\Vert>C_{1}\right\}  \leq\delta.
\label{eq:ratetight}%
\end{equation}
From \eqref{bd1}
\begin{equation}
\int_{\mathbb{H}_{T}}\Vert z\Vert^{2}Q_{n}(d\mathbf{v})\leq2(M+1),\;\int
_{\mathbb{H}_{T}}\hat{\ell}(\eta)Q_{n}(d\mathbf{v})\leq M+1 \label{eq:eq1018}%
\end{equation}
and using \eqref{eq:eq1824} with $\sigma=1,u=1$ and $v=\eta_{j}[0,\zeta]$,
\[
\sum_{j\in\mathbb{L}}\int_{\mathbb{H}_{T}}\eta_{j}[0,\zeta]Q_{n}%
(d\mathbf{v})\leq|\mathbb{L}|\zeta e+\int_{\mathbb{H}_{T}}\hat{\ell}%
(\eta)Q_{n}(d\mathbf{v})\leq|\mathbb{L}|\zeta e+M+1.
\]
The inequality in \eqref{eq:ratetight} is now immediate from the last two
displays. Thus $\{Q_{n}\}$ is pre-compact in $\mathcal{P}_{1}(\mathbb{H}_{T}%
)$. Next we argue the pre-compactness of $\{\xi_{n}\}$. We first show that
\begin{equation}
\sup_{n\in\mathbb{N}}\sup_{0\leq t\leq T}\Vert\xi_{n}(t)\Vert^{2}\doteq
C_{2}<\infty. \label{eq:eq1115}%
\end{equation}
Since $Q_{n}\in\hat{\mathcal{A}}(\xi_{n})$, we have
\[
\xi_{n}(t)=x+\int_{\mathbb{H}_{t}}b(\xi_{n}(s),y)Q_{n}(d\mathbf{v}%
)+\int_{\mathbb{H}_{t}}a(\xi_{n}(s),y)zQ_{n}(d\mathbf{v}).
\]
Using the linear growth property of $a,b$ (Remark \ref{assmp_lingrowth}), we
have
\[
\Vert\xi_{n}(t)\Vert^{2}\leq3\Vert x\Vert^{2}+3T\int_{\mathbb{H}_{t}}%
\kappa_{1}^{2}(\Vert\xi_{n}(s)\Vert+1)^{2}Q_{n}(d\mathbf{v})+3\kappa_{1}%
^{2}\int_{\mathbb{H}_{t}}(\Vert\xi_{n}(s)\Vert+1)^{2}Q_{n}(d\mathbf{v}%
)\int_{\mathbb{H}_{t}}\Vert z\Vert^{2}Q_{n}(d\mathbf{v}).
\]
Thus from \eqref{eq:eq1018}
\[
\Vert\xi_{n}(t)\Vert^{2}\leq3\Vert x\Vert^{2}+6\kappa_{1}^{2}(T+2(M+1))\int
_{[0,t]}(\Vert\xi_{n}(s)\Vert^{2}+1)ds.
\]
The inequality in \eqref{eq:eq1115} now follows by Gronwall's inequality.
Next, consider fluctuations of $\xi_{n}$. For $0\leq t_{0}\leq t_{1}\leq T$,
\begin{align*}
\Vert\xi_{n}(t_{1})-\xi_{n}(t_{0})\Vert &  \leq\int_{\mathbb{H}_{t_{1}%
}\setminus\mathbb{H}_{t_{0}}}\Vert b(\xi_{n}(s),y)\Vert Q_{n}(d\mathbf{v}%
)+\int_{\mathbb{H}_{t_{1}}\setminus\mathbb{H}_{t_{0}}}\Vert a(\xi
_{n}(s),y)\Vert\Vert z\Vert Q_{n}(d\mathbf{v})\\
&  \leq\int_{\mathbb{H}_{t_{1}}\setminus\mathbb{H}_{t_{0}}}\kappa_{1}(\Vert
\xi_{n}(s)\Vert+1)Q_{n}(d\mathbf{v})\\
&  \quad+\left(  \int_{\mathbb{H}_{t_{1}}\setminus\mathbb{H}_{t_{0}}}%
(\kappa_{1}(\Vert\xi_{n}(s)\Vert+1))^{2}Q_{n}(d\mathbf{v})\int_{\mathbb{H}%
_{T}}\Vert z\Vert^{2}Q_{n}(d\mathbf{v})\right)  ^{1/2}\\
&  \leq C_{3}|t_{1}-t_{0}|^{1/2},
\end{align*}
where the last inequality uses \eqref{eq:eq1115} and \eqref{eq:eq1018} and
$C_{3}$ depends only on $C_{2},M,\kappa_{1}$ and $T$. This estimate together
with \eqref{eq:eq1115} shows that $\{\xi_{n}\}$ is pre-compact in
$C([0,T]:\mathbb{R}^{d})$. We now prove (ii). Let $(Q,\xi)$ be a limit point
of the sequence $\{(Q_{n},\xi_{n})\}_{n\in\mathbb{N}}$. Part (a) is immediate
from \eqref{bd1} using Fatou's lemma and the lower semicontinuity of
$\hat{\ell}$. Consider now part (b). We assume without loss of generality that
the full sequence converges to $(Q,\xi)$. From the Lipschitz property of $a$
(Assumption \ref{assum}), we have
\begin{align*}
\int_{\mathbb{H}_{T}}\Vert a(\xi_{n}(s),y)-a(\xi(s),y)\Vert\Vert z\Vert
Q_{n}(d\mathbf{v})  &  \leq d_{{\tiny {\mbox{{lip}}}}}\int_{\mathbb{H}_{T}%
}\Vert\xi_{n}(s)-\xi(s)\Vert\Vert z\Vert Q_{n}(d\mathbf{v})\\
&  \leq d_{{\tiny {\mbox{{lip}}}}}\sup_{s\leq T}\Vert\xi_{n}(s)-\xi
(s)\Vert\int_{\mathbb{H}_{T}}\Vert z\Vert Q_{n}(d\mathbf{v})\\
&  \leq d_{{\tiny {\mbox{{lip}}}}}\sup_{s\leq T}\Vert\xi_{n}(s)-\xi
(s)\Vert\sqrt{T}\sqrt{2(M+1)}\\
&  \rightarrow0
\end{align*}
as $n\rightarrow\infty$. A similar calculation shows that as $n\rightarrow
\infty$
\[
\int_{\mathbb{H}_{T}}\Vert b(\xi_{n}(s),y)-b(\xi(s),y)\Vert Q_{n}%
(d\mathbf{v})\rightarrow0.
\]
Since $(s,y,\eta,z)\mapsto(b(\xi(s),y),a(\xi(s),y))$ is a continuous and
bounded map and $\int_{\mathbb{H}_{T}}\Vert z\Vert^{2}Q_{n}(d\mathbf{v}%
)\leq2(M+1)$, it follows from convergence of $Q_{n}$ to $Q$ that
\begin{equation}
\int_{\mathbb{H}_{T}}[b(\xi(s),y)+a(\xi(s),y)z]Q_{n}(d\mathbf{v}%
)\rightarrow\int_{\mathbb{H}_{T}}[b(\xi(s),y)+a(\xi(s),y)z]Q(d\mathbf{v}).
\label{eq:eq1139}%
\end{equation}
Combining the last three convergence statements we have (b). Next we consider
part (c). By Lemma \ref{Gen_ineq} we have that for $M_{0}\in\lbrack1,\infty)$,
$j\in\mathbb{L}$ and $t\in\lbrack0,T]$,
\begin{equation}
\int_{\mathbb{H}_{T}}\left\vert A_{yj}^{\eta}(\xi_{n}(s))-A_{yj}^{\eta}%
(\xi(s))\right\vert Q_{n}(d\mathbf{v})\leq c_{1}\left(  e^{M_{0}}\kappa
_{2}\sup_{s\leq T}\Vert\xi_{n}(s)-\xi(s)\Vert+\frac{1}{M_{0}}\int
_{\mathbb{H}_{T}}\hat{\ell}(\eta)Q_{n}(d\mathbf{v})\right)  .
\label{eqn:Abounds}%
\end{equation}
Sending $n\rightarrow\infty$ and then $M_{0}\rightarrow\infty$ we see from
\eqref{eq:eq1018} that the left side of \eqref{eqn:Abounds} converges to $0$
as $n\rightarrow\infty$. Finally, by the Skorohod representation theorem,
\eqref{eq:eq1018} and Lemma \ref{lem:genQ},
\begin{equation}
\int_{\mathbb{H}_{T}}A_{yj}^{\eta}(\xi(s))Q_{n}(d\mathbf{v})\rightarrow
\int_{\mathbb{H}_{T}}A_{yj}^{\eta}(\xi(s))Q(d\mathbf{v}). \label{E.1}%
\end{equation}
Combining this with (\ref{eqn:Abounds}) and recalling that $Q_{n}\in
\hat{\mathcal{A}}(\xi_{n})$, we have (c). This completes the proof.
\end{proof}

\section{Large Deviation Upper Bound}

\label{uppbd} The main result of this section is Theorem \ref{thm:mainuppbd},
which shows that for all $F\in C_{b}(C([0,T]:\mathbb{R}^{d}))$
\begin{equation}
\limsup_{\varepsilon\rightarrow0}\varepsilon\log\mathbb{E}\left[  \exp\left(
-\varepsilon^{-1}F(X^{\varepsilon})\right)  \right]  \leq-\inf_{\xi\in
C([0,T]:\mathbb{R}^{d})}\{F(\xi)+I(\xi)\}. \label{eq:uppbd}%
\end{equation}
To do this we show a lower bound on the corresponding variational representations.

Let $\mathcal{P}$ denote the predictable $\sigma$-field on $[0,T]\times\Omega$
associated with the filtration $\left\{  \mathcal{F}_{t}:0\leq t\leq
T\right\}  $. Let $\bar{\mathcal{A}}_{0}$ be the class of all $(\mathcal{P}%
\otimes\mathcal{B}[0,\zeta]\backslash\mathcal{B}[0,\infty))$ measurable maps
$\varphi:[0,\zeta]\times\lbrack0,T]\times\Omega\rightarrow\lbrack0,\infty)$
and let
\[
\bar{\mathcal{A}}\doteq\{\varphi=(\varphi_{ij})_{(i,j)\in\mathbb{T}}%
:\varphi_{ij}\in\bar{\mathcal{A}}_{0}\mbox{ for all }(i,j)\in\mathbb{T}\}.
\]
For $h:[0,T]\rightarrow{\mathbb{R}}^{m}$, let $\tilde{L}_{T}(h)\doteq\frac
{1}{2}\int_{0}^{T}\Vert h(s)\Vert^{2}\ ds$. Also, for $g=(g_{ij}%
)_{(i,j)\in\mathbb{T}}$ such that $g_{ij}:[0,\zeta]\times\lbrack
0,T]\rightarrow\lbrack0,\infty)$, let
\[
\bar{L}_{T}(g)\doteq\sum_{(i,j)\in\mathbb{T}}\int_{[0,\zeta]\times\lbrack
0,T]}\ell(g_{ij}(s,z))\lambda_{\zeta}(dz)ds.
\]
Consider the following spaces. Let
\begin{gather*}
P_{2}^{M}\doteq\{h:[0,T]\rightarrow{\mathbb{R}}^{m}:\tilde{L}_{T}(h)\leq
M\},\\
S^{M}\doteq\{g=(g_{ij})_{(i,j)\in\mathbb{T}}:\mbox{ for each }(i,j)\in
\mathbb{T},g_{ij}:[0,\zeta]\times\lbrack0,T]\rightarrow\lbrack0,\infty
)\mbox{ and }\bar{L}_{T}(g)\leq M\},\\
\quad P_{2}\doteq\cup_{M=1}^{\infty}P_{2}^{M},\quad S\doteq\cup_{M=1}^{\infty
}S^{M},
\end{gather*}
and
\begin{align*}
\mathcal{P}_{2}^{M}  &  \doteq\left\{  \psi:\psi\text{ is }\mathcal{P}%
\backslash\mathcal{B}(\mathbb{R}^{m})\text{ measurable and }\psi\in P_{2}%
^{M}\text{, a.s. }\mathbb{P}\right\}  ,\quad\mathcal{P}_{2}\doteq\cup
_{M=1}^{\infty}\mathcal{P}_{2}^{M},\\
\mathcal{S}^{M}  &  \doteq\left\{  \varphi\in\bar{\mathcal{A}}:\varphi
(\cdot,\omega)\in S^{M},\ \mathbb{P}\mbox{ a.s}\right\}  ,\quad\mathcal{S}%
\doteq\cup_{M=1}^{\infty}\mathcal{S}^{M}.
\end{align*}
Set $\mathcal{U}=\mathcal{P}_{2}\times\mathcal{S}$. For $f=(h,g)\in
P_{2}\times S$, let
\begin{equation}
{L}_{T}(f)=\tilde{L}_{T}(h)+\bar{L}_{T}(g). \label{eqn:sumofcosts}%
\end{equation}

With this notation the variational representation of \cite{BDM11} says that
\begin{equation}
-\varepsilon\log\mathbb{E}\left[  \exp\left(  -\varepsilon^{-1}%
F(X^{\varepsilon})\right)  \right]  =\inf_{u=(\psi,\varphi)\in\mathcal{U}%
}{\mathbb{E}}\left[  {L}_{T}(u)+F\circ\mathcal{G}^{\varepsilon}\left(
\sqrt{\varepsilon}W+\int_{0}^{\cdot}\psi(s)ds,\,\varepsilon N^{\varepsilon
^{-1}\varphi}\right)  \right]  . \label{VR}%
\end{equation}
In fact, a closer inspection of the proof of Theorem 2.8 of \cite{BDM11} (see
\cite[Theorem 2.4]{BCD13}) shows that \eqref{VR} can be strengthened as
follows. For $n\in\mathbb{N}$, define
\[
\bar{\mathcal{A}}_{b,n}=\{\varphi=(\varphi_{ij})\in\mathcal{S}:\varphi
_{ij}(r,s,\omega)\in\lbrack n^{-1},n],\mbox{ for all }(r,s,\omega)\in
\lbrack0,\zeta]\times\lbrack0,T]\times\Omega,(i,j)\in\mathbb{T}\}
\]
and let $\bar{\mathcal{A}}_{b}=\cup_{n=1}^{\infty}\bar{\mathcal{A}}_{b,n}$.
Also let $\mathcal{U}_{b}=\mathcal{P}_{2}\times\bar{\mathcal{A}}_{b}$. Then in
the equality in \eqref{VR}, $\mathcal{U}$ on the right side can be replaced by
$\mathcal{U}_{b}$.

Since $P_{2}^{M}$ is a closed ball in $L^{2}([0,T])$, it is compact under the
weak topology. A $g\in S^{M}$ can be identified with $\theta_{T}^{g}%
=(\theta_{T}^{g_{ij}})_{(i,j)\in\mathbb{T}}$ such that each $\theta
_{T}^{g_{ij}}$ is a measure on $[0,\zeta]\times\lbrack0,T]$, defined by
\[
\theta_{T}^{g_{ij}}(A)=\int_{A}g_{ij}(r,s)\lambda_{\zeta}(dr)ds,\quad
A\in\mathcal{B}([0,\zeta]\times\lbrack0,T]).
\]
With the usual weak convergence topology on the space of finite measures on
$[0,\zeta]\times\lbrack0,T]$, this identification induces a topology on
$S^{M}$ under which it is a compact space. Throughout, we use these topologies
on $P_{2}^{M}$ and $S^{M}$.

Controlled versions of processes will be denoted by an overbar, with the
particular controls used clear from context. Thus for $(\psi,\varphi
)\in\mathcal{U}_{b}$ we consider the coupled equations
\begin{equation}
\left.
\begin{aligned} \label{fastslow_ctrl} d\bar{X}^{\varepsilon}(t) &= b(\bar{X}^{\varepsilon }(t), \bar{Y}^{\varepsilon }(t)) dt + \sqrt{\varepsilon} a(\bar{X}^{\varepsilon }(t), \bar{Y}^{\varepsilon }(t)) dW(t)\\ & \hspace{0.5cm} + a(\bar{X}^{\varepsilon }(t), \bar{Y}^{\varepsilon }(t)) \psi (t) dt\\ d\bar Y^{\varepsilon}(t) &= \sum_{(i,j) \in \mathbb{T}}\int_{r \in [0,\zeta]} (j-i) 1_{\{\bar Y^{\varepsilon}(t-)=i\}} 1_{E_{ij}(\bar X^{\varepsilon}(t))}(r) N_{ij}^{\varepsilon^{-1}\varphi_{ij}}( dr \times dt) \end{aligned}\right\}
\end{equation}
with $(\bar{X}^{\varepsilon}(0),\bar{Y}^{\varepsilon}(0))=(x_{0},y_{0})$.
Recall the map $\mathcal{G}^{\varepsilon}$ introduced below \eqref{fastslow}.
From unique pathwise solvability of \eqref{fastslow_ctrl} and a standard
argument based on Girsanov's theorem (see for example \cite[Section
3.2]{BCD13}) it follows that $\mathcal{G}^{\varepsilon}\left(  \sqrt
{\varepsilon}W+\int_{0}^{\cdot}\psi^{\varepsilon}(s)ds,\,(\varepsilon
N_{ij}^{\varphi_{ij}^{\varepsilon}/\varepsilon})\right)  $ is the unique
solution of \eqref{fastslow_ctrl} with $\psi$ and $\varphi$ replaced by
$\psi^{\varepsilon}$ and $\varphi^{\varepsilon}$. Thus the representation in
\eqref{VR} yields
\begin{equation}
-\varepsilon\log\mathbb{E}\left[  \exp\left(  -\varepsilon^{-1}%
F(X^{\varepsilon})\right)  \right]  =\inf_{u=(\psi,\varphi)\in\mathcal{U}_{b}%
}{\mathbb{E}}\left[  {L}_{T}(u)+F\left(  \bar{X}^{\varepsilon}\right)
\right]  . \label{VR22}%
\end{equation}
The following lemma will be used in proving a tightness property. In many
places below we will consider controls $u$ subject to an a.s. constraint of
the form $L_{T}(u)\leq M$. To simplify the notation, the almost sure
qualification is omitted.

\begin{lemma}
\label{L2bd} For every $M\in(0,\infty)$
\[
\sup_{\varepsilon\in(0,1)}\sup_{u=(\psi,\phi)\in\mathcal{U}_{b}:L_{T}(u)\leq
M}\mathbb{E}\left(  \sup_{0\leq t\leq T}\Vert\bar{X}^{\varepsilon}(t)\Vert
^{2}\right)  <\infty.
\]

\end{lemma}

\begin{proof}
Fix $M\in(0,\infty)$ and $u=(\psi,\phi)\in\mathcal{U}_{b}$ with $L_{T}(u)\leq
M$. Now write $\bar{X}^{\varepsilon}$ as
\begin{equation}
\bar{X}^{\varepsilon}(t)=x_{0}+\bar{B}^{\varepsilon}(t)+\bar{A}^{\varepsilon
}(t)+\bar{\mathcal{E}}^{\varepsilon}(t),\label{eqn:decompX}%
\end{equation}
where
\begin{align*}
\bar{B}^{\varepsilon}(t) &  =\int_{0}^{t}b(\bar{X}^{\varepsilon}(s),\bar
{Y}^{\varepsilon}(s))ds;\\
\bar{A}^{\varepsilon}(t) &  =\int_{0}^{t}a(\bar{X}^{\varepsilon}(s),\bar
{Y}^{\varepsilon}(s))\psi^{\varepsilon}(s)ds;\\
\bar{\mathcal{E}}^{\varepsilon}(t) &  =\sqrt{\varepsilon}\int_{0}^{t}a(\bar
{X}^{\varepsilon}(s),\bar{Y}^{\varepsilon}(s))dW(s).
\end{align*}
By Remark \ref{assmp_lingrowth},
\begin{align*}
\mathbb{E}\left[  \sup_{r\leq t}\Vert\bar{A}^{\varepsilon}(r)\Vert^{2}\right]
&  \leq\mathbb{E}\int_{0}^{t}\Vert a(\bar{X}^{\varepsilon}(s),\bar
{Y}^{\varepsilon}(s))\Vert^{2}ds\int_{0}^{t}\Vert\psi(s)\Vert^{2}ds\\
&  \leq4\kappa_{1}M\int_{0}^{t}(1+\mathbb{E}\Vert\bar{X}^{\varepsilon}%
(s)\Vert^{2})ds\\
&  \leq4\kappa_{1}MT+4\kappa_{1}M\int_{0}^{t}\mathbb{E}\left(  \sup_{s\leq
v}\Vert\bar{X}^{\varepsilon}(s)\Vert^{2}\right)  dv.
\end{align*}
Also, by Doob's maximal inequality we have with $C_{1}=8\kappa_{1}^{2}$,
\begin{align}
\mathbb{E}\left(  \sup_{r\leq t}\Vert\bar{\mathcal{E}}^{\varepsilon}%
(r)\Vert^{2}\right)   &  \leq4\varepsilon\mathbb{E}\left[  \int_{0}^{t}\Vert
a(\bar{X}^{\varepsilon}(s),\bar{Y}^{\varepsilon}(s))\Vert^{2}ds\right]
\nonumber\\
&  \leq8\kappa_{1}^{2}\varepsilon\int_{0}^{t}\left(  1+\mathbb{E}\left[
\Vert\bar{X}^{\varepsilon}(s)\Vert^{2}\right]  \right)  ds\label{Error_term}\\
&  \leq C_{1}\varepsilon T+C_{1}\varepsilon\int_{0}^{t}\mathbb{E}\left(
\sup_{s\leq v}\Vert\bar{X}^{\varepsilon}(s)\Vert^{2}\right)  dv.\nonumber
\end{align}
Similar calculation shows that
\[
\mathbb{E}\left[  \sup_{r\leq t}\Vert\bar{B}^{\varepsilon}(r)\Vert^{2}\right]
\leq2\kappa_{1}T+2\kappa_{1}T\int_{0}^{t}\mathbb{E}\left(  \sup_{s\leq v}%
\Vert\bar{X}^{\varepsilon}(s)\Vert^{2}\right)  .
\]
Hence from (\ref{eqn:decompX}) it follows that with $C_{2}=C_{1}%
(T+1)+2\kappa_{1}(T+1)(1+2M)$
\[
\mathbb{E}\left[  \sup_{r\leq t}\Vert\bar{X}^{\varepsilon}(r)\Vert^{2}\right]
\leq C_{2}+C_{2}\int_{0}^{t}\mathbb{E}\left(  \sup_{s\leq v}\Vert\bar
{X}^{\varepsilon}(s)\Vert^{2}\right)  dv.
\]
The lemma now follows from Gronwall's inequality.
\end{proof}

\begin{proposition}
\label{ctrlproc_tight} Fix $M\in(0,\infty)$. For $\varepsilon\in(0,1)$ let
$u^{\varepsilon}=(\psi^{\varepsilon},\varphi^{\varepsilon})\in\mathcal{U}_{b}$
be such that $L_{T}(u^{\varepsilon})\leq M$. Let $(\bar{X}^{\varepsilon}%
,\bar{Y}^{\varepsilon})$ solve \eqref{fastslow_ctrl} with $\psi$ and $\varphi$
replaced by $\psi^{\varepsilon}$ and $\varphi^{\varepsilon}$. Then $\{\bar
{X}^{\varepsilon}\}_{\varepsilon\in(0,1)}$ is tight in $C([0,T]:\mathbb{R}%
^{d})$.
\end{proposition}

\begin{proof}
Write $\bar{X}^{\varepsilon}$ as in (\ref{eqn:decompX}). From
\eqref{Error_term} and Lemma \ref{L2bd}
\[
\mathbb{E}\sup_{t\in\lbrack0,T]}\Vert\bar{\mathcal{E}}^{\varepsilon}%
(t)\Vert^{2}\leq C_{1}\varepsilon T\left(  1+\sup_{\varepsilon\in
(0,1)}\mathbb{E}\sup_{t\in\lbrack0,T]}\Vert\bar{X}^{\varepsilon}(t)\Vert
^{2}\right)  \rightarrow0,
\]
as $\varepsilon\rightarrow0$. It thus suffices to prove the tightness of
$\{\bar{A}^{\varepsilon}\}$ and $\{\bar{B}^{\varepsilon}\}$. For this, note
that for $0\leq h\leq T$ and $\varepsilon\in(0,1)$
\begin{align*}
&  \mathbb{E}\sup_{0\leq t_{1}\leq t_{2}\leq T,t_{2}\leq t_{1}+h}\Vert\bar
{A}^{\varepsilon}(t_{2})-\bar{A}^{\varepsilon}(t_{1})\Vert^{2}\\
&  \quad=\mathbb{E}\sup_{0\leq t_{1}\leq t_{2}\leq T,t_{2}\leq t_{1}%
+h}\left\Vert \int_{t_{1}}^{t_{2}}a(\bar{X}^{\varepsilon}(s),\bar
{Y}^{\varepsilon}(s))\psi^{\varepsilon}(s)ds\right\Vert ^{2}\\
&  \quad\leq2\kappa_{1}^{2}\mathbb{E}\sup_{0\leq t_{1}\leq T-h}\left(
\int_{t_{1}}^{t_{1}+h}(1+\Vert\bar{X}^{\varepsilon}(s)\Vert^{2})\ ds\int
_{0}^{T}\Vert\psi^{\varepsilon}(s)\Vert^{2}ds\right) \\
&  \quad\leq4\kappa_{1}^{2}Mh\left(  1+\sup_{\varepsilon\in(0,1)}%
\mathbb{E}\sup_{0\leq t\leq T}\Vert\bar{X}^{\varepsilon}(t)\Vert^{2}\right)  ,
\end{align*}
where the last inequality used the fact that $L_{T}(u^{\varepsilon})\leq M$.
Lemma \ref{L2bd} now gives that $\{\bar{A}^{\varepsilon}\}$ is tight in
$C([0,T]:\mathbb{R}^{d})$. Similar calculations show that $\{\bar
{B}^{\varepsilon}\}$ is tight as well. The result follows.
\end{proof}

\begin{remark}
\label{rem:rem530}\textrm{\ The estimate in Proposition \ref{ctrlproc_tight}
in particular shows that for every $M\in(0,\infty)$ there exists some
$c_{2}(M)\in(0,\infty),$ such that%
\[
\sup_{\varepsilon\in(0,1)}\sup_{0\leq s\leq T-\Delta}\;\sup_{u^{\varepsilon
}\in\mathcal{U}_{b}:L_{T}(u^{\varepsilon})\leq M}\mathbb{E}\sup_{0\leq
t\leq\Delta}\Vert\bar{X}^{\varepsilon}(s+t)-\bar{X}^{\varepsilon}(s)\Vert
^{2}\leq c_{2}(M)\Delta
\]
for any $\Delta\in\lbrack0,T]$. }
\end{remark}

Given $\varepsilon>0$, let $u^{\varepsilon}=(\psi^{\varepsilon},\varphi
^{\varepsilon})\in\mathcal{U}_{b}$ and $(\bar{X}^{\varepsilon},\bar
{Y}^{\varepsilon})$ be as in Proposition \ref{ctrlproc_tight}. Note that by
(\ref{fastslow_ctrl}) only the controlled rates $\varphi_{\bar{Y}%
^{\varepsilon}(t-),\cdot}^{\varepsilon}$ affect the evolution of $(\bar
{X}^{\varepsilon},\bar{Y}^{\varepsilon})$. In proving the Laplace upper bound,
we can (and will) assume without loss of generality that
\begin{equation}
\varphi_{ij}^{\varepsilon}(t,z)=1\mbox{ for all }(i,j)\in\mathbb{T}%
\mbox{ such that }i\in\mathbb{L}\setminus\{Y^{\varepsilon}(t-)\}\mbox{
and }(t,z)\in\lbrack0,T]\times\lbrack0,\zeta]. \label{eq:eq752-17}%
\end{equation}

The proof of the Laplace upper bound relies on the asymptotic analysis of the
following occupation measure. We fix a collection $\{\Delta_{\varepsilon
}\}_{\varepsilon>0}$ of positive reals such that
\begin{equation}
\Delta_{\varepsilon}\rightarrow0,\;\frac{\Delta_{\varepsilon}}{\varepsilon
}\rightarrow\infty,\;\varepsilon\rightarrow0. \label{eq:eq501}%
\end{equation}
By convention we will take $\varphi_{ij}^{\varepsilon}(t,r)=1$ if
$(i,j)\notin\mathbb{T}$. Also by convention we set $\varphi_{ij}^{\varepsilon
}(u,r)=1$ for all $(i,j)$, $\psi^{\varepsilon}(u)=0$ and $\bar{Y}%
^{\varepsilon}(u)=\bar{Y}^{\varepsilon}(T)$ if $u\in\lbrack T,T+\Delta
_{\varepsilon}]$. For $t\in\lbrack0,T]$, $\eta^{\varepsilon}(t)\in
(\mathcal{M}_{F}[0,\zeta])^{|\mathbb{L}|}$ is defined to be $\eta
^{\varepsilon}(t)=(\eta_{j}^{\varepsilon}(t))_{j\in\mathbb{L}}$, where
\begin{equation}
\eta_{j}^{\varepsilon}(t)(F)=\sum_{i\in\mathbb{L}}1_{\{Y^{\varepsilon
}(t-)=i\}}\int_{F}\varphi_{ij}^{\varepsilon}(t,r)\lambda_{\zeta}(dr),\quad
F\in\mathcal{B}[0,\zeta]. \label{eq:eq1347}%
\end{equation}
Then define $Q^{\varepsilon}\in\mathcal{P}_{1}(\mathbb{H}_{T})$ by
\begin{equation}
Q^{\varepsilon}(A\times B\times C\times D)\doteq\int_{\lbrack0,T]}%
1_{A}(s)\left(  \frac{1}{\Delta_{\varepsilon}}\int_{s}^{s+\Delta_{\varepsilon
}}1_{B}(\bar{Y}^{\varepsilon}(u))1_{C}(\eta^{\varepsilon}(u))1_{D}%
(\psi^{\varepsilon}(u))du\right)  ds. \label{eqn:defofQ}%
\end{equation}
The main step in the proof will be to characterize the limit points of
$Q^{\varepsilon}$. We begin with some preliminary estimates.

\begin{lemma}
\label{lem:diffest} Fix $M\in(0,\infty)$. For $\varepsilon\in(0,1)$, let
$u^{\varepsilon}=(\psi^{\varepsilon},\varphi^{\varepsilon})\in\mathcal{U}_{b}$
satisfy $L_{T}(u^{\varepsilon})\leq M$ and let $(\bar{X}^{\varepsilon},\bar
{Y}^{\varepsilon})$, be as in Proposition \ref{ctrlproc_tight}. Let
$\Delta_{\varepsilon}$ be as in \eqref{eq:eq501}. Then
\begin{equation}
\sup_{0\leq t\leq T}\mathbb{E}\left\Vert \int_{0}^{t}b(\bar{X}^{\varepsilon
}(s),\bar{Y}^{\varepsilon}(s))ds-\int_{0}^{t}\frac{1}{\Delta_{\varepsilon}%
}\int_{s}^{(s+\Delta_{\varepsilon})\wedge T}b(\bar{X}^{\varepsilon}(s),\bar
{Y}^{\varepsilon}(u))duds\right\Vert ^{2}\rightarrow0, \label{sup1}%
\end{equation}
and
\begin{equation}
\sup_{0\leq t\leq T}\mathbb{E}\left\Vert \int_{0}^{t}a(\bar{X}^{\varepsilon
}(s),\bar{Y}^{\varepsilon}(s))\psi^{\varepsilon}(s)ds-\int_{0}^{t}\frac
{1}{\Delta_{\varepsilon}}\int_{s}^{(s+\Delta_{\varepsilon})\wedge T}a(\bar
{X}^{\varepsilon}(s),\bar{Y}^{\varepsilon}(u))\psi^{\varepsilon}%
(u)duds\right\Vert ^{2}\rightarrow0, \label{sup2}%
\end{equation}
as $\varepsilon\rightarrow0$.
\end{lemma}

\begin{proof}
We only prove \eqref{sup2}. The proof of \eqref{sup1} is similar but easier
and therefore omitted. Changing the order of the integration, for $t\geq
\Delta_{\varepsilon}$,
\begin{align*}
\int_{0}^{t}\frac{1}{\Delta_{\varepsilon}}\int_{s}^{(s+\Delta_{\varepsilon
})\wedge T}a(\bar{X}^{\varepsilon}(s),\bar{Y}^{\varepsilon}(u))\psi
^{\varepsilon}(u)duds= &  \int_{0}^{\Delta_{\varepsilon}}\frac{1}%
{\Delta_{\varepsilon}}\int_{0}^{u}a(\bar{X}^{\varepsilon}(s),\bar
{Y}^{\varepsilon}(u))\psi^{\varepsilon}(u)dsdu\\
\quad\quad &  +\int_{\Delta_{\varepsilon}}^{t}\frac{1}{\Delta_{\varepsilon}%
}\int_{u-\Delta_{\varepsilon}}^{u}a(\bar{X}^{\varepsilon}(s),\bar
{Y}^{\varepsilon}(u))\psi^{\varepsilon}(u)dsdu\\
\quad\quad &  +\int_{t}^{(t+\Delta_{\varepsilon})\wedge T}\frac{1}%
{\Delta_{\varepsilon}}\int_{u-\Delta_{\varepsilon}}^{t}a(\bar{X}^{\varepsilon
}(s),\bar{Y}^{\varepsilon}(u))\psi^{\varepsilon}(u)dsdu.
\end{align*}
Also for $t\geq\Delta_{\varepsilon}$,
\[
\int_{0}^{t}a(\bar{X}^{\varepsilon}(s),\bar{Y}^{\varepsilon}(s))\psi
^{\varepsilon}(s)ds=\int_{0}^{t}\frac{1}{\Delta_{\varepsilon}}\int
_{u-\Delta_{\varepsilon}}^{u}a(\bar{X}^{\varepsilon}(u),\bar{Y}^{\varepsilon
}(u))\psi^{\varepsilon}(u)dsdu.
\]
Thus for such $t$,
\begin{align}
&  \left\Vert \int_{0}^{t}a(\bar{X}^{\varepsilon}(s),\bar{Y}^{\varepsilon
}(s))\psi^{\varepsilon}(s)ds-\int_{0}^{t}\frac{1}{\Delta_{\varepsilon}}%
\int_{s}^{(s+\Delta_{\varepsilon})\wedge T}a(\bar{X}^{\varepsilon}(s),\bar
{Y}^{\varepsilon}(u))\psi^{\varepsilon}(u)duds\right\Vert \nonumber\\
&  \quad\leq\left\Vert \int_{0}^{\Delta_{\varepsilon}}\frac{1}{\Delta
_{\varepsilon}}\int_{u-\Delta_{\varepsilon}}^{u}a(\bar{X}^{\varepsilon
}(u),\bar{Y}^{\varepsilon}(u))\psi^{\varepsilon}(u)dsdu-\int_{0}%
^{\Delta_{\varepsilon}}\frac{1}{\Delta_{\varepsilon}}\int_{0}^{u}a(\bar
{X}^{\varepsilon}(s),\bar{Y}^{\varepsilon}(u))\psi^{\varepsilon}%
(u)dsdu\right\Vert \nonumber\\
&  \quad\quad+\left\Vert \int_{\Delta_{\varepsilon}}^{t}\frac{1}%
{\Delta_{\varepsilon}}\int_{u-\Delta_{\varepsilon}}^{u}a(\bar{X}^{\varepsilon
}(u),\bar{Y}^{\varepsilon}(u))\psi^{\varepsilon}(u)dsdu-\int_{\Delta
_{\varepsilon}}^{t}\frac{1}{\Delta_{\varepsilon}}\int_{u-\Delta_{\varepsilon}%
}^{u}a(\bar{X}^{\varepsilon}(s),\bar{Y}^{\varepsilon}(u))\psi^{\varepsilon
}(u)dsdu\right\Vert \nonumber\\
&  \quad\quad+\left\Vert \int_{t}^{(t+\Delta_{\varepsilon})\wedge T}\frac
{1}{\Delta_{\varepsilon}}\int_{u-\Delta_{\varepsilon}}^{t}a(\bar
{X}^{\varepsilon}(s),\bar{Y}^{\varepsilon}(u))\psi^{\varepsilon}%
(u)dsdu\right\Vert \nonumber\\
&  \quad=T_{t}^{(1)}+T_{t}^{(2)}+T_{t}^{(3)}.\label{eq:eq259_17}%
\end{align}
For $T^{(1)}$ we have
\begin{align}
(T_{t}^{(1)})^{2} &  \leq2\left\Vert \int_{0}^{\Delta_{\varepsilon}}\frac
{1}{\Delta_{\varepsilon}}\int_{u-\Delta_{\varepsilon}}^{u}a(\bar
{X}^{\varepsilon}(u),\bar{Y}^{\varepsilon}(u))\psi^{\varepsilon}%
(u)dsdu\right\Vert ^{2}\nonumber\\
&  \quad+2\left\Vert \int_{0}^{\Delta_{\varepsilon}}\frac{1}{\Delta
_{\varepsilon}}\int_{0}^{u}a(\bar{X}^{\varepsilon}(s),\bar{Y}^{\varepsilon
}(u))\psi^{\varepsilon}(u)dsdu\right\Vert ^{2}\nonumber\\
&  \leq4\left(  \kappa_{1}\sup_{0\leq s\leq T}(\Vert\bar{X}^{\varepsilon
}(s)\Vert+1)\right)  ^{2}\left(  \int_{0}^{\Delta_{\varepsilon}}\Vert
\psi^{\varepsilon}(u)\Vert du\right)  ^{2}\nonumber\\
&  \leq16\kappa_{1}^{2}\left(  \sup_{0\leq s\leq T}\Vert\bar{X}^{\varepsilon
}(s)\Vert^{2}+1\right)  \cdot\Delta_{\varepsilon}M.\label{eq:eq529}%
\end{align}
Similarly,
\begin{equation}
(T_{t}^{(3)})^{2}\leq8\kappa_{1}^{2}\left(  \sup_{0\leq s\leq t}\Vert\bar
{X}^{\varepsilon}(s)\Vert^{2}+1\right)  \cdot\Delta_{\varepsilon
}M.\label{eq:eq530}%
\end{equation}
Thus in view of Lemma \ref{L2bd}, $\sup_{0\leq t\leq T}\mathbb{E}(T_{t}%
^{(i)})^{2}\rightarrow0$, as $\varepsilon\rightarrow0$, for $i=1,3$. For the
second term, using the Lipschitz property of $a$ we have%
\begin{align*}
(T_{t}^{(2)})^{2} &  \leq\int_{\Delta_{\varepsilon}}^{t}\int_{u-\Delta
_{\varepsilon}}^{u}\frac{1}{\Delta_{\varepsilon}^{2}}d_{{\tiny {\mbox{{lip}}}%
}}^{2}\Vert\bar{X}^{\varepsilon}(u)-\bar{X}^{\varepsilon}(s)\Vert^{2}%
dsdu\cdot\int_{\Delta_{\varepsilon}}^{t}\int_{u-\Delta_{\varepsilon}}%
^{u}\left\Vert \psi^{\varepsilon}(u)\right\Vert ^{2}dsdu\\
&  \leq\int_{\Delta_{\varepsilon}}^{t}\int_{u-\Delta_{\varepsilon}}^{u}%
\frac{1}{\Delta_{\varepsilon}^{2}}d_{{\tiny {\mbox{{lip}}}}}^{2}\Vert\bar
{X}^{\varepsilon}(u)-\bar{X}^{\varepsilon}(s)\Vert^{2}dsdu\cdot\Delta
_{\varepsilon}2M,
\end{align*}
where $d_{{\tiny {\mbox{{lip}}}}}$ is as in Assumption \ref{assum}. Using
Remark \ref{rem:rem530} we have
\[
\sup_{\Delta_{\varepsilon}\leq t\leq T}\mathbb{E}(T_{t}^{(2)})^{2}\leq
\frac{d_{{\tiny {\mbox{{lip}}}}}^{2}}{\Delta_{\varepsilon}^{2}}T\Delta
_{\varepsilon}c_{2}(M)\Delta_{\varepsilon}^{2}2M\leq2d_{{\tiny {\mbox{{lip}}}%
}}^{2}c_{2}(M)T\Delta_{\varepsilon}M,
\]
which converges to 0 as $\varepsilon\rightarrow0$. Thus we have shown that
\eqref{sup2} holds with the $\sup_{0\leq t\leq T}$ on the left replaced by
$\sup_{\Delta_{\varepsilon}\leq t\leq T}$. A similar calculation can be used
to prove the statement for $\sup_{0\leq t\leq\Delta_{\varepsilon}}%
\mathbb{E}(T_{t}^{(2)})^{2}$. The result follows.
\end{proof}

Next, with $u^{\varepsilon}=(\psi^{\varepsilon},\varphi^{\varepsilon}%
)\in\mathcal{U}_{b}$ and $\bar{X}^{\varepsilon}$ as in Proposition
\ref{ctrlproc_tight}, define $C([0,T]:\mathbb{R}^{d})$-valued random variables
$\bar{Z}^{\varepsilon}$ by
\begin{equation}
\bar{Z}^{\varepsilon}(t)\doteq\bar{X}^{\varepsilon}(t)-x_{0}-\int
_{\mathbb{H}_{t}}b(\bar{X}^{\varepsilon}(s),y)Q^{\varepsilon}(d\mathbf{v}%
)-\int_{\mathbb{H}_{t}}a(\bar{X}^{\varepsilon}(s),y)zQ^{\varepsilon
}(d\mathbf{v}),\;t\in\lbrack0,T]. \label{eq:qe}%
\end{equation}
The following lemma shows that $\bar{Z}^{\varepsilon}(t)\rightarrow0$, for all
$t\in[0,T]$, as $\varepsilon\rightarrow0$.

\begin{lemma}
\label{lem:lem748} Fix $M\in(0,\infty)$ and for $\varepsilon\in(0,1)$ let
$u^{\varepsilon}=(\psi^{\varepsilon},\varphi^{\varepsilon})$ be as in
Proposition \ref{ctrlproc_tight}. Also, let $\bar{Z}^{\varepsilon}$ be as in
\eqref{eq:qe}. Then
\[
\sup_{0\leq t\leq T}{\mathbb{E}}\Vert\bar{Z}^{\varepsilon}(t)\Vert
^{2}\rightarrow0,
\]
as $\varepsilon\rightarrow0$.
\end{lemma}

\begin{proof}
From (\ref{eqn:decompX}), $\bar{Z}^{\varepsilon}(t)$ can be written as
\begin{align*}
\bar{Z}^{\varepsilon}(t)  &  =\bar{\mathcal{E}}^{\varepsilon}(t)+\int_{0}%
^{t}b(\bar{X}^{\varepsilon}(s),\bar{Y}^{\varepsilon}(s))ds-\int_{\mathbb{H}%
_{t}}b(\bar{X}^{\varepsilon}(s),y)Q^{\varepsilon}(d\mathbf{v})\\
&  \hspace*{0.1cm}\quad+\int_{0}^{t}a(\bar{X}^{\varepsilon}(s),\bar
{Y}^{\varepsilon}(s))\psi^{\varepsilon}(s)ds-\int_{\mathbb{H}_{t}}a(\bar
{X}^{\varepsilon}(s),y)zQ^{\varepsilon}(d\mathbf{v})\\
&  =\bar{\mathcal{E}}^{\varepsilon}(t)+\int_{0}^{t}b(\bar{X}^{\varepsilon
}(s),\bar{Y}^{\varepsilon}(s))ds-\int_{0}^{t}\frac{1}{\Delta_{\varepsilon}%
}\int_{s}^{(s+\Delta_{\varepsilon})\wedge T}b(\bar{X}^{\varepsilon}(s),\bar
{Y}^{\varepsilon}(u))duds\\
&  \quad+\hspace*{0.1cm}\int_{0}^{t}a(\bar{X}^{\varepsilon}(s),\bar
{Y}^{\varepsilon}(s))\psi^{\varepsilon}(s)ds-\int_{0}^{t}\frac{1}%
{\Delta_{\varepsilon}}\int_{s}^{(s+\Delta_{\varepsilon})\wedge T}a(\bar
{X}^{\varepsilon}(s),\bar{Y}^{\varepsilon}(u))\psi^{\varepsilon}(u)duds,
\end{align*}
where $\bar{\mathcal{E}}^{\varepsilon}$ is as in the proof of Lemma
\ref{L2bd}. As argued in the proof of Proposition \ref{ctrlproc_tight},
$\mathbb{E}\left(  \sup_{r\leq t}\Vert\bar{\mathcal{E}}^{\varepsilon}%
(r)\Vert^{2}\right)  \rightarrow0$. The result now follows from Lemma
\ref{lem:diffest}.
\end{proof}

We now give a convenient lower bound for ${L}_{T}(u^{\varepsilon})$ for
$u^{\varepsilon} \in\mathcal{U}_{b}$.

\begin{lemma}
\label{Lem:ineqrate} For $\varepsilon\in(0,1)$, let $u^{\varepsilon}%
=(\psi^{\varepsilon},\varphi^{\varepsilon})\in\mathcal{U}_{b}$ and define
$Q^{\varepsilon}$ as in (\ref{eqn:defofQ}). Then
\[
{L}_{T}(u^{\varepsilon})\geq\int_{\mathbb{H}_{T}}\left[  \frac{1}{2}\Vert
z\Vert^{2}+\hat{\ell}(\eta)\right]  Q^{\varepsilon}(d\mathbf{v}).
\]

\end{lemma}

\begin{proof}
Recall from \eqref{eqn:sumofcosts} that
\[
{L}_{T}(u^{\varepsilon})=\frac{1}{2}\int_{0}^{T}\Vert\psi^{\varepsilon
}(s)\Vert^{2}ds+\sum_{(i,j)\in\mathbb{T}}\int_{[0,T]\times\lbrack0,\zeta]}%
\ell(\varphi_{ij}^{\varepsilon}(r,s))\lambda_{\zeta}(dr)ds,
\]
and using the convention for the definition of $\psi^{\varepsilon}(u)$ and
$\eta^{\varepsilon}(u)$ when $u>T$,
\[
\int_{\mathbb{H}_{T}}\Vert z\Vert^{2}Q^{\varepsilon}(d\mathbf{v})=\int_{0}%
^{T}\frac{1}{\Delta_{\varepsilon}}\int_{s}^{s+\Delta_{\varepsilon}}\Vert
\psi^{\varepsilon}(u)\Vert^{2}duds,\;\int_{\mathbb{H}_{T}}\hat{\ell}%
(\eta)Q^{\varepsilon}(d\mathbf{v})=\int_{0}^{T}\frac{1}{\Delta_{\varepsilon}%
}\int_{s}^{s+\Delta_{\varepsilon}}\hat{\ell}(\eta^{\varepsilon}(u))duds,
\]
where $\eta^{\varepsilon}$ is as in \eqref{eq:eq1347}. As in the proof of
\eqref{sup2}, changing the order of integration, we can rewrite the first term
in the last display as
\[
\int_{0}^{\Delta_{\varepsilon}}\frac{1}{\Delta_{\varepsilon}}\int_{0}^{u}%
\Vert\psi^{\varepsilon}(u)\Vert^{2}dsdu+\int_{\Delta_{\varepsilon}}^{T}%
\frac{1}{\Delta_{\varepsilon}}\int_{u-\Delta_{\varepsilon}}^{u}\Vert
\psi^{\varepsilon}(u)\Vert^{2}dsdu+\int_{T}^{T+\Delta_{\varepsilon}}\frac
{1}{\Delta_{\varepsilon}}\int_{u-\Delta_{\varepsilon}}^{T}\Vert\psi
^{\varepsilon}(u)\Vert^{2}dsdu,
\]
where the third term is 0 using our convention that $\psi^{\varepsilon}(u)=0$,
for $u>T$. We can write
\[
\int_{0}^{T}\Vert\psi^{\varepsilon}(s)\Vert^{2}ds=\int_{0}^{T}\frac{1}%
{\Delta_{\varepsilon}}\int_{u-\Delta_{\varepsilon}}^{u}\Vert\psi^{\varepsilon
}(u)\Vert^{2}dsdu,
\]
and thus
\begin{align*}
&  \int_{0}^{T}\Vert\psi^{\varepsilon}(s)\Vert^{2}ds-\int_{\mathbb{H}_{T}%
}\Vert z\Vert^{2}Q^{\varepsilon}(d\mathbf{v})\\
&  \quad=\int_{0}^{\Delta_{\varepsilon}}\frac{1}{\Delta_{\varepsilon}}\int
_{0}^{\Delta_{\varepsilon}}\Vert\psi^{\varepsilon}(u)\Vert^{2}dsdu-\int
_{0}^{\Delta_{\varepsilon}}\frac{1}{\Delta_{\varepsilon}}\int_{0}^{u}\Vert
\psi^{\varepsilon}(u)\Vert^{2}dsdu\\
&  \quad=\int_{0}^{\Delta_{\varepsilon}}\frac{1}{\Delta_{\varepsilon}}\int
_{u}^{\Delta_{\varepsilon}}\Vert\psi^{\varepsilon}(u)\Vert^{2}dsdu\geq0.
\end{align*}
For the second term note from \eqref{eq:eq752-17} and \eqref{eq:eq1347} that
\begin{align*}
\sum_{(i,j)\in\mathbb{T}}\int_{[0,T]\times\lbrack0,\zeta]}\ell(\varphi
_{ij}^{\varepsilon}(r,s))\lambda_{\zeta}(dr)ds  &  =\sum_{i\in\mathbb{L}}%
\int_{[0,T]\times\lbrack0,\zeta]}1_{\{Y^{\varepsilon}(s)=i\}}\ell(\varphi
_{ij}^{\varepsilon}(r,s))\lambda_{\zeta}(dr)ds\\
&  =\int_{[0,T]}\hat{\ell}(\eta^{\varepsilon}(s))ds,
\end{align*}
and we have in a similar manner that
\[
\int_{\lbrack0,T]}\hat{\ell}(\eta^{\varepsilon}(s))ds-\int_{\mathbb{H}_{T}%
}\hat{\ell}(\eta)Q^{\varepsilon}(d\mathbf{v})=\int_{0}^{\Delta_{\varepsilon}%
}\frac{1}{\Delta_{\varepsilon}}\int_{u}^{\Delta_{\varepsilon}}\hat{\ell}%
(\eta^{\varepsilon}(u))dsdu\geq0.
\]
The result follows.
\end{proof}

We now prove the tightness of $\{(\bar X^{\varepsilon}, Q^{\varepsilon})\}$
and characterize the limit points.

\begin{proposition}
\label{ctrlmeas_tight} Fix $M \in(0,\infty)$. Define $Q^{\varepsilon}$ as in
\eqref{eqn:defofQ}. Suppose for $\varepsilon\in(0,1)$, $u^{\varepsilon}%
=(\psi^{\varepsilon}, \varphi^{\varepsilon}) \in\mathcal{U}_{b}$ is such that
$L_{T}(u^{\varepsilon}) \le M$. Then $\{(\bar X^{\varepsilon}, Q^{\varepsilon
})\}$ is a tight family of $C([0,T]:\mathbb{R}^{d})\times\mathcal{M}%
_{F}(\mathbb{H}_{T})$-valued random variables. Furthermore, if $(\xi, Q)$ is a
weak limit point of $\{(\bar{X}^{\varepsilon},Q^{\varepsilon})\}$, then

\begin{enumerate}
\item $M\geq\int_{\mathbb{H}_{T}}\left[  \frac{1}{2}\Vert z\Vert^{2}+\hat
{\ell}(\eta)\right]  Q(d\mathbf{v});$

\item Equations \eqref{xi} and \eqref{inv} hold a.s.
\end{enumerate}
\end{proposition}

\begin{proof}
Tightness of $\{\bar{X}^{\varepsilon}\}$ was shown in Proposition
\ref{ctrlproc_tight}. Next we argue the tightness of $\{Q^{\varepsilon}\}$.
From Lemma \ref{Lem:ineqrate} we have
\begin{equation}
\int_{\mathbb{H}_{T}}\left[  \frac{1}{2}\Vert z\Vert^{2}+\hat{\ell}%
(\eta)\right]  Q^{\varepsilon}(d\mathbf{v})\leq M. \label{eq:eq737}%
\end{equation}
To prove the tightness of $\{Q^{\varepsilon}\}$, it suffices to show that for
any $\delta\in(0,\infty)$, there exists $C_{1}\in(0,\infty)$ such that:
\[
\sup_{\varepsilon}\mathbb{E}Q^{\varepsilon}\left\{  (s,y,\eta,z)\in
\mathbb{H}_{T}:\;\sum_{j\in\mathbb{L}}\eta_{j}[0,\zeta]+\Vert z\Vert
>C_{1}\right\}  \leq\delta.
\]
However, this is proved exactly as \eqref{eq:ratetight} using \eqref{eq:eq737}
instead of \eqref{bd1}.
The inequality in part 1 follows immediately from Lemma \ref{Lem:ineqrate}
using Fatou's Lemma and lower semicontinuity of $\eta\mapsto\hat{\ell}(\eta)$.
We now prove 2. For this we assume without loss of generality (using the
Skorokhod representation) that $(\bar{X}^{\varepsilon},Q^{\varepsilon})$
converges a.s. to $(\xi,Q)$. Following similar steps as in the proof of
Proposition \ref{prop611} (see the proof of \eqref{eq:eq1139}), we conclude
that
\[
\int_{\mathbb{H}_{T}}[b(\bar{X}^{\varepsilon}(s),y)+a(\bar{X}^{\varepsilon
}(s),y)z]Q^{\varepsilon}(d\mathbf{v})\rightarrow\int_{\mathbb{H}_{T}}%
[b(\xi(s),y)+a(\xi(s),y)z]Q(d\mathbf{v}).
\]
It now follows from \eqref{eq:qe} and Lemma \ref{lem:lem748} that \eqref{xi}
holds. To prove that \eqref{inv} holds, we estimate the difference between
$\int_{\mathbb{H}_{t}}A_{y,j}^{\eta}(\bar{X}^{\varepsilon}(s))Q^{\varepsilon
}(d\mathbf{v})$ and $\int_{0}^{t}A_{\bar{Y}^{\varepsilon}(u),j}^{\eta
^{\varepsilon}(u)}(\bar{X}^{\varepsilon}(u))du$ for $j\in\mathbb{L}$ and
$t\in\lbrack0,T]$. By a change of the order of integration, we have for
$t\in\lbrack\Delta_{\varepsilon},T]$
\begin{align*}
\int_{\mathbb{H}_{t}}A_{y,j}^{\eta}(\bar{X}^{\varepsilon}(s))Q^{\varepsilon
}(d\mathbf{v})  &  =\int_{0}^{t}\frac{1}{\Delta_{\varepsilon}}\int
_{s}^{s+\Delta_{\varepsilon}}A_{\bar{Y}^{\varepsilon}(u),j}^{\eta
^{\varepsilon}(u)}(\bar{X}^{\varepsilon}(s))duds\\
&  =\int_{0}^{\Delta_{\varepsilon}}\frac{1}{\Delta_{\varepsilon}}\int_{0}%
^{u}A_{\bar{Y}^{\varepsilon}(u),j}^{\eta^{\varepsilon}(u)}(\bar{X}%
^{\varepsilon}(s))dsdu\\
&  \quad+\int_{\Delta_{\varepsilon}}^{t}\frac{1}{\Delta_{\varepsilon}}%
\int_{u-\Delta_{\varepsilon}}^{u}A_{\bar{Y}^{\varepsilon}(u),j}^{\eta
^{\varepsilon}(u)}(\bar{X}^{\varepsilon}(s))dsdu\\
&  \quad+\int_{t}^{t+\Delta_{\varepsilon}}\frac{1}{\Delta_{\varepsilon}}%
\int_{u-\Delta_{\varepsilon}}^{t}A_{\bar{Y}^{\varepsilon}(u),j}^{\eta
^{\varepsilon}(u)}(\bar{X}^{\varepsilon}(s))dsdu.
\end{align*}
Write%
\[
\int_{0}^{t}A_{\bar{Y}^{\varepsilon}(u),j}^{\eta^{\varepsilon}(u)}(\bar
{X}^{\varepsilon}(u))du=\int_{0}^{t}\frac{1}{(\Delta_{\varepsilon})\wedge
u}\int_{(u-\Delta_{\varepsilon})_{+}}^{u}A_{\bar{Y}^{\varepsilon}(u),j}%
^{\eta^{\varepsilon}(u)}(\bar{X}^{\varepsilon}(u))dsdu.
\]
Using this along with Lemma \ref{Gen_ineq} and Remark \ref{rem:rem1017} we
have that there exists $C_{2}\in(0,\infty)$ such that, for any $M_{0}%
\in\lbrack1,\infty)$, $t\in\lbrack\Delta_{\varepsilon},T]$ and $\varepsilon
\in(0,1)$,
\begin{align*}
&  \left\vert \int_{\mathbb{H}_{t}}A_{y,j}^{\eta}(\bar{X}^{\varepsilon
}(s))Q^{\varepsilon}(d\mathbf{v})-\int_{0}^{t}A_{\bar{Y}^{\varepsilon}%
(u),j}^{\eta^{\varepsilon}(u)}(\bar{X}^{\varepsilon}(u))du\right\vert \\
&  \quad\leq\left\vert \int_{0}^{\Delta_{\varepsilon}}\frac{1}{\Delta
_{\varepsilon}}\int_{0}^{u}A_{\bar{Y}^{\varepsilon}(u),j}^{\eta^{\varepsilon
}(u)}(\bar{X}^{\varepsilon}(s))dsdu\right\vert +\left\vert \int_{0}%
^{\Delta_{\varepsilon}}\frac{1}{u}\int_{0}^{u}A_{\bar{Y}^{\varepsilon}%
(u),j}^{\eta^{\varepsilon}(u)}(\bar{X}^{\varepsilon}(u))dsdu\right\vert \\
&  \quad\quad+\int_{\Delta_{\varepsilon}}^{t}\frac{1}{\Delta_{\varepsilon}%
}\int_{u-\Delta_{\varepsilon}}^{u}\left\vert A_{\bar{Y}^{\varepsilon}%
(u),j}^{\eta^{\varepsilon}(u)}(\bar{X}^{\varepsilon}(s))-A_{\bar
{Y}^{\varepsilon}(u),j}^{\eta^{\varepsilon}(u)}(\bar{X}^{\varepsilon
}(u))\right\vert dsdu\\
&  \quad\quad+\left\vert \int_{t}^{t+\Delta_{\varepsilon}}\frac{1}%
{\Delta_{\varepsilon}}\int_{u-\Delta_{\varepsilon}}^{t}A_{\bar{Y}%
^{\varepsilon}(u),j}^{\eta^{\varepsilon}(u)}(\bar{X}^{\varepsilon
}(s))dsdu\right\vert \\
&  \quad\leq C_{2}e^{M_{0}}|\Delta_{\varepsilon}|+C_{2}e^{M_{0}}%
\sup_{|s-s^{\prime}|\leq\Delta_{\varepsilon},s,s^{\prime}\in\lbrack0,T]}%
\Vert\bar{X}^{\varepsilon}(s)-\bar{X}^{\varepsilon}(s^{\prime})\Vert
+\frac{C_{2}}{M_{0}}\int_{0}^{T}\hat{\ell}(\eta^{\varepsilon}(u))du\\
&  \quad\leq C_{2}e^{M_{0}}|\Delta_{\varepsilon}|+C_{2}e^{M_{0}}%
\sup_{|s-s^{\prime}|\leq\Delta_{\varepsilon},s,s^{\prime}\in\lbrack0,T]}%
\Vert\bar{X}^{\varepsilon}(s)-\bar{X}^{\varepsilon}(s^{\prime})\Vert
+\frac{C_{2}M}{M_{0}}.
\end{align*}
A similar calculation shows that the above inequality is also true for all
$t\in\lbrack0,\Delta_{\varepsilon}]$. Thus
\begin{align}
&  \limsup_{\varepsilon\rightarrow0}\sup_{t\in\lbrack0,T]}\left\vert
\int_{\mathbb{H}_{t}}A_{y,j}^{\eta}(\bar{X}^{\varepsilon}(s))Q^{\varepsilon
}(d\mathbf{v})-\int_{0}^{t}A_{\bar{Y}^{\varepsilon}(u),j}^{\eta^{\varepsilon
}(u)}(\bar{X}^{\varepsilon}(u))du\right\vert \nonumber\\
&  \quad\leq\limsup_{M_{0}\rightarrow\infty}\left[  C_{2}e^{M_{0}}%
\limsup_{\varepsilon\rightarrow0}\sup_{|s-s^{\prime}|\leq\Delta_{\varepsilon
},s,s^{\prime}\in\lbrack0,T]}\sup_{\bar{\varepsilon}\in(0,1)}\Vert\bar
{X}^{\bar{\varepsilon}}(s)-\bar{X}^{\bar{\varepsilon}}(s^{\prime})\Vert
+\frac{C_{2}M}{M_{0}}\right] \nonumber\\
&  \quad\leq\limsup_{M_{0}\rightarrow\infty}\frac{C_{2}M}{M_{0}}=0,
\label{eq:ab1202}%
\end{align}
where the second inequality follows on noting that since $\bar{X}%
^{\varepsilon}\rightarrow\xi$, the collection $\{\bar{X}^{\varepsilon}\}$ is
equicontinuous. Recall the sets $E_{ij}(x)$ defined in (\ref{eqn:defEij}).
Then from \eqref{eq:eq752-17} and \eqref{eq:eq1347}, for any $\phi$ mapping
$\mathbb{L}$ to $\mathbb{R}$
\begin{align}
&  \phi(\bar{Y}^{\varepsilon}(t))-\phi(y_{0})\nonumber\\
&  \quad=\sum_{(i,j)\in\mathbb{T}}(\phi(j)-\phi(i))\int_{[0,\zeta
]\times\lbrack0,t]}1_{\{\bar{Y}^{\varepsilon}(s-)=i\}}1_{E_{ij}(\bar
{X}^{\varepsilon}(s))}(r)N_{ij}^{{\varepsilon}^{-1}\varphi_{ij}^{\varepsilon}%
}(dr\times ds)\nonumber\\
&  \quad={\varepsilon}^{-1}\sum_{(i,j)\in\mathbb{T}}(\phi(j)-\phi
(i))\int_{[0,\zeta]\times\lbrack0,t]}1_{\{\bar{Y}^{\varepsilon}(s)=i\}}%
1_{E_{ij}(\bar{X}^{\varepsilon}(s))}(r)\varphi_{ij}^{\varepsilon}%
(r,s)\lambda_{\zeta}(dr)ds+\bar{M}_{\phi}^{\varepsilon}(t)\nonumber\\
&  \quad={\varepsilon}^{-1}\sum_{(i,j)\in\mathbb{T}}(\phi(j)-\phi(i))\int
_{0}^{t}1_{\{\bar{Y}^{\varepsilon}(s)=i\}}\eta_{j}^{\varepsilon}(E_{ij}%
(\bar{X}^{\varepsilon}(s)))ds+\bar{M}_{\phi}^{\varepsilon}(t),
\label{Ito_form}%
\end{align}
where $\bar{M}_{\phi}^{\varepsilon}$ is the martingale given by
\[
\bar{M}_{\phi}^{\varepsilon}(t)\doteq\sum_{(i,j)\in\mathbb{T}}(\phi
(j)-\phi(i))\int_{[0,\zeta]\times\lbrack0,t]}1_{\{\bar{Y}^{\varepsilon
}(s-)=i\}}1_{E_{ij}(\bar{X}^{\varepsilon}(s))}(r)\tilde{N}_{ij}^{{\varepsilon
}^{-1}\varphi_{ij}^{\varepsilon}}(dr\times ds)
\]
and $\tilde{N}_{ij}^{{\varepsilon}^{-1}\varphi_{ij}^{\varepsilon}}(dr\times
ds)=N_{ij}^{{\varepsilon}^{-1}\varphi_{ij}^{\varepsilon}}(dr\times
ds)-{\varepsilon}^{-1}\varphi_{ij}^{\varepsilon}(r,s)dr\,ds$. By Doob's
inequality
\begin{align*}
\mathbb{E}\sup_{0\leq s\leq T}|\varepsilon\bar{M}_{\phi}^{\varepsilon
}(s)|^{2}  &  \leq4\varepsilon\sum_{(i,j)\in\mathbb{T}}(\phi(j)-\phi
(i))^{2}\int_{[0,\zeta]\times\lbrack0,T]}\varphi_{ij}^{\varepsilon
}(r,s)\lambda_{\zeta}(dr)ds\\
&  \leq16\varepsilon\Vert\phi\Vert_{\infty}^{2}\sum_{(i,j)\in\mathbb{T}}%
\int_{[0,\zeta]\times\lbrack0,T]}\varphi_{ij}^{\varepsilon}(r,s)\lambda
_{\zeta}(dr)ds\\
&  \leq16\varepsilon\Vert\phi\Vert_{\infty}^{2}\left(  \zeta Te+M\right)  ,
\end{align*}
where the last inequality uses \eqref{eq:eq1824}. It follows that $\sup_{0\leq
s\leq T}|\varepsilon\bar{M}_{\phi}^{\varepsilon}(s)|$ converges to $0$ in
probability as $\varepsilon\rightarrow0$. Next, from \eqref{Ito_form} we see
that
\[
\varepsilon(\phi(\bar{Y}^{\varepsilon}(t))-\phi(y_{0}))=\sum_{(i,j)\in
\mathbb{T}}(\phi(j)-\phi(i))\int_{0}^{t}1_{\{\bar{Y}^{\varepsilon}(s)=i\}}%
\eta_{j}^{\varepsilon}(E_{ij}(\bar{X}^{\varepsilon}(s)))ds+\varepsilon\bar
{M}_{\phi}^{\varepsilon}(t),\;t\in\lbrack0,T].
\]
Since $\phi$ is bounded, we conclude that as $\varepsilon\rightarrow0$
\[
\sup_{0\leq t\leq T}\left\vert \sum_{(i,j)\in\mathbb{T}}(\phi(j)-\phi
(i))\int_{0}^{t}1_{\{\bar{Y}^{\varepsilon}(s)=i\}}\eta_{j}^{\varepsilon
}(E_{ij}(\bar{X}^{\varepsilon}(s)))ds\right\vert \rightarrow0.
\]
For fixed $j\in\mathbb{L}$, taking $\phi\doteq1_{\{j\}}$, we now see from
\eqref{eq:eq855mzr} that
\[
\sup_{0\leq t\leq T}\left\vert \sum_{i\in\mathbb{L}}\int_{0}^{t}1_{\{\bar
{Y}^{\varepsilon}(s)=i\}}\eta_{j}^{\varepsilon}(E_{ij}(\bar{X}^{\varepsilon
}(s)))ds\right\vert =\sup_{0\leq t\leq T}\left\vert \int_{0}^{t}A_{\bar
{Y}^{\varepsilon}(s),j}^{\eta^{\varepsilon}(s)}(\bar{X}^{\varepsilon
}(s))ds\right\vert
\]
converges to $0$ as $\varepsilon\rightarrow0$. Hence from \eqref{eq:ab1202},
$\int_{\mathbb{H}_{t}}A_{y,j}^{\eta}(\bar{X}^{\varepsilon}(s))Q^{\varepsilon
}(d\mathbf{v})\rightarrow0$, uniformly in $t\in\lbrack0,T]$. Now as in the
proof of Proposition \ref{prop611} (see \eqref{eqn:Abounds} and \eqref{E.1}),
$\int_{\mathbb{H}_{t}}A_{y,j}^{\eta}(\bar{X}^{\varepsilon}(s))Q^{\varepsilon
}(d\mathbf{v})\rightarrow\int_{\mathbb{H}_{t}}A_{y,j}^{\eta}(\xi
(s))Q(d\mathbf{v})$. Thus \eqref{inv} is satisfied and the result follows.
\end{proof}

We now prove the upper bound in \eqref{eq:uppbd}. Recall that for
$\varepsilon>0$, $(X^{\varepsilon},Y^{\varepsilon})$ is given as the unique
pathwise solution of \eqref{fastslow}.

\begin{theorem}
\label{thm:mainuppbd} For any $F \in C_{b}( C([0,T]:\mathbb{R}^{d}))$, the
inequality in \eqref{eq:uppbd} holds.
\end{theorem}

\begin{proof}
Using \eqref{VR22} for every $\varepsilon>0$, we can find $(\psi^{\varepsilon
},\varphi^{\varepsilon})\in\mathcal{U}_{b}$ such that
\[
-\varepsilon\log\mathbb{E}\left[  \exp\left(  -\varepsilon^{-1}%
F(X^{\varepsilon})\right)  \right]  \geq{\mathbb{E}}\left[  \bar{L}_{T}%
(\psi^{\varepsilon},\varphi^{\varepsilon})+F(\bar{X}^{\varepsilon})\right]
-{\varepsilon},
\]
where $\bar{X}^{\varepsilon}$ solves \eqref{fastslow_ctrl} (with
$(\psi,\varphi)$ replaced with $(\psi^{\varepsilon},\varphi^{\varepsilon})$).
Since $F$ is bounded
\[
\sup_{\varepsilon\in(0,1)}\mathbb{E}\bar{L}_{T}(\psi^{\varepsilon}%
,\varphi^{\varepsilon})\leq2\left\Vert F\right\Vert _{\infty}+1<\infty,
\]
By a localization argument (see, e.g., \cite[Section A.3]{BDM11}) we can
assume without loss of generality that for some $M\in(0,\infty)$
\[
\sup_{\varepsilon\in(0,1)}\bar{L}_{T}(\psi^{\varepsilon},\varphi^{\varepsilon
})\leq M\mbox{ a.s.}.
\]
Now Proposition \ref{ctrlmeas_tight} implies that $(\bar{X}^{\varepsilon
},Q^{\varepsilon})$ is tight and any limit point $(\xi,Q)$ satisfies
\eqref{ab2013} -- \eqref{inv} a.s. and consequently $Q\in\hat{\mathcal{A}}%
(\xi)$ a.s., where $\hat{\mathcal{A}}(\xi)$ was introduced above
\eqref{ab2013}. Assume without loss of generality that the convergence to
$(\xi,Q)$ holds along the full sequence. Then
\begin{align*}
\liminf_{\varepsilon\rightarrow0}-\varepsilon\log\mathbb{E}_{x}\left[
\exp\left(  -\varepsilon^{-1}F(X^{\varepsilon})\right)  \right]   &
\geq\mathbb{E}\left[  \left(  \int_{\mathbb{H_{T}}}\left[  \frac{1}{2}\Vert
z\Vert^{2}+\hat{\ell}(\eta)\right]  Q(d\mathbf{v})\right)  +F(\xi)\right]  \\
&  \geq\mathbb{E}\left[  \inf_{Q\in\hat{\mathcal{A}}(\xi)}\left(
\int_{\mathbb{H_{T}}}\left[  \frac{1}{2}\Vert z\Vert^{2}+\hat{\ell}%
(\eta)\right]  Q(d\mathbf{v})\right)  +F(\xi)\right]  \\
&  =\mathbb{E}\left[  I(\xi)+F(\xi)\right]  \\
&  \geq\inf_{\xi\in C([0,T]:\mathbb{R}^{d})}\left[  I(\xi)+F(\xi)\right]  ,
\end{align*}
where the first inequality uses Fatou's lemma and part (i) of Proposition
\ref{ctrlmeas_tight}, the second uses the property that $Q\in\hat{\mathcal{A}%
}(\xi)$ a.s. and the third uses the definition of the rate function in
\eqref{ratefn} and Proposition \ref{prop:ieqihat}.
\end{proof}

\section{Near Optimal Paths with a Unique Characterization}

\label{sec:uniqchae17} In order to prove the large deviation lower bound
\eqref{maintoshowlow}, a natural approach is to consider a $\xi$ that is a
near infimum for the right side in \eqref{maintoshowlow} and construct a
sequence of controls $(\psi^{\varepsilon},\varphi^{\varepsilon})$ such that
$\bar{X}^{\varepsilon}\Rightarrow\xi$ where $\bar{X}^{\varepsilon}$ is as in
\eqref{fastslow_ctrl} with $(\psi,\varphi)$ replaced by $(\psi^{\varepsilon
},\varphi^{\varepsilon})$ respectively. Along with an appropriate convergence
of costs, the variational representation in \eqref{VR22} can then be used to
argue that \eqref{maintoshowlow} holds. For a near optimal $\xi$, let
$(u,\varphi,\pi)\in\mathcal{A}(\xi)$ be a near infimum for the expression on
the right side of \eqref{eq:rtfnnew}. The control pair $(u,\varphi)$ suggests
a natural sequence of controls $(\psi^{\varepsilon},\varphi^{\varepsilon})$
(see \eqref{eq:contfromlim}) for the construction of controlled processes
$\bar{X}^{\varepsilon}$ and an occupation measure $Q^{\varepsilon}$ of the
form in \eqref{eqn:defofQ}. Our strategy in the proof of the lower bound given
in Section \ref{lowbd} will be to show that any limit point $\bar{\xi}$ of
$\bar{X}^{\varepsilon}$ and a suitable marginal $\bar{\pi}$ of the limit point
$\bar{Q}$ of $Q^{\varepsilon}$ solves the system in
\eqref{eq:stateq}-\eqref{eq:eqinvar} for the given $(u,\varphi)$. The key
result then needed in order to complete the proof is to argue that the system
admits a unique solution for the given choice of $(u,\varphi)$, thereby
proving $(\bar{\xi},\bar{\pi})=(\xi,\pi)$ a.s. Although proving such a result
for an arbitrary $\xi$ and an arbitrary $(u,\varphi,\pi)\in\mathcal{A}(\xi)$
appears to be challenging, in this section we show that one can perturb $\xi$
slightly to $\xi^{\ast}$, without affecting the cost too much, and find a near
optimal $(u^{\ast},\varphi^{\ast},\pi^{\ast})\in\mathcal{A}(\xi^{\ast})$ such
that the desired uniqueness property discussed above does in fact hold for
$(u^{\ast},\varphi^{\ast})$. See in particular parts 4 and 5 of the following proposition.

\begin{proposition}
\label{prop:prop4.1} Let $\xi\in C([0,T]:\mathbb{R}^{d})$ be such that
$I(\xi)<\infty$. Fix $\gamma\in(0,1)$. Then there exists $\xi^{\ast}\in
C([0,T]:\mathbb{R}^{d})$ such that
\begin{equation}
\Vert\xi-\xi^{\ast}\Vert_{T}\doteq\sup_{0\leq s\leq T}\Vert\xi(s)-\xi^{\ast
}(s)\Vert<\gamma, \label{eq:eq811_17}%
\end{equation}
and there is $(u^{\ast},\varphi^{\ast}=(\varphi_{ij}^{\ast}),\pi^{\ast}%
=(\pi_{i}^{\ast}))\in\mathcal{A}(\xi^{\ast})$ with the following properties.

\begin{enumerate}
\item For some constants $m_{2},m_{3}\in(0,\infty)$ and all $(s,z)\in
\lbrack0,T]\times\lbrack0,\zeta]$ and $(i,j)\in\mathbb{T}$,
\[
m_{3}\geq\varphi_{ij}^{\ast}(s,z)\geq m_{2}.
\]

\item There is a measurable map $\varrho: [0,T]\times\mathbb{R}^{d}
\to\mathcal{P}(\mathbb{L})$ such that for all $(s,x) \in[0,T]\times
\mathbb{R}^{d}$
\[
\sum_{i\in\mathbb{L}} \varrho_{i}(s,x) A_{ij}^{\varphi_{i}^{*}(s,\cdot)}(x) =
0,
\]
and for some $c_{1} \in(0,\infty)$
\[
\sup_{s\in[0,T]} \max_{i\in\mathbb{L}} |\varrho_{i}(s,x) - \varrho
_{i}(s,\tilde x)| \le c_{1}\|x-\tilde x\|, \mbox{ for all } x,\tilde x
\in\mathbb{R}^{d},
\]
\begin{equation}
\label{eq:eq337_17}\inf_{(s,x)\in[0,T]\times\mathbb{R}^{d}}\min_{i
\in\mathbb{L}} \varrho_{i}(s,x) \ge c_{1}^{-1}.
\end{equation}

\item If for any $(s,x)\in\lbrack0,T]\times\mathbb{R}^{d}$, $\bar{\pi}%
\in\mathcal{P}(\mathbb{L})$ satisfies
\[
\sum_{i\in\mathbb{L}}\bar{\pi}_{i}A_{ij}^{\varphi_{i}^{\ast}(s,\cdot)}(x)=0,
\]
then $\bar{\pi}=\varrho(s,x)$. In particular, $\varrho(s,\xi^{\ast}%
(s))=\pi^{\ast}(s)$.

\item If for the given $u^{\ast}$ and $\varphi^{\ast}$, \eqref{eq:stateq} and
\eqref{eq:eqinvar} are satisfied for any other $(\tilde{\xi},\tilde{\pi})\in
C([0,T]:\mathbb{R}^{d})\times M([0,T]:\mathcal{P}(\mathbb{L}))$, then
$(\tilde{\xi},\tilde{\pi})=(\xi^{\ast},\pi^{\ast})$.

\item The cost associated with $(u^{\ast},\varphi^{\ast})$ satisfies:
\begin{equation}
\sum_{i}\frac{1}{2}\int_{0}^{T}\Vert u_{i}^{\ast}(s)\Vert^{2}\pi_{i}^{\ast
}(s)ds+\sum_{(i,j)\in\mathbb{T}}\int_{[0,\zeta]\times\lbrack0,T]}\ell
(\varphi_{ij}^{\ast}(s,z))\pi_{i}^{\ast}(s)\lambda_{\zeta}(dz)ds\leq
I(\xi)+\gamma.\label{eq:eq336_17}%
\end{equation}

\end{enumerate}
\end{proposition}

\begin{remark}
\label{rem:stratofproof}

We now given an outline of the proof strategy. Let $\xi\in C([0,T]:\mathbb{R}%
^{d})$ be such that $I(\xi)<\infty$ and fix $\gamma\in(0,1)$. Let
$(u,\varphi,\pi)\in\mathcal{A}(\xi)$ be such that
\begin{equation}
\sum_{i}\frac{1}{2}\int_{0}^{T}\Vert u_{i}(s)\Vert^{2}\pi_{i}(s)ds+\sum
_{(i,j)\in\mathbb{T}}\int_{[0,\zeta]\times\lbrack0,T]}\ell(\varphi
_{ij}(s,z))\pi_{i}(s)\lambda_{\zeta}(dz)ds\leq I(\xi)+\frac{\gamma}%
{2}.\label{eq:eq1143_17}%
\end{equation}
Note that there are four time dependent objects appearing in the limit
deterministic controlled dynamics: the trajectory $\xi$, the empirical measure
on the fast variables $\pi$, the controls $u$ that correspond to shifting the
mean of the Brownian noises, and the thinning function $\varphi$ that controls
the rates for the fast variables. In addition, there is complete coupling of
the fast and slow variables, and in particular the dynamics of the fast
variables at time $s$ will depend on both the controlled rates and $\xi$ at
$s$. The key issue regarding uniqueness is to make sure that these thinning
functions $\varphi$ can be bounded away from zero, which will imply ergodicity
of the associated Markov processes giving the uniqueness of the corresponding
$\pi$. If for a given collection $(u,\varphi,\pi,\xi)$ the rates are not
bounded away from zero, then we must show they can be perturbed so this is
true, while at the same time making only a small change in $\xi$ and the cost.

The steps are as follows. (a) We first perturb $\pi$ to $\pi^{\delta}$ (see
\eqref{eq:pitipidel}), so that every state has strictly positive mass under
$\pi^{\delta}$. This positivity is used crucially in the remaining steps. (b)
Replacing $\pi$ with $\pi^{\delta}$ in \eqref{eq:stateq} leads to a
perturbation of the target trajectory $\xi$. To ensure that the trajectory
perturbation is not too large, we modify the control $u$ to $u^{\delta}$ in a
way that compensates for the change in $\pi$ (see \eqref{eq:contupert}). (c)
The perturbed measure $\pi^{\delta}$ need not be stationary (i.e., satisfy
\eqref{eq:eqinvar}) for the original thinning control $\varphi$ and the new
trajectory $\xi^{\delta}$. In order to remedy this we next perturb $\varphi$
to $\varphi^{\delta}$ (see \eqref{eq:vphitovphid}). (d) With the perturbed
$\varphi^{\delta}$ \eqref{eq:eqinvar} is satisfied so the $\pi^{\delta}$ would
be stationary, but with $\xi$ rather than $\xi^{\delta}$. In particular the
constructed $(u^{\delta},\varphi^{\delta},\pi^{\delta})$ is not in general in
$\mathcal{A}(\xi^{\delta})$. This leads to our last modification where we
change $\varphi^{\delta}$ to $\tilde{\varphi}^{\delta}$. It is at this point
that the formulation of the original dynamics as the solution to an SDE driven
by a collection of PRMs, and corresponding formulation of the control problem
in terms of thinning functions, is very convenient (see \eqref{eq:finalmod}).
With this change we now have a $(u^{\delta},\tilde{\varphi}^{\delta}%
,\pi^{\delta})\in\mathcal{A}(\xi^{\delta})$. Furthermore with $\delta
=\delta^{\ast}$ sufficiently small these perturbed quantities $(u^{\delta
^{\ast}},\tilde{\varphi}^{\delta^{\ast}},\pi^{\delta^{\ast}},\xi^{\delta
^{\ast}})=(u^{\ast},\varphi^{\ast},\pi^{\ast},\xi^{\ast})$ will satisfy all
the desired properties.
\end{remark}

\noindent\emph{Proof of Proposition \ref{prop:prop4.1}.} Let $\xi$ and
$(u,\varphi,\pi)\in\mathcal{A}(\xi)$ be as in Remark \ref{rem:stratofproof}.
In particular,
\[
\sum_{i}\frac{1}{2}\int_{0}^{T}\Vert u_{i}(s)\Vert^{2}\pi_{i}(s)ds+\sum
_{(i,j)\in\mathbb{T}}\int_{[0,\zeta]\times\lbrack0,T]}\ell(\varphi
_{ij}(s,z))\pi_{i}(s)\lambda_{\zeta}(dz)ds\leq I(\xi)+\frac{\gamma}{2}.
\]
We claim that without loss of generality it can be assumed that for some
$m_{0}\in(0,\infty)$,
\begin{equation}
\sup_{(s,z)\in\lbrack0,T]\times\lbrack0,\zeta]}\max_{(i,j)\in\mathbb{T}%
}[\varphi_{ij}(s,z)\pi_{i}(s)]\leq m_{0}. \label{eq:eq845}%
\end{equation}
We first prove the proposition assuming the claim. The proof of the claim is
given at the end.

For $x\in\mathbb{R}^{d}$ let $\nu(x)$ as in Theorem \ref{assum3} be the
stationary distribution for the fast system when the slow variable equals $x$.
Fix $\delta>0$ and define
\begin{equation}
\pi_{j}^{\delta}(s)\doteq(1-\delta)\pi_{j}(s)+\delta\nu_{j}(\xi(s)).
\label{eq:pitipidel}%
\end{equation}
Note that
\[
\sup_{s\in\lbrack0,T]}\sum_{j}|\pi_{j}^{\delta}(s)-\pi_{j}(s)|\leq2\delta.
\]
Define
\begin{equation}
u_{j}^{\delta}(s)\doteq u_{j}(s)\frac{\pi_{j}(s)}{\pi_{j}^{\delta}(s)}%
,\;s\in\lbrack0,T],j\in\mathbb{L}. \label{eq:contupert}%
\end{equation}
Then $u_{j}^{\delta}(s)\pi_{j}^{\delta}(s)=u_{j}(s)\pi_{j}(s)$ for all $s$ and
$j$. Define
\begin{equation}
\xi^{\delta}(t)=x_{0}+\sum_{j}\int_{0}^{t}b_{j}(\xi^{\delta}(s))\pi
_{j}^{\delta}(s)ds+\sum_{j}\int_{0}^{t}a_{j}(\xi^{\delta}(s))u_{j}^{\delta
}(s)\pi_{j}^{\delta}(s)ds. \label{eq:eq1144_17}%
\end{equation}
From the Lipschitz properties of $b_{j},\sigma_{j}$, (\ref{eq:eq1144_17}) has
a unique solution for the given $u^{\delta}$ and $\pi^{\delta}$. Note that
with $M=I(\xi)+1$,
\[
\int_{0}^{T}\left(  \sum_{j}\pi_{j}^{\delta}(s)(|u_{j}^{\delta}(s)|+1)\right)
ds\leq(T+(2TM)^{1/2})\doteq a(M).
\]
Then by Gronwall's lemma
\[
\Vert\xi-\xi^{\delta}\Vert_{T}\leq K\delta,
\]
where $K=2M_{0}a(M)e^{d_{{\tiny {\mbox{{\em lip}}}}}a(M)}$ and $M_{0}%
=\sup_{0\leq s\leq T,j\in\mathbb{L}}(\Vert b_{j}(\xi(s))\Vert+\Vert a_{j}%
(\xi(s))\Vert)$. Now define, for $(i,j)\in\mathbb{T}$ and $(s,z)\in
\lbrack0,T]\times\lbrack0,\zeta]$,
\begin{equation}
\varphi_{ij}^{\delta}(s,z)\doteq(1-\delta)\frac{\pi_{i}(s)}{\pi_{i}^{\delta
}(s)}\varphi_{ij}(s,z)+\delta\frac{\nu_{i}(\xi(s))}{\pi_{i}^{\delta}(s)}
\label{eq:vphitovphid}%
\end{equation}
and
\begin{equation}
\beta_{ij}^{\delta}(s)\doteq\int_{E_{ij}(\xi(s))}\varphi_{ij}^{\delta
}(s,z)\lambda_{\zeta}(dz),\;s\in\lbrack0,T],i\neq j, \label{eqn:defrhodelta}%
\end{equation}
and $\beta_{ii}^{\delta}(s)\doteq-\sum_{j:j\neq i}\beta_{ij}^{\delta}(s)$.
Then since the $\pi_{i}^{\delta}$ will cancel and $(u,\varphi,\pi
)\in\mathcal{A}(\xi)$, with $A^{\delta}(s)\doteq(\beta_{ij}^{\delta}(s))$
\begin{equation}
\pi^{\delta}(s)A^{\delta}(s)=0. \label{eq:eq1147_17}%
\end{equation}
Thus $\pi^{\delta}(s)$ is stationary for $A^{\delta}(s)$. However, from
\eqref{eq:contrates} $\beta_{ij}^{\delta}(s)=\rho_{i}^{\varphi_{i}^{\delta
}(s,\cdot)}(\xi(s))$, and hence $(u^{\delta},\varphi^{\delta},\pi^{\delta})$
is not in general in $\mathcal{A}(\xi^{\delta})$. We now construct a further
modification, $\tilde{\varphi}^{\delta}$, of $\varphi^{\delta}$ such that
$(u^{\delta},\tilde{\varphi}^{\delta},\pi^{\delta})\in\mathcal{A}(\xi^{\delta
})$, and such that the uniqueness of \eqref{eq:stateq} - \eqref{eq:eqinvar}
(with $(u,\varphi)$ replaced by $(u^{\delta},\tilde{\varphi}^{\delta})$) holds.

Let
\begin{equation}
\rho_{ij}(x)\doteq c_{i}(x)r_{ij}(x)=\lambda_{\zeta}(E_{ij}(x)),
\label{eq:jumpintens}%
\end{equation}
where we recall that $c_{i}(x)$ is the overall rate of transitions out of $i$
for the fast process when the slow process is in state $x$, and $r_{ij}(x)$
gives the probability of transition to state $j$. For $(i,j)\in\mathbb{T}$
define
\begin{equation}
\tilde{\varphi}_{ij}^{\delta}(s,z)=\left\{
\begin{array}
[c]{cc}%
\frac{\beta_{ij}^{\delta}(s)}{\rho_{ij}(\xi^{\delta}(s))}, & z\in E_{ij}%
(\xi^{\delta}(s)),\\
\  & \\
1 & \mbox{ otherwise }.
\end{array}
\right.  \label{eq:finalmod}%
\end{equation}
Then for such $(i,j)$
\[
\beta_{ij}^{\delta}(s)=\int_{E_{ij}(\xi(s))}\varphi_{ij}^{\delta}%
(s,z)\lambda_{\zeta}(dz)=\int_{E_{ij}(\xi^{\delta}(s))}\tilde{\varphi}%
_{ij}^{\delta}(s,z)\lambda_{\zeta}(dz).
\]
Then, from \eqref{eq:eq1144_17} and \eqref{eq:eq1147_17}, $(u^{\delta}%
,\tilde{\varphi}^{\delta},\pi^{\delta})\in\mathcal{A}(\xi^{\delta})$ and
$\beta_{ij}^{\delta}(s)=\rho_{ij}^{\tilde{\varphi}_{i}^{\delta}(s,\cdot)}%
(\xi^{\delta}(s))$, where $\rho_{ij}^{\psi}(x)$ was defined in
(\ref{eq:contrates}), and $\tilde{\varphi}_{i}^{\delta}$ denotes the
collection of controls $(\tilde{\varphi}_{ij}^{\delta},j\in\mathbb{L})$. Next
note that by construction, for $(i,j)\in\mathbb{T}$,
\[
\varphi_{ij}^{\delta}(s,z)\geq\delta\frac{\nu_{i}(\xi(s))}{\pi_{i}^{\delta
}(s)}\geq\delta\underline{\nu},\mbox{ for all }(s,z)\in\lbrack0,T]\times
\lbrack0,\zeta]
\]
and, from \eqref{eq:eq845}, for all $\delta>0$
\[
\varphi_{ij}^{\delta}(s,z)\pi_{i}^{\delta}(s)\leq m_{0}+1\doteq m_{1}.
\]
Let $\underline{r}\doteq\underline{\varsigma}\kappa_{3}$, where $\underline
{\varsigma}$ and $\kappa_{3}$ are defined above Assumption \ref{assum2}. Then
for $(i,j)\in\mathbb{T}$ the definition (\ref{eqn:defrhodelta}) implies%
\[
\frac{m_{1}\zeta}{\delta\underline{\nu}}\geq\beta_{ij}^{\delta}(s)\geq
\delta\underline{\nu}\,\underline{r}.
\]
Also, from \eqref{eq:jumpintens}, for each $s$
\[
\int_{\lbrack0,\zeta]}\ell(\tilde{\varphi}_{ij}^{\delta}(s,z))\lambda_{\zeta
}(dz)=\ell\left(  \frac{\beta_{ij}^{\delta}(s)}{\rho_{ij}(\xi^{\delta}%
(s))}\right)  \rho_{ij}(\xi^{\delta}(s))
\]
and using convexity
\[
\int_{\lbrack0,\zeta]}\ell(\varphi_{ij}^{\delta}(s,z))\lambda_{\zeta}%
(dz)\geq\ell\left(  \frac{\beta_{ij}^{\delta}(s)}{\rho_{ij}(\xi(s))}\right)
\rho_{ij}(\xi(s)).
\]
It is easy to check that for $a\geq0$ and $b,c>0$,
\[
\left\vert \ell\left(  \frac{a}{b}\right)  b-\ell\left(  \frac{a}{c}\right)
c\right\vert \leq\left(  1+\frac{a}{b\wedge c}\right)  |b-c|.
\]
The Lipschitz properties of the underlying transition rates $\rho_{ij}(\cdot)$
(see \eqref{liptheta}) yield the following inequalities, each of which is
explained after the display:
\begin{align*}
&  \sum_{(i,j)\in\mathbb{T}}\int_{[0,\zeta]\times\lbrack0,T]}\pi_{i}^{\delta
}(s)\ell(\tilde{\varphi}_{ij}^{\delta}(s,z))\lambda_{\zeta}(dz)ds-\sum
_{(i,j)\in\mathbb{T}}\int_{[0,\zeta]\times\lbrack0,T]}\pi_{i}^{\delta}%
(s)\ell(\varphi_{ij}^{\delta}(s,z))\lambda_{\zeta}(dz)ds\\
&  \quad\leq\kappa_{2}\Vert\xi-\xi^{\delta}\Vert_{T}\int_{[0,T]}\sum
_{(i,j)\in\mathbb{T}}\pi_{i}^{\delta}(s)\left(  \frac{\beta_{ij}^{\delta}%
(s)}{\underline{r}}+1\right)  ds\\
&  \quad\leq\kappa_{2}T|\mathbb{L}|\Vert\xi-\xi^{\delta}\Vert_{T}+\frac
{\kappa_{2}}{\underline{r}}\Vert\xi-\xi^{\delta}\Vert_{T}\int_{[0,\zeta
]\times\lbrack0,T]}\sum_{(i,j)\in\mathbb{T}}\pi_{i}^{\delta}(s)\varphi
_{ij}^{\delta}(s,z)\lambda_{\zeta}(dz)ds\\
&  \quad\leq\kappa_{2}T|\mathbb{L}|\Vert\xi-\xi^{\delta}\Vert_{T}+\frac
{\kappa_{2}}{\underline{r}}\Vert\xi-\xi^{\delta}\Vert_{T}\int_{[0,\zeta
]\times\lbrack0,T]}\left(  |\mathbb{L}|+\sum_{(i,j)\in\mathbb{T}}\pi
_{i}(s)\varphi_{ij}(s,z)\right)  \lambda_{\zeta}(dz)ds\\
&  \quad\leq\delta K_{1}.
\end{align*}
The first inequality uses the previous three displays, the second uses the
definition of $\beta_{ij}^{\delta}$ in (\ref{eqn:defrhodelta}), the third uses
(\ref{eq:vphitovphid}), and the final one uses the definition
\[
K_{1}\doteq K\kappa_{2}\left(  T|\mathbb{L}|+\frac{1}{\underline{r}}\left(
|\mathbb{L}|\zeta T(1+e)+eM\right)  \right)  ,
\]
$\Vert\xi-\xi^{\delta}\Vert_{T}\leq\delta$, and the fact that $x\leq
e(1+\ell(x))$ for all $x\geq0$. Hence we obtain
\begin{align*}
\int_{\lbrack0,\zeta]\times\lbrack0,T]}\sum_{(i,j)\in\mathbb{T}}\pi
_{i}^{\delta}(s)\ell(\tilde{\varphi}_{ij}^{\delta}(s,z))\lambda_{\zeta}(dz)ds
&  \leq\int_{\lbrack0,\zeta]\times\lbrack0,T]}\sum_{(i,j)\in\mathbb{T}}\pi
_{i}^{\delta}(s)\ell(\varphi_{ij}^{\delta}(s,z))\lambda_{\zeta}(dz)ds+K_{1}%
\delta\\
&  \leq\int_{\lbrack0,\zeta]\times\lbrack0,T]}\sum_{(i,j)\in\mathbb{T}}\pi
_{i}(s)\ell(\varphi_{ij}(s,z))\lambda_{\zeta}(dz)ds+K_{1}\delta,
\end{align*}
where the last line is a consequence of the fact that
\begin{align}
\pi_{i}^{\delta}(s)\ell(\tilde{\varphi}_{ij}^{\delta}(s,z))  &  =\pi
_{i}^{\delta}(s)\ell\left(  (1-\delta)\frac{\pi_{i}(s)}{\pi_{i}^{\delta}%
(s)}\varphi_{ij}(s,z)+\delta\frac{\nu_{i}(\xi(s))}{\pi_{i}^{\delta}(s)}\right)
\nonumber\\
&  \leq\pi_{i}^{\delta}(s)(1-\delta)\frac{\pi_{i}(s)}{\pi_{i}^{\delta}(s)}%
\ell(\varphi_{ij}(s,z))\leq\pi_{i}(s)\ell(\varphi_{ij}(s,z)).
\label{eq:eq1259}%
\end{align}
Next note that if $\pi_{i}(s)\leq\pi_{i}^{\delta}(s)$ then%
\[
\Vert u_{i}^{\delta}(s)\Vert^{2}\pi_{i}^{\delta}(s)=\Vert u_{i}(s)\Vert
^{2}\frac{\pi_{i}(s)}{\pi_{i}^{\delta}(s)}\pi_{i}(s)\leq\Vert u_{i}%
(s)\Vert^{2}\pi_{i}(s).
\]
However if $\pi_{i}(s)\geq\pi_{i}^{\delta}(s)$ then $\pi_{i}(s)\geq\nu_{i}%
(\xi(s))$ follows, and thus $\pi_{i}^{\delta}(s)\geq\underline{\nu}$.
Therefore
\[
\left\vert \frac{\pi_{i}(s)}{\pi_{i}^{\delta}(s)}-1\right\vert \leq
\frac{2\delta}{\pi_{i}^{\delta}(s)}\leq\frac{2\delta}{\underline{\nu}}.
\]
Thus in this case
\[
\Vert u_{i}^{\delta}(s)\Vert^{2}\pi_{i}^{\delta}(s)\leq\Vert u_{i}(s)\Vert
^{2}\pi_{i}(s)+\frac{2\delta}{\underline{\nu}}\Vert u_{i}(s)\Vert^{2}\pi
_{i}(s).
\]
Combining the two cases we have from \eqref{eq:eq1143_17}
\begin{equation}
\frac{1}{2}\int_{0}^{T}\sum_{i}\Vert u_{i}^{\delta}(s)\Vert^{2}\pi_{i}%
^{\delta}(s)ds\leq\frac{1}{2}\int_{0}^{T}\sum_{i}\Vert u_{i}(s)\Vert^{2}%
\pi_{i}(s)ds+\frac{2\delta}{\underline{\nu}}M. \label{eq:eq100_17}%
\end{equation}
Taking $\delta^{\ast}\doteq\min\{\gamma/K,\gamma/4K_{1},\gamma\underline{\nu
}/8M\}$ we now see that with
\[
(\xi^{\ast},u^{\ast},\varphi^{\ast},\pi^{\ast})\doteq(\xi^{\delta^{\ast}%
},u^{\delta^{\ast}},\tilde{\varphi}^{\delta^{\ast}},\pi^{\delta^{\ast}}),
\]
$(u^{\ast},\varphi^{\ast},\pi^{\ast})\in\mathcal{A}(\xi^{\ast})$. Also, for
all $(s,z)\in\lbrack0,T]\times\lbrack0,\zeta]$ and $i\in\mathbb{L}$,
\begin{equation}
m_{3}\doteq\frac{m_{1}\zeta}{\delta^{\ast}\underline{r}\,\underline{\nu}}%
\vee1\geq\varphi_{ij}^{\ast}(s,z)\geq\frac{\delta^{\ast}\underline{\nu
}\,\underline{r}}{\zeta}\wedge1\doteq m_{2}, \label{eq:eq901}%
\end{equation}
namely item 1 in the proposition holds. From \eqref{eq:eq901} and the
definition (\ref{eq:contrates}), for $(i,j)\in\mathbb{T}$, $s\in\lbrack0,T]$
and $x\in\mathbb{R}^{d}$
\begin{equation}
\rho_{ij}^{\varphi_{i}^{\ast}(s,\cdot)}(x)\in\lbrack m_{2}\underline{r}%
,m_{3}\zeta]. \label{eq:eq906}%
\end{equation}
Also, from \eqref{liptheta}, for all $x,x^{\prime}\in\mathbb{R}^{d}$,
\begin{equation}
|\rho_{ij}^{\varphi_{i}^{\ast}(s,\cdot)}(x)-\rho_{ij}^{\varphi_{i}^{\ast
}(s,\cdot)}(x^{\prime})|\leq m_{3}\kappa_{2}\Vert x-x^{\prime}\Vert.
\label{eq:eq1210}%
\end{equation}
Using \eqref{eq:eq906} and \eqref{eq:eq1210} it follows as in the proof of
Lemma \ref{invmeas_lip2} that items 2 and 3 in the proposition hold. Next if
\eqref{eq:stateq} (with $u$ replaced by $u^{\ast}$) and \eqref{eq:eqinvar}
(with $\varphi$ replaced by $\varphi^{\ast}$) are satisfied for any other
$(\tilde{\xi},\tilde{\pi})\in C([0,T]:\mathbb{R}^{d})\times\mathbb{M}%
([0,T]:\mathcal{P}(\mathbb{L}))$, then we must have from item 3 in the
proposition that $\tilde{\pi}(s)=\varrho(s,\tilde{\xi}(s))$ for all
$s\in\lbrack0,T]$. The Lipschitz property of $b,a$ and $\varrho$ then gives
that $\xi^{\ast}(t)=\tilde{\xi}(t)$, and thus also $\pi^{\ast}(t)=\tilde{\pi
}(t)$, for all $t\in\lbrack0,T]$. This completes the proof of item 4. Finally
consider item 5. Note that from \eqref{eq:eq1259}, \eqref{eq:eq100_17} and the
choice of $\delta^{\ast}$,%
\begin{align*}
&  \frac{1}{2}\sum_{i}\int_{0}^{T}\Vert u_{i}^{\ast}(s)\Vert^{2}\pi_{i}^{\ast
}(s)ds+\sum_{(i,j)\in\mathbb{T}}\int_{[0,\zeta]\times\lbrack0,T]}\ell
(\varphi_{ij}^{\ast}(s,z))\pi_{i}^{\ast}(s)\lambda_{\zeta}(dz)ds\\
&  \quad\leq\sum_{i}\frac{1}{2}\int_{0}^{T}\Vert u_{i}(s)\Vert^{2}\pi
_{i}(s)ds+\sum_{(i,j)\in\mathbb{T}}\int_{[0,\zeta]\times\lbrack0,T]}%
\ell(\varphi_{ij}(s,z))\pi_{i}(s)\lambda_{\zeta}(dz)ds+\gamma/2\\
&  \quad\leq I(\xi)+\gamma,
\end{align*}
where the last line is from \eqref{eq:eq1143_17}. This proves 5. We now prove
the claim made in \eqref{eq:eq845}. Note that we do not change the dynamics at
all if we redefine $\varphi_{ij}$ in the following way. With $\varphi$ on the
right equal to the old version and $\varphi$ on the left the new, for
$(i,j)\in\mathbb{T}$ set%
\[
\varphi_{ij}(s,z)=\left\{
\begin{array}
[c]{cc}%
\frac{\rho_{ij}^{\varphi_{i}(s,\cdot)}(\xi(s))}{\rho_{ij}(\xi(s))}, & z\in
E_{ij}(\xi(s)),\\
\  & \\
1, & z\in\lbrack0,\zeta]\setminus E_{ij}(\xi(s)).
\end{array}
\right.
\]
This amounts to assuming that outside $E_{ij}(\xi(s))$ the controlled jump
rates are the same as for the original system, and that within $E_{ij}%
(\xi(s))$ they are constant in $z$, in such a way the overall jump rates do
not change. Owing to convexity of $\ell$ and $\ell(1)=0$, this can only lower
cost while preserving the dynamics, and could have been assumed for any
candidate control for the jumps from the outset. Let
\begin{equation}
v(s)\doteq1+\sum_{(i,j)\in\mathbb{T}}\pi_{i}(s)\rho_{ij}^{\varphi_{i}}%
(\xi(s)),\;\bar{\rho}_{ij}^{\varphi_{i}}(\xi(s))\doteq\rho_{ij}^{\varphi_{i}%
}(\xi(s))/v(s) \label{eq:eqnormalv}%
\end{equation}
and for $\alpha>0$ and $(i,j)\in\mathbb{T}$
\[
\varphi_{ij}^{\alpha}(s,z)=\left\{
\begin{array}
[c]{cc}%
\alpha\frac{\bar{\rho}_{ij}^{\varphi_{i}}(\xi(s))}{\rho_{ij}(\xi(s))}, & z\in
E_{ij}(\xi(s)),\\
\  & \\
\frac{\alpha}{v(s)}, & z\in\lbrack0,\zeta]\setminus E_{ij}(\xi(s)).
\end{array}
\right.
\]
Thus for $\alpha>0$ the controlled jump rates are uniformly scaled by
$\alpha/v(s)$, and therefore $\sum_{i\in\mathbb{L}}\pi_{i}(s)A^{\varphi
_{i}^{\alpha}(s,\cdot)}(\xi(s))=0$. Since with $\alpha=v(s)$, $\varphi
^{\alpha}=\varphi$,
\[
\inf_{\alpha>0}\sum_{(i,j)\in\mathbb{T}}\pi_{i}(s)\int_{[0,\zeta]}\ell
(\varphi_{ij}^{\alpha}(s,z))\lambda_{\zeta}(dz)\leq\sum_{(i,j)\in\mathbb{T}%
}\pi_{i}(s)\int_{[0,\zeta]}\ell(\varphi_{ij}(s,z))\lambda_{\zeta}(dz).
\]
We now compute the infimum on the left side. For notational simplicity, write
$\bar{\rho}_{ij}^{\varphi_{i}(s,\cdot)}(\xi(s))$ as $\bar{\rho}_{ij}$, and
$\rho_{ij}(\xi(s))$ as $\rho_{ij}$. Also let $\theta_{ij}\doteq\zeta
-\lambda_{\zeta}(E_{ij}(\xi(s)))$. Note that $1\leq\theta_{ij}\leq\zeta$.
Differentiating with respect to $\alpha$ and setting the derivative to $0$, we
get
\begin{align*}
&  \log(\alpha)\sum_{(i,j)\in\mathbb{T}}\pi_{i}(s)\left(  \bar{\rho}%
_{ij}+\frac{\theta_{ij}}{v(s)}\right) \\
&  \quad=-\sum_{(i,j)\in\mathbb{T}}\pi_{i}(s)\left(  \bar{\rho}_{ij}%
\log\left(  \frac{\bar{\rho}_{ij}}{\rho_{ij}}\right)  +\frac{\theta_{ij}%
}{v(s)}\log\left(  \frac{1}{v(s)}\right)  \right)  .
\end{align*}
Thus
\[
\log(\alpha)=-\frac{\sum_{(i,j)\in\mathbb{T}}\pi_{i}(s)\left(  \bar{\rho}%
_{ij}\log\left(  \frac{\bar{\rho}_{ij}}{\rho_{ij}}\right)  +\frac{\theta_{ij}%
}{v(s)}\log\left(  \frac{1}{v(s)}\right)  \right)  }{\sum_{(i,j)\in\mathbb{T}%
}\pi_{i}(s)\left(  \bar{\rho}_{ij}+\frac{\theta_{i}}{v(s)}\right)  }.
\]
It is easily checked that there is $m_{6}\in(0,\infty)$ such that for all
$s\in\lbrack0,T]$ the minimizing $\alpha$ satisfies $\alpha\leq m_{6}$. Thus
from the definition of $\varphi^{\alpha}$ and since from \eqref{eq:eqnormalv}
$\pi_{i}(s)\bar{\rho}_{ij}(\xi(s))\leq1$, we see that with this $\alpha$, for
all $(i,j)\in\mathbb{T}$, $(s,z)\in\lbrack0,T]\times\lbrack0,\zeta]$
\[
\pi_{i}(s)\varphi_{ij}^{\alpha}(s,z)\leq\max\{m_{6}/\underline{r},m_{6}\}.
\]
This proves the claim in \eqref{eq:eq845} and completes the proof of the
proposition. \hfill\qed

\section{Large Deviation Lower Bound}

\label{lowbd}

The goal of this section is to prove the following theorem.

\begin{theorem}
\label{thm:mainlowbd} For any $F\in C_{b}(C([0,T]:\mathbb{R}^{d}))$, the
inequality in \eqref{maintoshowlow} holds, namely
\begin{equation}
\liminf_{\varepsilon\rightarrow0}\varepsilon\log\mathbb{E}\left[  \exp\left(
-\varepsilon^{-1}F(X^{\varepsilon})\right)  \right]  \geq-\inf_{\xi\in
C([0,T]:\mathbb{R}^{d})}\{F(\xi)+I(\xi)\}, \label{eq:lowbdnew}%
\end{equation}
where $X^{\varepsilon}$ solves \eqref{fastslow}, and $I$ is as defined in \eqref{ratefn}.
\end{theorem}

To do this we will show an upper bound on the corresponding variational
representations. The situation is in some sense simpler than the corresponding
lower bound, since a fixed control is being used. Thus the analysis is
essentially just the law of large numbers for a two time scale system, with
the main effort being to justify the replacement of the empirical measure of
the fast component by the corresponding stationary measure in the limit. We
begin with an elementary lemma.

\begin{lemma}
\label{lem:lusincgce} Let $m_{n}$, $m$ be finite measures on $[0,T]\times
\mathbb{L}$ such that the first marginal of $m_{n}$ and $m$ is Lebesgue
measure:
\[
m_{n}([a,b]\times\mathbb{L})=m([a,b]\times\mathbb{L})=b-a\mbox{ for all
}0\leq a\leq b\leq T.
\]
Suppose that $m_{n}$ converges weakly to $m$. Let $v:[0,T]\rightarrow
\mathbb{R}$ be an integrable map, i.e., $\int_{0}^{T}|v(s)|ds<\infty$. Then
\begin{equation}
\int_{\mathbb{L}\times\lbrack0,T]}v(s)1_{\{j\}}(y)m_{n}(dy\times
ds)\rightarrow\int_{\mathbb{L}\times\lbrack0,T]}v(s)1_{\{j\}}(y)m(dy\times
ds). \label{eq:eq637_17}%
\end{equation}

\end{lemma}

\begin{proof}
Clearly \eqref{eq:eq637_17} holds if $v$ is continuous. Now suppose that $v$
is bounded and let $M_{0}=\sup_{s\in\lbrack0,T]}|v(s)|$. Fix $\gamma>0$. Then
by Lusin's theorem, there is a continuous function $\bar{v}:[0,T]\rightarrow
\mathbb{R}$ such that $\sup_{s\in\lbrack0,T]}|\bar{v}(s)|\leq M_{0}$ and
$\lambda_{T}(s\in\lbrack0,T]:v(s)\neq\bar{v}(s))\leq\gamma$. Since
\eqref{eq:eq637_17} holds with $v$ replaced with $\bar{v}$, we have
\begin{align*}
&  \limsup_{n\rightarrow\infty}\left\vert \int_{\mathbb{L}\times\lbrack
0,T]}v(s)1_{\{j\}}(y)m_{n}(dy\times ds)-\int_{\mathbb{L}\times\lbrack
0,T]}v(s)1_{\{j\}}(y)m(dy\times ds)\right\vert \\
&  \leq\limsup_{n\rightarrow\infty}\left\vert \int_{\mathbb{L}\times
\lbrack0,T]}\bar{v}(s)1_{\{j\}}(y)m_{n}(dy\times ds)-\int_{\mathbb{L}%
\times\lbrack0,T]}\bar{v}(s)1_{\{j\}}(y)m(dy\times ds)\right\vert
+2M_{0}\gamma\\
&  =2M_{0}\gamma.
\end{align*}
Sending $\gamma\rightarrow0$ we see that \eqref{eq:eq637_17} holds for all
bounded $v$. Finally, consider a general integrable $v$. Then
\eqref{eq:eq637_17} holds with $v$ replaced with $v_{M}\doteq v1_{\{|v|\leq
M\}}$. Also, as $M\rightarrow\infty$
\[
\sup_{n\in\mathbb{N}}\int_{\mathbb{L}\times\lbrack0,T]}|v(s)|1_{\{|v(s)|\geq
M\}}1_{\{j\}}(y)m_{n}(dy\times ds)\leq\int_{\{|v|\geq M\}}|v(s)|ds\rightarrow
0.
\]
This shows that \eqref{eq:eq637_17} holds for a general integrable $v$ and
completes the proof of the lemma. \hfill
\end{proof}

\textbf{Proof of Theorem \ref{thm:mainlowbd}} Without loss of generality we
can assume that $F$ is Lipschitz continuous \cite[Corollary 1.2.5]{DE97},
i.e., for some $C_{F}\in(0,\infty)$,
\[
|F(\xi)-F(\tilde{\xi})|\leq C_{F}\Vert\xi-\tilde{\xi}\Vert_{T}%
\mbox{ for all }\xi,\tilde{\xi}\in C([0,T]:\mathbb{R}^{d}).
\]

Fix $\gamma\in(0,1)$. Let $\xi\in C([0,T]:\mathbb{R}^{d})$ be such that
\begin{equation}
I(\xi)+F(\xi)\leq\inf_{\zeta\in C([0,T]:\mathbb{R}^{d})}[I(\zeta
)+F(\zeta)]+\gamma. \label{eq:eqnearopt}%
\end{equation}
Since $I(\xi)<\infty$, we can find $\xi^{\ast}\in C([0,T]:\mathbb{R}^{d})$,
$(u^{\ast},\varphi^{\ast},\pi^{\ast})\in\mathcal{A}(\xi^{\ast})$ and a
measurable $\varrho:[0,T]\times\mathbb{R}^{d}\rightarrow\mathcal{P}%
(\mathbb{L})$ with properties stated in Proposition \ref{prop:prop4.1}. Let
$(\bar{X}^{\varepsilon},\bar{Y}^{\varepsilon})$ be defined through
\eqref{fastslow_ctrl}, where the controls $(\psi^{\varepsilon},\varphi
^{\varepsilon})\in\mathcal{U}$ are defined in feedback form as
\begin{equation}
\psi^{\varepsilon}(s)\doteq\sum_{j=1}^{|\mathbb{L}|}1_{\{\bar{Y}^{\varepsilon
}(s-)=j\}}u_{j}^{\ast}(s),\;\;\varphi_{ij}^{\varepsilon}(s)\doteq\varphi
_{ij}^{\ast}(s)1_{\{\bar{Y}^{\varepsilon}(s-)=i\}}+1_{\{\bar{Y}^{\varepsilon
}(s-)\neq i\}},\;s\in\lbrack0,T],(i,j)\in\mathbb{T}. \label{eq:contfromlim}%
\end{equation}
From (\ref{VR22}) we get that
\[
-\varepsilon\log\mathbb{E}\left[  \exp\left(  -\varepsilon^{-1}%
F(X^{\varepsilon})\right)  \right]  \leq\mathbb{E}\left[  \tilde{L}_{T}%
(\psi^{\varepsilon})+\bar{L}_{T}(\varphi^{\varepsilon})+F(\bar{X}%
^{\varepsilon})\right]  ,
\]
where $\tilde{L}_{T}$ and $\bar{L}_{T}$ were defined at the beginning of
Section \ref{uppbd}. With $\Delta_{\varepsilon}$ as in \eqref{eq:eq501},
define $Q^{\varepsilon}\in\mathcal{P}_{1}(\mathbb{H}_{T})$ by
\eqref{eqn:defofQ} where $\eta^{\varepsilon}$ is as introduced in
\eqref{eq:eq1347}. From \eqref{eq:eq336_17} and \eqref{eq:eq337_17} (note that
\eqref{eq:eq337_17} gives a lower bound on $\pi_{i}^{\ast}(s)$) it follows
that for all $\varepsilon>0$
\begin{equation}
\tilde{L}_{T}(\psi^{\varepsilon})+\bar{L}_{T}(\varphi^{\varepsilon})\leq
c_{1}(I(\xi)+1)\doteq M<\infty. \label{eq:eq451_17}%
\end{equation}
It then follows from Proposition \ref{ctrlmeas_tight} that $(\bar
{X}^{\varepsilon},Q^{\varepsilon})$ is a tight family of $C([0,T]:\mathbb{R}%
^{d})\times\mathcal{M}_{F}(\mathbb{H}_{T})$-valued random variables, and if
$(\bar{\xi},\bar{Q})$ is a weak limit point of $\{(\bar{X}^{\varepsilon
},Q^{\varepsilon})\}$ then equations \eqref{xi} and \eqref{inv} hold with
$(\xi,Q)$ replaced with $(\bar{\xi},\bar{Q})$. Disintegrate $\bar{Q}$ as
\begin{equation}
\bar{Q}(ds\times\{y\}\times d\eta\times dz)=ds\bar{\pi}_{y}(s)[\bar
{Q}]_{34|12}(d\eta\times dz). \label{eq:eq456_17}%
\end{equation}
We will now show that, a.s.,
\begin{equation}
(\bar{\xi}(s),\bar{\pi}(s))=(\xi^{\ast}(s),\pi^{\ast}(s))\mbox{ for a.e. }s\in
\lbrack0,T]. \label{eq:eq346_17}%
\end{equation}

We can assume without loss of generality that convergence of $(\bar
{X}^{\varepsilon},Q^{\varepsilon})$ to $(\bar{\xi},\bar{Q})$ holds a.s. along
the full sequence. We begin by showing that for every $y\in\mathbb{L}$,
$t\in\lbrack0,T]$, and any continuous map $h:[0,T]\rightarrow\mathbb{R}^{m}$,
with $h(s)^{\prime}$ denoting the transpose,%
\begin{equation}
\int_{\mathbb{H}_{t}}h(s)^{\prime}1_{\{j\}}(y)z\bar{Q}(d\mathbf{v})=\int
_{0}^{t}h(s)^{\prime}u_{j}^{\ast}(s)\bar{\pi}_{j}(s)ds. \label{eq:eq444_17}%
\end{equation}
As in the proof of Proposition \ref{ctrlmeas_tight}, using \eqref{eq:eq451_17}
we have that for all $t\in\lbrack0,T]$ and $j\in\mathbb{L}$
\begin{equation}
\int_{\mathbb{H}_{t}}h(s)^{\prime}1_{\{j\}}(y)zQ^{\varepsilon}(d\mathbf{v}%
)\rightarrow\int_{\mathbb{H}_{t}}h(s)^{\prime}1_{\{j\}}(y)z\bar{Q}%
(d\mathbf{v}). \label{eq:eq447_17}%
\end{equation}
The left side of (\ref{eq:eq447_17}) equals
\[
\int_{\lbrack0,t]}h(s)^{\prime}\frac{1}{\Delta_{\varepsilon}}\int
_{s}^{(s+\Delta_{\varepsilon})\wedge T}1_{\{\bar{Y}^{\varepsilon}(r)=j\}}%
u_{j}^{\ast}(r)dr=\int_{[0,t]}h(s)^{\prime}u_{j}^{\ast}(s)\frac{1}%
{\Delta_{\varepsilon}}\int_{s}^{(s+\Delta_{\varepsilon})\wedge T}1_{\{\bar
{Y}^{\varepsilon}(r)=j\}}dr+\mathcal{R}_{t}^{\varepsilon},
\]
where $\mathcal{R}_{t}^{\varepsilon}$ is simply the difference of the first
term on the right and the term on the left. Denoting the marginals of
$Q^{\varepsilon}$ and $\bar{Q}$ on the first two coordinates by
$[Q^{\varepsilon}]_{12}$ and $[\bar{Q}]_{12}$, the first term on the right
side equals
\[
\int_{\lbrack0,t]\times\mathbb{L}}h(s)^{\prime}u_{j}^{\ast}(s)1_{\{j\}}%
(y)[Q^{\varepsilon}]_{12}(dy\times ds).
\]
From Lemma \ref{lem:lusincgce} and the fact that $u_{j}^{\ast}$ is square
integrable, this converges a.s. to
\[
\int_{\lbrack0,t]\times\mathbb{L}}h(s)^{\prime}u_{j}^{\ast}(s)1_{\{j\}}%
(y)[\bar{Q}]_{12}(dy\times ds)=\int_{[0,t]}h(s)^{\prime}u_{j}^{\ast}%
(s)\bar{\pi}_{j}(s)ds,
\]
where the equality follows from \eqref{eq:eq456_17}. Thus in order to prove
\eqref{eq:eq444_17}, in view of \eqref{eq:eq447_17}, it suffices to show that
$\mathcal{R}_{t}^{\varepsilon}\rightarrow0$ in probability as $\varepsilon
\rightarrow0$ for all $t\in\lbrack0,T]$. By a similar calculation as in
\eqref{eq:eq259_17}
\begin{align*}
|\mathcal{R}_{t}^{\varepsilon}|  &  =\left\vert \int_{[0,t]}h(s)^{\prime}%
\frac{1}{\Delta_{\varepsilon}}\int_{s}^{(s+\Delta_{\varepsilon})\wedge
T}1_{\{\bar{Y}^{\varepsilon}(r)=j\}}(u_{j}^{\ast}(r)-u_{j}^{\ast
}(s))drds\right\vert \\
&  \leq\tilde{T}_{t}^{(1)}+\tilde{T}_{t}^{(2)}+\tilde{T}_{t}^{(3)}%
\end{align*}
where, as in \eqref{eq:eq529} and \eqref{eq:eq530},
\[
(\tilde{T}_{t}^{(1)})^{2}+(\tilde{T}_{t}^{(3)})^{2}\leq8\Vert h\Vert_{\infty
}^{2}\Delta_{\varepsilon}\int_{0}^{T}\Vert u_{j}^{\ast}(s)\Vert^{2}ds
\]
and
\begin{equation}
\tilde{T}_{t}^{(2)}=\int_{[\Delta_{\varepsilon},t]}\left(  1_{\{\bar
{Y}^{\varepsilon}(r)=j\}}\frac{1}{\Delta_{\varepsilon}}\int_{r-\Delta
_{\varepsilon}}^{r}h(s)^{\prime}(u_{j}^{\ast}(r)-u_{j}^{\ast}(s))ds\right)
dr. \label{eq:eq656_17}%
\end{equation}
We will use the following fact several times: If $f:[0,T]\rightarrow
\mathbb{R}$ is integrable and $g:[0,T]\rightarrow\mathbb{R}$ is bounded, then
\begin{equation}
\int_{\lbrack\Delta_{\varepsilon},t]}\left(  1_{\{\bar{Y}^{\varepsilon
}(r)=j\}}\frac{1}{\Delta_{\varepsilon}}\int_{r-\Delta_{\varepsilon}}%
^{r}g(s)(f(r)-f(s))ds\right)  dr\rightarrow0,\mbox{ as }\varepsilon
\rightarrow0. \label{eq:eqgenffact}%
\end{equation}
This is a consequence of the generalized dominated convergence theorem, which
says that if $\psi_{\varepsilon},\psi,h_{\varepsilon},h:[0,t]\rightarrow
\mathbb{R}$ are integrable maps such that $\psi_{\varepsilon}(s)\rightarrow
\psi(s)$, $h_{\varepsilon}(s)\rightarrow h(s)$ for a.e. $s\in(0,t)$;
$|\psi_{\varepsilon}|\leq h_{\varepsilon}$ a.e.; and
\[
\int_{\lbrack0,t]}h_{\varepsilon}(s)ds\rightarrow\int_{\lbrack0,t]}%
h(s)ds\mbox{ as }\varepsilon\rightarrow0,
\]
then
\[
\int_{\lbrack0,t]}\psi_{\varepsilon}(s)ds\rightarrow\int_{\lbrack0,t]}%
\psi(s)ds\mbox{ as }\varepsilon\rightarrow0.
\]
Indeed, the term in the parentheses in \eqref{eq:eqgenffact} converges to $0$
for a.e. $r\in(0,t)$ and is bounded above in absolute value by
\[
\Vert g\Vert_{\infty}\left(  |f(r)|+\frac{1}{\Delta_{\varepsilon}}%
\int_{r-\Delta_{\varepsilon}}^{r}|f(s)|ds\right)  ,
\]
which converges for a.e. $r$ to $2\Vert g\Vert_{\infty}|f(r)|$. Also as
$\varepsilon\rightarrow0$
\[
\int_{\lbrack\Delta_{\varepsilon},t]}\frac{1}{\Delta_{\varepsilon}}%
\int_{r-\Delta_{\varepsilon}}^{r}|f(s)|dsdr\rightarrow\int_{0}^{t}|f(r)|dr.
\]
This proves \eqref{eq:eqgenffact}. Using this fact in \eqref{eq:eq656_17} we
now have that $\tilde{T}_{t}^{(2)}\rightarrow0$ as $\varepsilon\rightarrow0$.
Thus $\mathcal{R}_{t}^{\varepsilon}\rightarrow0$ and hence \eqref{eq:eq444_17} follows.

Applying \eqref{eq:eq444_17} to the rows of $a_{j}(\bar{\xi}(s))$ and summing
over $j$, we have that
\[
\int_{\mathbb{H}_{t}}a(\bar{\xi}(s),y)z\bar{Q}(d\mathbf{v})=\sum
_{j=1}^{|\mathbb{L}|}\int_{[0,t]}a_{j}(\bar{\xi}(s))u_{j}^{\ast}(s)\bar{\pi
}_{j}(s)ds.
\]
Also, recalling the representation of $\bar{Q}$ in \eqref{eq:eq456_17}
\[
\int_{\mathbb{H}_{t}}b(\bar{\xi}(s),y)\bar{Q}(d\mathbf{v})=\sum_{j=1}%
^{|\mathbb{L}|}\int_{[0,t]}b_{j}(\bar{\xi}(s))\bar{\pi}_{j}(s)ds.
\]
Thus we have shown that \eqref{eq:stateq} is satisfied with $(\xi,\pi,u)$
replaced with $(\bar{\xi},\bar{\pi},u^{\ast})$.

We next show that \eqref{eq:eqinvar} is satisfied as well, i.e.,
\begin{equation}
\sum_{i\in\mathbb{L}}\bar{\pi}_{i}(s)A_{ij}^{\varphi_{i,\cdot}^{\ast}%
(s,\cdot)}(\bar{\xi}(s))=0\mbox{ for a.e. }s\in\lbrack0,T]\mbox{ and all }j\in
\mathbb{L}. \label{Eq:eq458_17}%
\end{equation}
For this recall that \eqref{inv} is satisfied with $(\xi,Q)$ replaced with
$(\bar{\xi},\bar{Q})$, i.e.,
\begin{equation}
\int_{\mathbb{H}_{t}}A_{y,j}^{\eta}(\bar{\xi}(s))\bar{Q}(d\mathbf{v})=0.
\label{eq:eq911_17}%
\end{equation}
Also, from Lemma \ref{lem:genQ} (see, e.g., the proof of \eqref{E.1}), for any
continuous function $g:[0,T]\rightarrow\mathbb{R}^{d}$,
\begin{equation}
\lim_{\varepsilon\rightarrow0}\int_{\mathbb{H}_{t}}\eta_{j}(E_{yj}%
(g(s)))Q^{\varepsilon}(d\mathbf{v})=\int_{\mathbb{H}_{t}}A_{y,j}^{\eta
}(g(s))\bar{Q}(d\mathbf{v}), \label{eq:eq831_17}%
\end{equation}
and a similar argument as in \eqref{eq:ab1202} shows that, for every
$t\in\lbrack0,T]$,
\begin{equation}
\left\vert \int_{\mathbb{H}_{t}}\eta_{j}(E_{yj}(g(s)))Q^{\varepsilon
}(d\mathbf{v})-\int_{0}^{t}1_{\{\bar{Y}^{\varepsilon}(s)=i\}}\int
_{E_{ij}(g(s))}\varphi_{ij}^{\ast}(s,z)\lambda_{\zeta}(dz)ds\right\vert
\rightarrow0 \label{eq:eq831_17b}%
\end{equation}
as $\varepsilon\rightarrow0$. Next, as in the proof of the convergence of
$\mathcal{R}_{t}^{\varepsilon}\rightarrow0$, it follows that for every
$(i,j)\in\mathbb{T}$
\begin{equation}
\tilde{\mathcal{R}}_{t}^{\varepsilon}\doteq\int_{0}^{t}\beta_{ij}(s)\left(
\frac{1}{\Delta_{\varepsilon}}\int_{s}^{(s+\Delta_{\varepsilon})\wedge
T}1_{\{\bar{Y}^{\varepsilon}(r)=i\}}dr-1_{\{\bar{Y}^{\varepsilon}%
(s)=i\}}\right)  ds\rightarrow0 \label{eq:eq828_17}%
\end{equation}
in probability as $\varepsilon\rightarrow0$, where
\[
\beta_{ij}(s)\doteq\int_{E_{ij}(g(s))}\varphi_{ij}^{\ast}(s,z)\lambda_{\zeta
}(dz).
\]
Indeed the terms analogous to $\tilde{T}_{t}^{(1)}$ and $\tilde{T}_{t}^{(3)}$
are shown to converge to $0$ exactly as before using the bound
\eqref{eq:eq451_17}. The term analogous to $\tilde{T}_{t}^{(2)}$ given by
\[
\bar{T}_{t}^{(2)}\doteq\int_{\lbrack\Delta_{\varepsilon},t]}\left(
1_{\{\bar{Y}^{\varepsilon}(r)=i\}}\frac{1}{\Delta_{\varepsilon}}\int
_{r-\Delta_{\varepsilon}}^{r}(\beta_{ij}(r)-\beta_{ij}(s))ds\right)  dr
\]
can be shown to converges to $0$ by recalling the fact in
\eqref{eq:eqgenffact} and that $\int_{[0,T]}\beta_{ij}(s)ds<\infty$.

From the convergence in \eqref{eq:eq828_17} and the equalities in
\eqref{eq:eq831_17} and \eqref{eq:eq831_17b} it now follows that
\[
\sum_{i\in\mathbb{L}}\int_{[0,t]}\frac{1}{\Delta_{\varepsilon}}\int
_{s}^{(s+\Delta_{\varepsilon})\wedge T}1_{\{\bar{Y}^{\varepsilon}(r)=i\}}%
\int_{E_{ij}(g(s))}\varphi_{ij}^{\ast}(s,z)\lambda_{\zeta}(dz)drds\rightarrow
\int_{\mathbb{H}_{t}}A_{y,j}^{\eta}(g(s))\bar{Q}(d\mathbf{v})
\]
as $\varepsilon\rightarrow0$. However, the expression on the left side in the
previous display is the same as
\[
\int_{\lbrack0,t]\times\mathbb{L}}\left(  \int_{E_{yj}(g(s))}\varphi
_{yj}^{\ast}(s,z)\lambda_{\zeta}(dz)\right)  [Q^{\varepsilon}]_{12}(dy\times
ds).
\]
From \eqref{eq:eq451_17}
\[
\int_{\lbrack0,T]}\left(  \int_{[0,\zeta]}\varphi_{yj}^{\ast}(s,z)\lambda
_{\zeta}(dz)\right)  ds<\infty.
\]
Thus from Lemma \ref{lem:lusincgce}
\begin{align*}
&  \int_{[0,t]\times\mathbb{L}}\left(  \int_{E_{yj}(g(s))}\varphi_{yj}^{\ast
}(s,z)\lambda_{\zeta}(dz)\right)  [Q^{\varepsilon}]_{12}(dy\times ds)\\
&  \quad\rightarrow\int_{\lbrack0,t]\times\mathbb{L}}\left(  \int
_{E_{yj}(g(s))}\varphi_{yj}^{\ast}(s,z)\lambda_{\zeta}(dz)\right)  [\bar
{Q}]_{12}(dy\times ds)\\
&  \quad=\int_{[0,t]}\sum_{i\in\mathbb{L}}\bar{\pi}_{i}(s)A_{ij}%
^{\varphi_{i,\cdot}^{\ast}(s,\cdot)}(g(s))ds
\end{align*}
for all $t\in\lbrack0,T]$. Thus we have shown that for every continuous
$g:[0,T]\rightarrow\mathbb{R}^{d}$ and every $t\in\lbrack0,T]$,
\[
\int_{\mathbb{H}_{t}}A_{y,j}^{\eta}(g(s))\bar{Q}(d\mathbf{v})=\int_{[0,t]}%
\sum_{i\in\mathbb{L}}\bar{\pi}_{i}(s)A_{ij}^{\varphi_{i,\cdot}^{\ast}%
(s,\cdot)}(g(s))ds.
\]
Taking $g=\bar{\xi}$ and using \eqref{eq:eq911_17} we have \eqref{Eq:eq458_17}.

We have therefore shown that both \eqref{eq:stateq} and \eqref{eq:eqinvar}
hold with $(\xi, \pi, u, \varphi)$ replaced by $(\bar\xi, \bar\pi, u^{*},
\varphi^{*})$. Thus from part 4 of Proposition \ref{prop:prop4.1}, we have \eqref{eq:eq346_17}.

Finally we consider the convergence of costs. Note that
\[
\tilde{L}_{T}(\psi^{\varepsilon})=\frac{1}{2}\int_{0}^{T}\Vert\psi
^{\varepsilon}(s)\Vert^{2}ds=\frac{1}{2}\sum_{j\in\mathbb{L}}\int_{0}%
^{T}1_{\{\bar{Y}^{\varepsilon}(s)=j\}}\Vert u_{j}^{\ast}(s)\Vert^{2}ds.
\]
Also, for each $j$,
\[
\int_{0}^{T}1_{\{\bar{Y}^{\varepsilon}(s)=j\}}\Vert u_{j}^{\ast}(s)\Vert
^{2}ds=\int_{0}^{T}\Vert u_{j}^{\ast}(s)\Vert^{2}\frac{1}{\Delta_{\varepsilon
}}\int_{s}^{(s+\Delta_{\varepsilon})\wedge T}1_{\{\bar{Y}^{\varepsilon
}(r)=j\}}drds+\mathcal{R}^{\varepsilon,1},
\]
where $\mathcal{R}^{\varepsilon,1}$ is defined as the difference of the first
expression on the right and the expression on the left. The first expression
on the right side equals
\[
\int_{\lbrack0,T]\times\mathbb{L}}\Vert u_{j}^{\ast}(s)\Vert^{2}%
1_{\{j\}}(y)[Q^{\varepsilon}]_{12}(dy\times ds)
\]
which from Lemma \ref{lem:lusincgce} and the finiteness of $\int_{[0,T]}\Vert
u_{j}^{\ast}(s)\Vert^{2}ds$ converges to
\[
\int_{\lbrack0,T]}\Vert u_{j}^{\ast}(s)\Vert^{2}\pi_{j}^{\ast}(s)ds.
\]
Also as before, in order to show $\mathcal{R}^{\varepsilon,1}\rightarrow0$ we
need to show the convergence to $0$ of
\[
\int_{\lbrack\Delta_{\varepsilon},t]}\left(  1_{\{\bar{Y}^{\varepsilon
}(r)=i\}}\frac{1}{\Delta_{\varepsilon}}\int_{r-\Delta_{\varepsilon}}^{r}(\Vert
u_{j}^{\ast}(r)\Vert^{2}-\Vert u_{j}^{\ast}(s)\Vert^{2})ds\right)  dr,
\]
which follows from \eqref{eq:eqgenffact} and the fact that $\int_{[0,T]}\Vert
u_{j}^{\ast}(s)\Vert^{2}ds<\infty$. Thus we have shown that, as $\varepsilon
\rightarrow0$,
\[
\tilde{L}_{T}(\psi^{\varepsilon})\rightarrow\frac{1}{2}\sum_{j\in\mathbb{L}%
}\int_{0}^{T}\pi_{j}^{\ast}(s)\Vert u_{j}^{\ast}(s)\Vert^{2}ds.
\]
The second term in the cost can be handled in a similar manner. Note that
\[
\bar{L}_{T}(\varphi^{\varepsilon})=\sum_{(i,j)\in\mathbb{T}}\int
_{[0,T]\times\lbrack0,\zeta]}\ell(\varphi_{ij}^{\varepsilon}(s,z))\lambda
_{\zeta}(dz)ds=\sum_{(i,j)\in\mathbb{T}}\int_{[0,T]\times\lbrack0,\zeta
]}1_{\{\bar{Y}^{\varepsilon}(s)=i\}}\ell(\varphi_{ij}^{\ast}(s,z))\lambda
_{\zeta}(dz)ds.
\]
For $(i,j)\in\mathbb{T}$ write
\[
\int_{\lbrack0,T]\times\lbrack0,\zeta]}1_{\{\bar{Y}^{\varepsilon}(s)=i\}}%
\ell(\varphi_{ij}^{\ast}(s,z))\lambda_{\zeta}(dz)ds
\]
as
\[
\int_{\lbrack0,T]}\left(  \int_{[0,\zeta]}\ell(\varphi_{ij}^{\ast
}(s,z))\lambda_{\zeta}(dz)\right)  \frac{1}{\Delta_{\varepsilon}}\int
_{s}^{(s+\Delta_{\varepsilon})\wedge T}1_{\{\bar{Y}^{\varepsilon}%
(r)=i\}}drds+\tilde{\mathcal{R}}^{\varepsilon,1}.
\]
Then using the fact that $\int_{[0,T]\times\lbrack0,\zeta]}\ell(\varphi
_{ij}^{\ast}(s,z))\lambda_{\zeta}(dz)ds<\infty$ we see as before that
$\tilde{\mathcal{R}}^{\varepsilon,1}\rightarrow0$. Also, the first term in the
last display, once again using Lemma \ref{lem:lusincgce} and the integrability
noted in the last sentence, converges to
\[
\int_{\lbrack0,T]\times\lbrack0,\zeta]}\ell(\varphi_{ij}^{\ast}(s,z))\pi
_{i}^{\ast}(s)\lambda_{\zeta}(dz)ds.
\]
Combining the above observations we have, as $\varepsilon\rightarrow0$,
\[
\bar{L}_{T}(\varphi^{\varepsilon})\rightarrow\sum_{(i,j)\in\mathbb{T}}%
\int_{[0,T]\times\lbrack0,\zeta]}\ell(\varphi_{ij}^{\ast}(s,z))\pi_{i}^{\ast
}(s)\lambda_{\zeta}(dz)ds
\]
in probability. Using \eqref{eq:eq451_17} and the dominated convergence
theorem gives
\begin{align*}
&  \mathbb{E}\left[  \tilde{L}_{T}(\psi^{\varepsilon})+\bar{L}_{T}%
(\varphi^{\varepsilon})\right] \\
&  \rightarrow\sum_{i}\frac{1}{2}\int_{0}^{T}\Vert u_{i}^{\ast}(s)\Vert^{2}%
\pi_{i}^{\ast}(s)ds+\sum_{(i,j)\in\mathbb{T}}\int_{[0,\zeta]\times\lbrack
0,T]}\ell(\varphi_{ij}^{\ast}(s,z))\pi_{i}^{\ast}(s)\lambda_{\zeta}(dz)ds\\
&  \leq I(\xi)+\gamma.
\end{align*}
Using this and the convergence of $\bar{X}^{\varepsilon}$ to $\xi^{\ast}$ (see
\eqref{eq:eq346_17}), we have from \eqref{eq:eqnearopt} that
\begin{align*}
\limsup_{\varepsilon\rightarrow0}-\varepsilon\log\mathbb{E}\left[  \exp\left(
-\varepsilon^{-1}F(X^{\varepsilon})\right)  \right]   &  \leq F(\xi^{\ast
})+I(\xi)+\gamma\\
&  \leq F(\xi)+I(\xi)+\gamma+C_{F}\gamma\\
&  \leq\inf_{\zeta\in C([0,T]:\mathbb{R}^{d})}[I(\zeta)+F(\zeta)]+2\gamma
+C_{F}\gamma,
\end{align*}
where the second inequality uses \eqref{eq:eq811_17} in Proposition
\ref{prop:prop4.1}. Sending $\gamma\rightarrow0$ gives the desired lower bound
in \eqref{eq:lowbdnew}.

\vspace{\baselineskip}

\textsc{\noindent A. Budhiraja \newline Department of Statistics and
Operations Research\newline University of North Carolina\newline Chapel Hill,
NC 27599, USA\newline email: budhiraj@email.unc.edu \vspace{\baselineskip} }

\textsc{\noindent P. Dupuis\newline Division of Applied Mathematics\newline
Brown University\newline Providence, RI 02912, USA\newline email:
Paul\_Dupuis@brown.edu \vspace{\baselineskip} }

\textsc{\noindent A. Ganguly\newline Department of Mathematics\newline
Louisiana State University\newline Baton Rouge, LA 70803, USA\newline email:
aganguly@lsu.edu }
\end{document}